\documentclass{amsart}

\usepackage{paralist}
\usepackage{color}
\usepackage{mathrsfs}
\usepackage{amssymb}
\usepackage{amsmath}
\usepackage{amsfonts}
\usepackage{stmaryrd}
\usepackage[shortlabels]{enumitem}

\newtheorem*{thma}{Theorem~A}
\newtheorem*{thmb}{Theorem~B}

\newtheorem{theorem}{Theorem}[section]
\newtheorem{lemma}[theorem]{Lemma}
\newtheorem{proposition}[theorem]{Proposition}
\newtheorem{corollary}[theorem]{Corollary}
\newtheorem{claim}[theorem]{Claim}
\newtheorem{subclaim}[theorem]{Subclaim}
\newtheorem{fact}[theorem]{Fact}
\newtheorem{question}[theorem]{Question}

\theoremstyle{definition}
\newtheorem{definition}[theorem]{Definition}
\newtheorem{remark}[theorem]{Remark}

\usepackage[pdftex]{hyperref}
\hypersetup{
    colorlinks=true, 
    linktoc=all,     
    linkcolor=blue,  
    allcolors=blue,
}

\usepackage[framemethod=tikz]{mdframed}

\newcommand{\cf}{\mathrm{cf}}
\newcommand{\dom}{\mathrm{dom}}
\newcommand{\bb}{\mathbb}

\newcommand{\nacc}{\mathrm{nacc}}
\newcommand{\acc}{\mathrm{acc}}
\newcommand{\pred}{\mathrm{pred}}
\newcommand{\height}{\mathrm{ht}}

\newcommand{\mc}{\mathcal}
\newcommand{\power}{\ensuremath{\mathscr{P}}}
\newcommand{\sub}{\subseteq}
\newcommand{\ra}{\rightarrow}

\newcommand{\Add}{\mathrm{Add}}
\newcommand{\Coll}{\mathrm{Coll}}

\newcommand{\Q}{\bb{Q}}
\renewcommand{\P}{\bb{P}}

\newcommand{\M}{\bb{M}}
\newcommand{\K}{\bb{K}}
\renewcommand{\S}{\bb{S}}

\newcommand{\TP}{{\sf TP}}
\newcommand{\ITP}{{\sf ITP}}
\newcommand{\SP}{{\sf SP}}
\newcommand{\ISP}{{\sf ISP}}

\newcommand{\ZFC}{\sf ZFC}

\newcommand{\CH}{\sf CH}
\newcommand{\GCH}{\sf GCH}

\newcommand{\PFA}{\sf PFA}

\newcommand{\wKH}{\sf wKH}
\newcommand{\KH}{\sf KH}

\newcommand{\GMP}{\mathsf{GMP}}
\newcommand{\aut}{\mathrm{Aut}}
\renewcommand{\top}{\mathsf{top}}

\newcommand{\set}[2]{\ensuremath{\{#1 \,|\, #2 \}}}
\newcommand{\seq}[2]{\ensuremath{\langle #1 \,|\, #2 \rangle}}
\newcommand{\rest}[0]{\restriction}
\newcommand{\Ult}{\mathrm{Ult}}

\title[Almost Kurepa Suslin trees and destructibility of $\GMP$]{Almost Kurepa Suslin trees and destructibility of the Guessing Model Property}
\author{Chris Lambie-Hanson}
\address[Lambie-Hanson]{
Institute of Mathematics, 
Czech Academy of Sciences, 
{\v Z}itn{\'a} 25, Prague 1, 
115 67, Czech Republic
}
\email{lambiehanson@math.cas.cz}
\urladdr{https://clambiehanson.github.io}

\author{{\v S}{\'a}rka Stejskalov{\'a}}
\address[Stejskalov{\'a}]{Charles University, Department of Logic,
Celetn{\' a} 20, Prague~1, 
116 42, Czech Republic
}
\email{sarka.stejskalova@ff.cuni.cz}
\urladdr{logika.ff.cuni.cz/sarka}

\address{Institute of Mathematics, Czech Academy of Sciences, {\v Z}itn{\'a} 25, Prague 1, 115 67, Czech Republic}


\thanks{Both authors were supported by the Czech Academy of Sciences 
(RVO 67985840) and the GA\v{C}R project 23-04683S}

\begin{document}

\begin{abstract}
  Building on recent work of Krueger and the second author, we prove the consistency of 
  the Guessing Model Principle at $\omega_2$ together with the existence of an almost 
  Kurepa Suslin tree. In particular, it is consistent that the Guessing Model Principle 
  holds but is destructible by a ccc forcing of size $\omega_1$. We also prove the 
  consistency of the existence of a weak Kurepa tree together with the failure of 
  the Kurepa Hypothesis and a certain guessing model principle that, 
  for example, implies the tree property at $\omega_2$.
\end{abstract}

\maketitle
\tableofcontents

\section{Introduction}

Compactness principles play a central role in contemporary combinatorial set theory, 
particularly around questions concerning large cardinals, forcing, and 
canonical inner models. A paradigmatic example of a compactness principle at a 
cardinal $\kappa$ is the \emph{tree property} at $\kappa$, denoted $\TP(\kappa)$, 
which asserts the non-existence of $\kappa$-Aronszajn trees.
Among inaccessible cardinals, compactness principles are frequently used to 
characterize large cardinals. For example, an inaccessible cardinal $\kappa$ 
is weakly compact if and only if $\TP(\kappa)$ holds. When they hold at large 
cardinals, compactness principles are typically quite robust under forcing 
extensions. For example, if $\kappa$ is weakly compact and $\P$ is any forcing notion 
that has the $\mu$-cc for some $\mu < \kappa$, then $\TP(\kappa)$ continues to 
hold in the forcing extension by $\P$.

However, compactness principles used to characterize large cardinals can also 
consistently hold at accessible cardinals, such as $\omega_2$, and here questions 
of their robustness are much less clear. For example, by results of Mitchell and 
Silver (cf.\ \cite{MIT:tree}), $\TP(\omega_2)$ is equiconsistent with the 
existence of a weakly compact cardinal. Very few definitive answers are known, 
however, about the (consistent) (in)destructibility of $\TP(\omega_2)$ under 
ccc forcing extensions. For example, it is unknown whether it is consistent 
that $\TP(\omega_2)$ holds and is preserved in all ccc forcing extensions, 
but it is also unknown whether it is consistent that $\TP(\omega_2)$ holds 
and there exists a ccc forcing poset $\P$ forcing the failure of $\TP(\omega_2)$. 
For the best partial results on these questions, see \cite{UNGER:1} and 
\cite{hs_indestructibility}.

In this paper, we focus on an important strengthening of the tree property known 
as the \emph{guessing model property} (denoted $\GMP(\kappa)$). Just as $\TP(\kappa)$
characterizes weakly compact cardinals among inaccessible cardinals, $\GMP(\kappa)$ 
characterizes supercompact cardinals among inaccessible cardinals. Also, just as $\TP(\kappa)$ at 
a weakly compact cardinal is robust under forcings with good chain condition, 
$\GMP(\kappa)$ at a supercompact cardinal $\kappa$ continues to hold after 
forcing with any forcing poset that is $\mu$-cc for some $\mu < \kappa$. Here 
we show that this robustness can consistently fail at $\omega_2$; in particular,
letting the simple $\GMP$ denote $\GMP(\omega_2)$, we show that
it is consistent that $\GMP$ holds but there exists a ccc forcing of size 
$\omega_1$ forcing the failure of $\GMP$.

We achieve this task by showing that $\GMP$ is compatible with the existence of 
a particular combinatorial object of interest in its own right: an almost Kurepa 
Suslin tree, i.e., a Suslin tree $T$ such that after forcing with $T$, it becomes
a Kurepa tree (see Section \ref{sect: prelim} for a precise definition). Since $\GMP$ implies that there are no Kurepa trees, it cannot hold after forcing with $T$.

Almost Kurepa Suslin trees were first constructed by Jensen; see \cite{DJ_Souslin} for a sketch of the proof that $\Diamond^+$ implies their existence. Later, Bilaniuk \cite{Bil_thesis} constructed an almost Kurepa Suslin tree from the assumption that $\Diamond$ holds and there exists a Kurepa tree. In both of these constructions, the resulting Suslin tree is almost Kurepa because there are at least $\omega_2$-many automorphisms of the tree which are pairwise almost disjoint. By applying these automorphisms to a generically added cofinal branch, this consequently ensures that the tree will have at least $\omega_2$-many cofinal branches in the generic extension by itself. In the proof of our theorem, the property of being almost Kurepa is also ensured by the existence of $\omega_2$-many pairwise almost disjoint automorphisms of the Suslin tree.

\begin{thma}
  Suppose that there exists a supercompact cardinal. Then there exists a forcing extension in which
  \begin{enumerate}
    \item there exists an almost Kurepa Suslin tree;
    \item $\GMP$ holds.
  \end{enumerate}
  In particular, it is consistent that $\GMP$ holds and yet is destructible by a ccc forcing 
  of cardinality $\omega_1$.
\end{thma}

$\GMP$ implies the failure of the weak Kurepa hypothesis (i.e., there are no weak Kurepa trees, denoted $\neg\wKH$); hence, Theorem A provides a model where there are no weak Kurepa trees and there is an almost Kurepa Suslin tree. However, to achieve $\neg\wKH$, the assumption of an inaccessible cardinal is sufficient. The following is an immediate corollary of the proof of Theorem A:

\begin{corollary} \label{cor: cor_11}
  Suppose that there exists an inaccessible cardinal. Then there exists a forcing extension in which
  \begin{enumerate}
    \item there exists an almost Kurepa Suslin tree;
    \item $\neg\wKH$ holds.
  \end{enumerate}
  In particular, it is consistent that $\neg\wKH$ holds and yet is destructible by a ccc forcing 
  of cardinality $\omega_1$.
\end{corollary}

Let us mention several results from the literature to put Theorem A in context.
\begin{itemize}
  \item In \cite{HLHS}, Honzik and the authors prove that $\GMP$ is always indestructible under 
  adding any number of Cohen reals. This shows that Theorem A is sharp in the sense that a countable 
  forcing poset can never destroy $\GMP$.
  \item Relative to the consistency of a supercompact cardinal, it is consistent that $\GMP$ holds 
  and is indestructible under all ccc forcings of cardinality $\omega_1$. In particular, this holds 
  in the extension by the classical Mitchell forcing $\M(\omega, \kappa)$, where $\kappa$ is 
  supercompact in the ground model. For an argument of this fact, see \cite[Theorem 2]{UNGER:1}; 
  the theorem proven there is about the tree property rather than $\GMP$, and thus only requires a 
  weakly compact cardinal rather than a supercompact cardinal, but the proof of the stronger fact 
  is essentially the same. This fact is additionally relevant to our main result due to our method of 
  proof: our forcing extension yielding Theorem A utilizes a variation on the classical Mitchell 
  forcing in which the forcing iterands adding Cohen subsets to $\omega_1$ are replaced by iterands 
  adding automorphisms to a distinguished free Suslin tree. This provides an illustration of the 
  versatility of Mitchell-type forcings, and we expect that a number of further consistency results 
  can be established by appropriately varying the iterands in Mitchell-type constructions.

  \item In \cite{ks}, Krueger and the second author show that, assuming the consistency of an 
  inaccessible cardinal, it is consistent with $\mathsf{CH}$ that the Kurepa Hypothesis fails and 
  yet there exists an almost Kurepa Suslin tree, in the process answering long-open questions of 
  Jin and Shelah \cite{jin_shelah} and of Moore \cite{moore_structural_analysis}. Since 
  $\GMP$ entails the failure of the weak Kurepa Hypothesis, and hence of the Kurepa Hypothesis, 
  Theorem A and Corollary \ref{cor: cor_11} can be seen as a strengthening of this result, at the 
  cost of necessarily forgoing $\mathsf{CH}$, which is incompatible with $\neg\wKH$. Our proof of 
  Theorem A is in large part an adaptation of the techniques developed in \cite{ks} to the setting of 
  Mitchell-type forcing extensions.

\item The notion of a Suslin tree being almost Kurepa is closely connected to the non-saturation of Aronszajn trees. By an observation of Moore, if $S$ is an almost Kurepa Suslin tree, then $S \otimes S$ is a non-saturated Aronszajn tree; for details, see \cite{ks}. Therefore, in the extension from Theorem~A, $\GMP$ holds and there is a non-saturated Aronszajn tree. In \cite{ks2}, Krueger and the second author establish a related result: under the assumption of the existence of a supercompact cardinal, there is a forcing $\P$ such that in the generic extension by $\P$, $\GMP$ holds and there exists a strongly non-saturated Aronszajn tree. Note that in their model, the Aronszajn tree is \emph{strongly} non-saturated\footnote{For the definition, see \cite{ks2}.}, but the forcing $\P$ is designed to add an Aronszajn tree which is strongly non-saturated rather than an almost Kurepa Suslin tree, and it is unclear whether an almost Kurepa Suslin tree exists in their model. On the other hand, the product of the almost Kurepa Suslin tree from Theorem~A with itself is only non-saturated and not strongly non-saturated.
\end{itemize}

In the second main theorem of the paper, we separate the weak Kurepa Hypothesis from the Kurepa 
Hypothesis in the presence of a weakening of $\GMP$ (which is still strong enough to 
imply $\TP(\omega_2)$), refining a previous result of the authors from \cite{kurepa_paper}:

\begin{thmb}
  Suppose that there is a supercompact cardinal $\kappa$. Then there is a forcing extension 
  in which 
  \begin{enumerate}
    \item $2^{\omega}=\omega_2=\kappa$;
    \item $\GMP^{\omega_2}(\omega_2)$ holds;
    \item there are no Kurepa tree, but there is a weak Kurepa tree.
  \end{enumerate}
\end{thmb}
In contrast to Theorem A, the proof of this theorem makes use of the classical Mitchell forcing, 
together with a natural forcing to add a weak Kurepa tree.

The structure of the remainder of the paper is as follows. In Section 
\ref{sect: prelim}, we present some relevant background information concerning 
trees, guessing models, Mitchell forcing, and automorphisms. In Section 
\ref{sect: thma_sect}, we prove Theorem A and Corollary \ref{cor: cor_11}, 
and in Section \ref{sect: thmb_sect}, 
we prove Theorem B. In Section \ref{sect: kurepa_sect}, we fill some gaps in 
the literature by surveying known results surrounding the consistent 
(in)destructibility of failures of the (weak) Kurepa Hypothesis. We end the 
section by providing a direct proof that the failure of the weak Kurepa 
Hypothesis is always preserved by $\sigma$-centered forcings. In Section 
\ref{sect: questions}, we present a few remaining open questions. Finally, 
Appendix \ref{appendix_a} provides an example, promised in Remark 
\ref{remark: term_forcing} below, of a two step iteration $\Add(\omega,1) 
\ast \dot{\Q}$ such that $\dot{\Q}$ is forced to be totally proper but for 
which forcing with the associated term forcing over $V$ collapses $\omega_1$.

\subsection{Notation and conventions}

We let $\mathrm{On}$ denote the class of all ordinals.
Given $\delta \in \mathrm{On}$, let $\Sigma(\delta)$ denote the set of 
successor ordinals less than $\delta$.

If $\vec{x}$ is a sequence of length $\beta$, then by convention, 
unless explicitly specified otherwise, for all $\alpha < \beta$, 
the $\alpha^{\mathrm{th}}$ element of $\vec{x}$ will be denoted by $x_\alpha$.

If $\kappa$ and $\lambda$ are cardinals, with $\kappa$ regular and infinite, 
then $\mathrm{Add}(\kappa, \lambda)$ is the forcing to add $\lambda$-many 
Cohen subsets to $\kappa$. Concretely, conditions in $\mathrm{Add}(\kappa, \lambda)$ 
are partial functions of cardinality less than $\kappa$ from $\lambda$ to 
${^{<\kappa}}2$. If $p,q \in \mathrm{Add}(\kappa,\lambda)$, then $q \leq p$ iff 
$\dom(q) \supseteq \dom(p)$ and $q(\gamma) \supseteq p(\gamma)$ for all 
$\gamma \in \dom(p)$.

\section{Preliminaries} \label{sect: prelim}

In this section, we review some background information on trees, guessing models, 
Mitchell forcing, and tree automorphisms.

\subsection{Trees and the guessing model property}

\begin{definition}
  A \emph{tree} is a partial order $(T, <_T)$ such that, for all $t \in T$, 
  the set $\pred_T(t) = \{s \in T \mid s <_T t\}$ is well-ordered by $<_T$. 
  If $T$ is a tree and $t \in T$, then we let $\height_T(t)$ denote the 
  order type of $(\pred_T(t), <_T)$. For all $\alpha \in \mathrm{On}$, 
  we let $T_\alpha$ denote $\{t \in T \mid \height_T(t) = \alpha\}$. 
  Expressions such as $T_{<\alpha}$ or $T_{\leq \alpha}$ are defined in the 
  obvious way. We let $\height(T) = \min\{\alpha \in \mathrm{On} \mid 
  T_\alpha = \emptyset\}$. $\height(T)$ is often referred to as the \emph{height} 
  of $T$. 
  
  Given a regular uncountable cardinal $\kappa$, a tree $T$ is called a 
  \emph{$\kappa$-tree} if $\height(T) = \kappa$ and $|T_\alpha| < \kappa$ for all 
  $\alpha < \kappa$. A $\kappa$-tree $T$ is \emph{normal} if it satisfies the 
  following two conditions:
  \begin{itemize}
    \item for all $\alpha < \beta < \kappa$ and all $t \in T_\alpha$, there 
    exists $s \in T_\beta$ such that $t <_T s$;
    \item for every limit ordinal $\alpha < \kappa$ and all $s,t \in T_\alpha$, 
    if $\pred_T(s) = \pred_T(t)$, then $s = t$.
  \end{itemize}
  If $T$ is a tree and $t \in T$, then we let $\mathrm{succ}_T(t)$ denote the 
  set of immediate successors of $t$ in $T$, i.e., the set 
  $\{s \in T \mid t <_T s \text{ and } \height_T(s) = \height_T(t) + 1\}$. 
  We say that $T$ is \emph{infinitely splitting} if 
  $\mathrm{succ}_T(t)$ is infinite for all $t \in T$.
  
  A \emph{branch} through a tree $T$ is a subset $b \subseteq T$ that is 
  $<_T$-downward closed and linearly ordered by $<_T$. We say that a branch 
  $b \subseteq T$ is a \emph{cofinal} branch if, for all $\alpha < \height(T)$, 
  we have $b \cap T_\alpha \neq \emptyset$. A subset $A \subseteq T$ is called 
  an \emph{antichain} if elements of $A$ are pairwise incomparable under $<_T$.
  We say that a subset $E \subset T$ is \emph{dense} if, for all $t \in T$, 
  there is $s \in E$ with $t \leq_T s$, and \emph{open} if, for all $t \in E$ 
  and all $s \in T$ with $t \leq_T s$, we have $s \in E$.
\end{definition}

We will mostly be interested in $\omega_1$-trees in this paper, so much of what follows will 
be in the specific context of $\omega_1$-trees, though it will be clear that essentially everything 
can be generalized to the setting of $\kappa$-trees for other regular uncountable cardinals $\kappa$.
Given a finite sequence $\langle T(i) \mid i < n \rangle$ of $\omega_1$-trees, let 
$T(0) \otimes T(1) \otimes \cdots \otimes T(n-1)$ denote the tree $T^*$ defined as follows. 
The underlying set of $T^*$ is
\[
  \bigcup_{\alpha < \omega_1} \prod_{i < n} T(i)_\alpha.
\]
If $\vec{x},\vec{y} \in T^*$, then we set $\vec{x} \leq_{T^*} \vec{y}$ if and only if 
$x_i \leq_{T(i)} y_i$ for all $i < n$. Note that, for all $\alpha < \omega_1$, we have 
$T^*_\alpha = \prod_{i < n} T(i)_\alpha$.

If $T$ is an $\omega_1$-tree and $x \in T$, then we let $T_x$ denote the tree
$\{y \in T \mid x \leq_T y\}$, with the induced order from $T$. As long as $T$ is normal, 
it follows that $T_x$ is an $\omega_1$-tree for all $x \in T$. We note that there is 
a slight abuse of notation here, given our convention that $T_\alpha$ denotes 
the $\alpha^{\mathrm{th}}$ level of the tree $T$, but in practice there will never be 
any risk of confusion: we will always use Roman letter such as ``$x$" to denote 
elements of a tree and Greek letters such as ``$\alpha$" to denote levels of 
the tree. If $\alpha < \omega_1$ and 
$\vec{x}$ is a finite sequence from $T_\alpha$ of length $n < \omega$, then we let 
\[
  T_{\vec{x}} := T_{x_0} \otimes T_{x_1} \otimes \cdots \otimes T_{x_{n-1}}.
\]
Such a tree $T_{\vec{x}}$ is called a \emph{derived tree of $T$}.
Given $\vec{z} \in T_{\vec{x}}$, we will let $\height_T(\vec{z})$ denote the 
unique $\beta < \omega_1$ such that $z_i \in T_\beta$ for all $i < n$.
Recall that a \emph{Suslin tree} is an $\omega_1$-tree $T$ with no uncountable 
branches and no uncountable antichains. It is routine to prove that an 
(infinitely) splitting $\omega_1$-tree $T$ is Suslin if and only if, for 
every dense, open subset $E \subseteq T$, there is $\beta < \omega_1$ such that 
$T_\beta \subseteq E$.
 
We say that a Suslin tree $T$ is a \emph{free Suslin tree} if, for all $\alpha < \omega_1$ and all finite injective 
sequences $\vec{x}$ from $T_\alpha$, the tree $T_{\vec{x}}$ is Suslin. A \emph{Kurepa tree} is an $\omega_1$-tree $T$ with at least $\omega_2$-many 
  cofinal branches. We say that a Suslin tree $T$ is \emph{almost Kurepa} if after forcing with $T$, $T$ is a Kurepa tree. A \emph{weak Kurepa tree} is a tree $T$ of height and size $\omega_1$ with at least $\omega_2$-many cofinal branches. The \emph{Kurepa Hypothesis} ($\KH$) is the assertion that there exists a Kurepa tree. The \emph{weak Kurepa Hypothesis} ($\wKH$) is the assertion that there exists a weak Kurepa tree. 

Given a regular uncountable cardinal $\kappa$, we say that the \emph{tree property holds} at $\kappa$ 
(denoted $\TP(\kappa)$), if every $\kappa$-tree has a cofinal branch. The tree property can be used to 
characterize weakly compact cardinals: an inaccessible cardinal $\kappa$ is weakly compact if and 
only if $\TP(\kappa)$ holds, and, by work of Mitchell and Silver \cite{MIT:tree}, $\TP(\omega_2)$ 
is equiconsistent with the existence of a weakly compact cardinal. Beginning in the 1970s with 
work of Jech \cite{jech_combinatorial} and Magidor \cite{magidor_combinatorial}, and continuing 
in the 2000s with work of Wei\ss\ \cite{weiss_combinatorial}, two-cardinal analogues of the tree 
property have been shown to characterize strongly compact and supercompact cardinals in a similar 
manner. 

Let $\kappa$ be a regular uncountable cardinal, and let $X$ be a set with $\vert X\vert  \geq \kappa$. 
We say that a sequence $\seq{d_x}{x\in\power_\kappa X}$ is a $(\kappa,X)$-list if $d_x\sub x$ for all $x\in \power_\kappa X$.

\begin{definition} \label{thin_slender_def}
Assume that $D=\seq{d_x}{x\in\power_\kappa X}$ is a $(\kappa,X)$-list.
\begin{enumerate}
\item The \emph{width} of $D$ is the least cardinal $\lambda$ such that 
$|d_x| < \lambda$ for all $x \in \power_\kappa X$.
\item We say that $D$ is $\emph{thin}$ if there is a closed unbounded set $C\sub \power_\kappa X$ such that $\vert \set{d_x\cap y}{y\sub x\in \power_\kappa X}\vert <\kappa$ for every $y\in C$.

\item  \label{slender_item} 
Let $\mu\le\kappa$ be an uncountable cardinal. We say that $D$ is $\mu$-\emph{slender} if for all sufficiently large $\theta$ there is a club $C\sub\power_\kappa H(\theta)$ such that for all $M\in C$ and all $y\in M\cap \power_\mu X$, we have $d_{M\cap X}\cap y\in M$.
\end{enumerate}
\end{definition}

\begin{definition}
Assume that $D=\seq{d_x}{x\in\power_\kappa X}$ is a $(\kappa,X)$-list and $d\sub X$.
\begin{enumerate}
\item We say that $d$ is a \emph{cofinal branch} of $D$ if for all $x\in\power_\kappa X$ there is $z_x\supseteq x$ such that $d\cap x=d_{z_x}\cap x$.
\item We say that $d$ is an \emph{ineffable branch} of $D$ if the set $\set{x\in \power_\kappa X}{d\cap x=d_x}$ is stationary.
\end{enumerate}
\end{definition}

\begin{definition} \label{tp_def}
Assume that $\mu\le\kappa$ is regular. We say that 
\begin{enumerate}
\item the $(\kappa,X)$-tree property holds, denoted $\TP(\kappa,X)$, if every thin $(\kappa,X)$-list has a cofinal branch.

\item the $(\kappa,X)$-ineffable tree property holds, denoted $\ITP(\kappa,X)$, if every thin $(\kappa,X)$-list has an ineffable branch.
\item the $(\mu,\kappa,X)$-slender tree property holds, denoted $\SP^\mu(\kappa,X)$, if every $\mu$-slender $(\kappa,X)$-list has a cofinal branch.
\item the $(\mu,\kappa,X)$-ineffable slender tree property holds, denoted $\ISP^\mu(\kappa,X)$, if every $\mu$-slender $(\kappa,X)$-list has an ineffable branch.
\end{enumerate}
\end{definition}

In this paper, we will be interested only in the ineffable slender tree property; for more about two-cardinal tree properties see \cite{weiss_combinatorial}. We introduce a couple of conventions. We will use notations such as $\mathsf{ISP}^\mu(\kappa)$ to assert that $\mathsf{ISP}^\mu(\kappa, \lambda)$ holds for all $\lambda \geq \kappa$. We let $\mathsf{ISP}$ denote $\mathsf{ISP}^{\omega_1}(\omega_2)$.

By \cite{magidor_combinatorial}, it follows that an inaccessible cardinal $\kappa$ is supercompact if and only if $\ISP^{\omega_1}(\kappa)$ holds. 
Wei\ss\ proves in \cite{weiss_combinatorial} that, if one forces with the standard Mitchell 
forcing $\M(\omega, \kappa)$, where $\kappa$ is supercompact, then $\ISP$ holds in the extension; 
similarly, Viale and Wei\ss\ prove in \cite{viale_weiss} that the Proper Forcing Axiom ($\PFA$) 
implies $\ISP$. It is not known whether, in analogy with the tree property, 
$\ISP$ implies the consistency of the existence of a supercompact cardinal; some partial results in 
this direction can be found in \cite{viale_weiss}.

In \cite{viale_weiss}, Viale and Wei\ss\ reformulate ineffable slender tree properties
in terms of the existence of combinatorial objects known as \emph{guessing models}, which we 
now recall.

\begin{definition} \label{guessing_model_def}
Let $\theta \geq \omega_2$ be a regular cardinal, and let $M \prec H(\theta)$ be an elementary submodel.
  \begin{enumerate}
    \item Given a set $x \in M$, a subset $d \subseteq x$, and an uncountable cardinal $\mu$, 
    we say that 
    \begin{enumerate}
      \item $d$ is \emph{$(\mu, M)$-approximated} if, for every $z \in M \cap \power_{\mu}(M)$, 
      we have $d \cap z \in M$;
      \item $d$ is \emph{$M$-guessed} if there is $e \in M$ such that $d \cap M = e \cap M$.
    \end{enumerate}
   \item $M$ is a \emph{$\mu$-guessing model for $x$} if
    every $(\mu, M)$-approximated subset of $x$ is $M$-guessed.
   \item $M$ is a \emph{$\mu$-guessing model} if, for every $x \in M$, it is a $\mu$-guessing 
    model for $x$.
  \end{enumerate}

Given uncountable cardinals $\mu \leq \kappa \leq \theta$, with $\kappa$ and $\theta$ regular, 
we let $\GMP^\mu(\kappa, \theta)$ denote the assertion that the set of $M \in \power_\kappa(H(\theta))$ 
such that $M$ is a $\mu$-guessing model is stationary in $\power_\kappa(H(\theta))$. Moreover, we let $\GMP^\mu(\kappa)$ denote the assertion that $\GMP^\mu(\kappa, \theta)$ holds for every regular cardinal 
$\theta \geq \kappa$. We let $\GMP(\kappa)$ denote $\GMP^{\omega_1}(\kappa)$ and let $\GMP$ denote $\GMP(\omega_2)$.   
\end{definition}

Viale and Wei\ss\ prove in \cite{viale_weiss} that the ineffable slender tree property is equivalent to the guessing model property. More precisely, they proved the following:

\begin{fact}\label{isp_guessing_fact}
  Suppose that $\mu \leq \kappa$ are regular uncountable cardinals. Then the following are equivalent:
  \begin{enumerate}
    \item $\ISP^\mu(\kappa)$;
    \item $\GMP^\mu(\kappa)$.
  \end{enumerate}
In particular $\ISP$ is equivalent to $\GMP$
\end{fact}

\subsection{Preservation lemmas} In this short subsection we recall some useful preservation lemmas 
regarding product forcings and adding cofinal branches through trees.
The following fact is due to Baumgartner (see \cite{BAUM:iter}).
 
\begin{fact}\label{F:Knaster}
Let $\kappa$ be a regular cardinal and assume that $\P$ is a $\kappa$-Knaster forcing notion. If $T$ is a tree of height $\kappa$, then forcing with $\P$ does not add new cofinal branches to $T$.
\end{fact}

The following lemma appeared in \cite{EASTONregular}.

\begin{fact}\label{L:Easton}\emph{(Easton)}
Let $\kappa>\omega$ be a regular cardinal and assume that $\P$ and $\Q$ are forcing notions, where $\P$ is $\kappa$-cc and $\Q$ is $\kappa$-closed. Then the following hold:
\begin{enumerate}[(i)]
\item $\P$ forces that $\Q$ is $\kappa$-distributive.\label{L:Easton_Dis}
\item $\Q$ forces that $\P$ is $\kappa$-cc.\label{L:Easton_cc}
\end{enumerate}
\end{fact}

The following fact can be found in \cite{C:trees}. The argument is attributed to Magidor.

\begin{fact}\label{F:ccc-Knaster}
Let $\kappa>\omega$ be a regular cardinal and assume that $\P$ and $\Q$ are forcing notions such that $\P$ is $\kappa$-Knaster and $\Q$ is $\kappa$-cc. Then $\Q$ forces that $\P$ is $\kappa$-Knaster.
\end{fact}

The following fact is due to Silver (see \cite{abraham} for more details).

\begin{fact}\label{F:Closed}
Let $\kappa$, $\lambda$ be regular cardinals with $2^{\kappa}\geq\lambda$. Assume that $\P$ is a $\kappa^{+}$-closed forcing notion. If $T$ is a $\lambda$-tree, then forcing with $\P$ does not add new cofinal branches to $T$.
\end{fact}

These fact can be generalized as follows \cite{UNGER:1}.

\begin{fact}\label{F:ccc_Closed}
Let $\kappa < \lambda$ be regular cardinals and $2^{\kappa}\geq\lambda$. Assume that $\P$ and $\Q$ are forcing notions such that $\P$ is $\kappa^{+}$-cc and $\Q$ is $\kappa^{+}$-closed. If $T$ is a 
$\lambda$-tree in $V[\P]$, then forcing with $\Q$ over $V[\P]$ does not add cofinal branches to $T$.
\end{fact}

In \cite{kurepa_paper}, we prove an analogue of the previous lemma for more general trees, and as a corollary we obtain the following fact:


\begin{fact}\label{closed_preservation_fact}
Let $\xi$ be a cardinal and $\mu < \kappa\le\lambda$ be regular cardinals such that $2^\mu\ge \xi$ and $2^{<\mu}<\kappa$. Let $\Q$ be a $\mu^+$-closed forcing notion and $\P$ be a $\mu^+$-cc forcing notion. If $T$ is a $\power_\kappa\lambda$-tree with width at most $\xi$ in $V[\P]$, then forcing with $\Q$ over $V[\P]$ does not add a cofinal branch through $T$.
\end{fact}

\subsection{Mitchell forcing} \label{mitchell_section}

Let $\kappa$ be an ordinal with $\cf(\kappa) > \omega$. In practice, $\kappa$ will typically be (at least) an inaccessible cardinal, but we will sometimes need to consider $\kappa$ of the form $j(\lambda)$, where $j:V \rightarrow M$ is an 
elementary embedding with critical point $\lambda$, in which case $\kappa$ is inaccessible in $M$ but 
may not even be a cardinal in $V$.

For this subsection, let $\P$ denote the forcing $\Add(\omega, \kappa)$ for adding $\kappa$-many Cohen subsets 
to $\omega$.
For $\gamma < \kappa$, let $\P_\gamma$ denote the suborder $\Add(\omega, \gamma)$. We can now 
define the version of the Mitchell forcing that we will use, $\M=\M(\omega, \kappa)$, as follows. First, let $A$ be some unbounded subset of $\kappa$ with $\min(A) > \omega$. Recall that $\acc(A)$, the set of \emph{accumulation points} 
of $A$, is defined to be $\{\alpha \in A \mid \sup(A \cap \alpha) = \alpha\}$, and $\nacc(A) := A \setminus \acc(A)$. Also, $\acc^+(A):= \{\alpha < \sup(A) \mid \sup(A \cap \alpha) 
= \alpha\}$. Though our definition of $\M$ will depend on our choice of $A$, we suppress mention of 
$A$ in the notation; if we need to make the choice of $A$ explicit, we will speak of ``the Mitchell 
forcing $\M(\omega, \kappa)$ defined using the set $A$". 
Unless otherwise specified, we will always let $A$ be the set of
all inaccessible cardinals in the interval $(\omega, \kappa)$, but the definition of $\M$ and the properties 
listed below work for any choice of $A$.  Conditions in $\M$ are all pairs 
$(p, q)$ such that 
\begin{itemize}
  \item $p \in \bb{P}$;
  \item $q$ is a function and $\dom(q) \in [\nacc(A)]^{\leq \omega}$;
  \item for all $\alpha \in \dom(q)$, $p \restriction \alpha \Vdash_{\P_\alpha}$ 
  ``$q(\alpha) \in \Add(\omega_1,1)^{V^{\P_\alpha}}$''.
\end{itemize}
If $(p_0,q_0), (p_1,q_1) \in \M$, then $(p_1,q_1) \leq_\M (p_0,q_0)$ if and only if
\begin{itemize}
  \item $p_1 \leq_{\P} p_0 $;
  \item $\dom(q_1) \supseteq \dom(q_0)$;
  \item for all $\alpha \in \dom(q_0)$, $p_1 \restriction \alpha \Vdash_{\P_\alpha}$ 
  ``$q_1(\alpha) \leq q_0(\alpha)$''. 
\end{itemize} 
For $\delta < \kappa$, we let $\M_\delta$ denote the suborder of $\M$ consisting of all conditions 
$(p,q)$ such that the domains of both $p$ and $q$ are contained in $\delta$.

\begin{remark} \label{mitchell_remark}
The following are some of the key properties of $\M$ (cf.\ 
\cite{abraham}, \cite{MIT:tree} for further details and proofs):
\begin{enumerate}
  \item If $\kappa$ is inaccessible, then $\M$ is $\kappa$-Knaster.
  \item $\M$ has the $\omega_1$-covering and $\omega_1$-approximation properties. Together with the previous 
  item, this implies that, if $\kappa$ is inaccessible, then forcing with $\M$ preserves $\omega_1$ and all cardinals greater than or equal to $\kappa$.
  \item If $\kappa$ is inaccessible, then $\Vdash_{\M} $``$2^\omega = \kappa = \omega_2$''.
  \item There is a projection onto $\M$ from a forcing of the form $\Add(\omega, \kappa) \times 
  \Q$, where $\Q$ is $\omega_1$-closed (here $\Q$ is what is often called the \emph{term forcing} associated with $\M$; it can be thought of as the set of all pairs $(\emptyset,q)$ from $\M$, with the order inherited from $\M$).\label{mitchell_product}
  \item \label{quotient_approx_prop} 
  For all inaccessible cardinals $\delta \in \acc^+(A)$,
  there is a projection from $\M$ to $\M_\delta$ and, in $V^{\M_\delta}$, the quotient 
  forcing $\M/\M_\delta$ has the $\omega_1$-approximation property.
  \item \label{quotient_product_prop} 
  For all inaccessible cardinals $\delta \in \acc^+(A)$, let $\delta^\dagger$ denote 
  $\min(A \setminus \delta+1)$. Then, in $V^{\M_\delta}$, the quotient 
  forcing $\M/\M_\delta$ is of the form $\Add(\omega, \delta^\dagger - \delta) \ast \dot{\M}^\delta$, where, 
  in $V^{\M \ast \Add(\omega, \delta^\dagger - \delta)}$, there is a projection onto $\M^\delta$ from a forcing 
  of the form $\Add(\omega, \kappa - \delta^\dagger) \times \Q^*_\delta$, where $\Q^*_\delta$ is $\omega_1$-closed.
\end{enumerate}
\end{remark}

\subsection{Automorphisms of trees}

If $T$ is a tree, then an \emph{automorphism} of $T$ is a bijective 
map $g:T \ra T$ such that $g$ and $g^{-1}$ both preserve the order $<_T$.
Note that an automorphism must be level-preserving, i.e., for all 
$\alpha < \height(T)$, $g[T_\alpha] = T_\alpha$. We let $\mathrm{Aut}(T)$ denote 
the set of all automorphisms of $T$.

Fix for now a normal, infinitely splitting $\omega_1$-tree $T$.
If $\alpha < \beta < \omega_1$ and $x \in T_\beta$, then $x \restriction 
\alpha$ denotes the unique predecessor of $x$ in $T_\alpha$. If $X \subseteq T_\beta$, 
then $X \restriction \alpha$ denotes $\{x \restriction \alpha \mid x \in X\}$. We 
say that $X$ has \emph{unique drop-downs to $\alpha$} if, for all distinct 
$x,y \in X$, we have $x \restriction \alpha \neq y \restriction \alpha$. 
Notice that, if $\alpha < \omega_1$, then an element of $\mathrm{Aut}(T_{\leq \alpha})$ 
is uniquely determined by its values on $T_\alpha$.

Using the normality of $T$, the following proposition is immediate.

\begin{proposition} \label{prop: limit_extension}
  Suppose that $\beta < \omega_1$ is a limit ordinal and $g \in \mathrm{Aut}(T_{<\beta})$. 
  Then there exists at most one $g' \in \mathrm{Aut}(T_{\leq \beta})$ such that 
  $g' \supseteq g$, and the following are equivalent:
  \begin{enumerate}
    \item there exists $g' \in \mathrm{Aut}(T_{\leq \beta})$ such that $g' \supseteq g$;
    \item for all $x \in T_\beta$ and $j \in \{-1,1\}$, there exists $y \in T_\beta$ such 
    that, for all $\alpha < \beta$, we have $g^j(x \restriction \alpha) = y \restriction \alpha$.
  \end{enumerate}
\end{proposition}

The straightforward proof of the following proposition, showing that we have some 
control over extensions of automorphisms of initial segments of $T$, is left 
to the reader.

\begin{proposition} \label{prop: ext_prop}
  Suppose that $\alpha < \omega_1$ and $g \in \mathrm{Aut}(T_{\leq \alpha})$. 
  Suppose moreover that
  \begin{itemize}
    \item $X \subseteq T_{\alpha+1}$ is finite;
    \item for all $x \in X$ and $j \in \{-1,1\}$, we are given 
    $z^j_x \in T_{\alpha+1}$ in such a way that
    \begin{itemize}
      \item $g^j(x \restriction \alpha) = z^j_x \restriction \alpha$;
      \item if $x_0$ and $x_1$ are distinct elements of $X$ and $j \in \{-1,1\}$, 
      then $z^j_{x_0} \neq z^j_{x_1}$;
      \item if $x_0 \in X$, $j \in \{-1,1\}$, and $z^j_{x_0} = x_1 \in X$, then 
      $z^{1-j}_{x_1} = x_0$.
    \end{itemize}
  \end{itemize}
  Then there exists $h \in \mathrm{Aut}(T_{\leq \alpha+1})$ such that 
  \begin{itemize}
    \item $h \supseteq g$; and 
    \item for all $(x,j) \in X \times \{-1,1\}$, we have $h^j(x) = z^j_x$.
  \end{itemize}
\end{proposition}

We now recall some notions and basic results from \cite{ks}. Some definitions and lemmas are 
given descriptive names for ease of later referral; these names also come from \cite{ks}.

\begin{definition}[Consistency] \label{def: con}
  Suppose that $\alpha < \alpha' \leq \beta < \omega_1$ and 
  $g \in \aut(T_{\leq \beta})$. Suppose 
  moreover that $X \subseteq T_{\alpha'}$ has unique drop-downs to $\alpha$. We say that 
  $X \restriction \alpha$ and $X$ are \emph{$g$-consistent} if, for all $x,y \in X$, 
  we have $g(x) = y$ iff $g(x \restriction \alpha) = g(y \restriction \alpha)$. 
  Similarly, if $n < \omega$ and $\vec{x}$ is an injective $n$-tuple from $T_{\alpha'}$, 
  we say that $\vec{x} \restriction \alpha$ and $\vec{x}$ are 
  $g$-consistent if, for all $i,j < n$, we have $g(x_i) = x_j$ iff $g(x_i \restriction 
  \alpha) = x_j \restriction \alpha$.
\end{definition}

\begin{proposition} \label{prop: 54}
  Suppose that $\alpha < \omega_1$, $g \in \mathrm{Aut}(T_{\leq \alpha})$, 
  $X \subseteq T_{\alpha+1}$ is a finite set with unique drop-downs to $\alpha$, and 
  $Y \subseteq T_{\alpha+1}$ is a finite set disjoint from $X$. Then there is 
  $g' \in \mathrm{Aut}(T_{\leq \alpha + 1})$ such that
  \begin{itemize}
    \item $g' \supseteq g$;
    \item $X \restriction \alpha$ and $X$ are $g'$-consistent;
    \item $g'[Y] \cap (X \cup Y) = \emptyset$.
  \end{itemize}    
\end{proposition}

\begin{proof}
  For each $x \in T_\alpha$, let $\langle x_n \mid n < \omega \rangle$ be an injective 
  enumeration of $\mathrm{succ}_T(x)$. Let $X_0 := \{x \in T_\alpha \mid 
  X \cap \mathrm{succ}_T(x)\} \neq \emptyset$. Note that, since $X$ has unique drop-downs 
  to $\alpha$, for each $x \in X_0$, there is a unique $n < \omega$ such that 
  $x_n \in X$; denote this $n$ by $n_X(x)$. Similarly, let $Y_0 = \{x \in T_\alpha 
  \mid Y \cap \mathrm{succ}_T(x)\} \neq \emptyset$ and, for all $x \in Y_0$, let 
  $A_Y(x) = \{n < \omega \mid x_n \in Y\}$. Then each $A_Y(x)$ is a finite subset of 
  $\omega$ and does not contain $n_X(x)$ if the latter is defined. If $x \notin Y_0$, 
  let $A_Y(x) = \emptyset$.
  
  Now, for each $x \in T_\alpha$, fix a permutation $\pi_x$ of $\omega$ with the 
  following properties:
  \begin{itemize}
    \item if $x, g(x) \in X_0$, then $\pi_x(n_X(x)) = n_X(g(x))$;
    \item if $x \in Y_0$, then $\pi_x[A_Y(x)]$ is disjoint from 
    $A_Y(g(x))$ and does not contain $n_X(g(x))$ if the latter is defined.
  \end{itemize}
  Now define $g' \in \mathrm{Aut}(T_{\leq \alpha + 1})$ by setting 
  $g'(x_n) = g(x)_{\pi_x(n)}$ for all $(x,n) \in T_\alpha \times \omega$. 
  It is easily verified that $g'$ is as desired.
\end{proof}

\begin{definition}[Separation]
  Suppose that $\alpha \leq \beta < \omega_1$, $\mathcal{G} = \{g_\tau \mid \tau \in I\}$ is 
  an indexed family from $\aut(T_{\leq \beta})$, $n < \omega$ and $\vec{x}$ is an injective 
  $n$-tuple from $T_\alpha$. We say that 
  $\mc{G}$ is \emph{separated on $\vec{x}$} if, for all $m < n$, the following 
  statements hold:
  \begin{enumerate}
    \item for all $\tau \in I$, $g_\tau(x_m) \neq x_m$;
    \item there exists at most one triple $(k,j,\tau)$ such that $k < m$, $j \in \{-1,1\}$, 
    $\tau \in I$, and $g^j_\tau(x_m) = x_k$.
  \end{enumerate}
  If $X \subseteq T_\alpha$ is finite, then we say that $\mc{G}$ is \emph{separated on $X$} 
  if there is some injective enumeration $\vec{x}$ of $X$ such that $\mc{G}$ is separated on 
  $\vec{x}$. We say that $\mc{G}$ is \emph{separated} if it is separated on $X$ for every finite 
  $X \subseteq T_\beta$.
\end{definition}

The next five technical results are proved in \cite{ks}; we direct the reader there for 
the proofs.

\begin{proposition}[Persistence] {\cite[Lemma 5.7]{ks}} \label{prop: ksLem57}
  Suppose that $\alpha < \beta < \omega_1$, $\mc{G}$ is an indexed 
  family from $\mathrm{Aut}(T_{\leq\beta})$, and $X \subseteq T_\beta$ is a finite 
  set with unique drop-downs to $\alpha$. If $\mc{G}$ is 
  separated on $X \restriction \alpha$, then $\mc{G}$ is separated on $X$.
\end{proposition}

\begin{lemma}[Key Property] {\cite[Proposition 5.11]{ks}} \label{lemma: ksProp511}
  Suppose that $\alpha < \beta < \omega_1$, $n < \omega$, $\vec{x}$ is an injective 
  $n$-tuple from $T_\alpha$, and $\mc{G} = \{g_\tau \mid \tau \in I\}$ is a finite indexed 
  set from $\mathrm{Aut}(T_{\leq \beta})$ that is 
  separated on $\vec{x}$. Let $t \subseteq T_\beta$ be finite. Then there is an 
  $n$-tuple $\vec{y}$ from $T_\beta \setminus t$ such that
  \begin{itemize}
    \item for all $i < n$, we have $x_i <_T y_i$;
    \item for all $\tau \in I$, $\vec{x}$ and $\vec{y}$ are $g_\tau$-consistent.
  \end{itemize}
\end{lemma}

\begin{lemma}[1-Key Property] {\cite[Proposition 5.12]{ks}} \label{lemma: ksProp512}
  Suppose that $\alpha < \beta < \omega_1$, $n < \omega$, $\vec{x}$ is an injective 
  $n$-tuple from $T_\alpha$, and $\mc{G} = \{g_\tau \mid \tau \in I\}$ is a finite 
  indexed set from $\mathrm{Aut}(T_{\leq \beta})$ that is separated on $\vec{x}$. Let $m < n$, and fix $y^* \in T_\beta$ 
  with $x_m <_T y^*$. Then there is an $n$-tuple $\vec{y}$ from $T_\beta$ such that
  \begin{itemize}
    \item $y_m = y^*$;
    \item for all $i < n$, we have $x_i <_T y_i$;
    \item for all $\tau \in I$, $\vec{x}$ and $\vec{y}$ are $g_\tau$-consistent.
  \end{itemize}
\end{lemma}

\begin{lemma}[Extension] {\cite[Proposition 5.15]{ks}} \label{lemma: ksProp515}
  Suppose that $\alpha < \beta < \omega_1$, $X \subseteq T_\beta$ is a finite set with 
  unique drop-downs to $\alpha$, $\{f_\tau \mid \tau \in I\}$ is a countable indexed set 
  from $\mathrm{Aut}(T_{\leq \alpha})$, and $A \subseteq I$ is finite. 
  Then there exists a collection $\{g_\tau \mid \tau \in I\}$ from 
  $\mathrm{Aut}(T_{\leq \beta})$ such that
  \begin{itemize}
    \item for all $\tau \in I$, we have $f_\tau \subseteq g_\tau$;
    \item for all $\tau \in A$, $X \restriction \alpha$ and $X$ are $g_\tau$-consistent;
    \item if $\{f_\tau \mid \tau \in A\}$ is separated on $X \restriction \alpha$, 
    then $\{g_\tau \mid \tau \in I\}$ is separated.
  \end{itemize}
\end{lemma}

We will also need the following slight variation of the preceding lemma; it is 
a special case of \cite[Lemma 5.39]{ks}, so we refer the reader there for a proof.

\begin{lemma} \label{lemma: ksLem539}
  Suppose that $\alpha < \omega_1$, $X_0$ and $X_1$ are finite subsets of 
  $T_{\alpha+1}$ with unique drop-downs to $\alpha$ and $X_0 \cap X_1 = \emptyset$ 
  (though possibly $(X_0 \restriction \alpha) \cap (X_1 \restriction \alpha) \neq 
  \emptyset$), and $g \in \mathrm{Aut}(T_{\leq \alpha})$. Then there is 
  $g' \in \mathrm{Aut}(T_{\leq \alpha+1})$ such that $g' \supseteq g$ and, for each
  $i < 2$, $X_i \restriction \alpha$ and $X_i$ are $g'$-consistent.
\end{lemma}

We now recall the forcing poset designed to add a generic automorphism to a fixed 
$\omega_1$-tree by initial segments of successor height. This poset will be a building 
block in the generalized Mitchell forcing used to prove Theorem A in the next section.

\begin{definition}
  Given an $\omega_1$-tree $T$, let $\bb{A}(T)$ be the forcing poset whose underlying set 
  is $\bigcup \{\aut(T_{\leq \alpha}) \mid \alpha < \omega_1\}$, ordered by reverse inclusion. 
  Given $g \in \bb{A}(T)$, let $\top(g)$ be the unique $\alpha < \omega_1$ such that 
  $g \in \aut(T_{\leq \alpha})$. We will sometimes refer to $\top(g)$ as the 
  \emph{top level} of $g$.
\end{definition}

\begin{remark}
  If $T$ is a normal, infinitely splitting $\omega_1$-tree and $G$ is 
  an $\bb{A}(T)$-generic filter over 
  $V$, then, by Lemma \ref{lemma: ksProp515} and a standard genericity argument, 
  $g^* := \bigcup G$ is an automorphism of $T$. Note that, in general, 
  forcing with $\bb{A}(T)$ may collapse $\omega_1$, but under certain assumptions 
  on $T$, for instance, if $T$ is a free Suslin tree, $\bb{A}(T)$ will be 
  totally proper and hence will preserve $\omega_1$ (cf.\ \cite[Theorem 5.34]{ks}).
\end{remark}

Although $\bb{A}(T)$ may preserve $\omega_1$, the next proposition shows that, if 
$\CH$ fails, then it always collapses the continuum.

\begin{proposition} \label{prop: collapse}
  Suppose that $T$ is an infinitely splitting, normal $\omega_1$-tree. Then 
  \[
    \Vdash_{\bb{A}(T)} ``|(2^\omega)^V| = |\omega_1^V|".
  \]
\end{proposition}

\begin{proof}
  In $V$, fix, for each $x \in T$, an enumeration $\langle x_n \mid n < \omega \rangle$ 
  of $\mathrm{succ}_T(x)$. Recall that 
  $S_\infty$ denotes the group of permutations of $\omega$. If $h$ is an automorphism of 
  $T$, then $h$ induces a map $\rho_h : T \ra S_\infty$ as follows: for each 
  $x \in T$ and each $n < \omega$, let $\rho_h(x)(n)$ be the unique 
  $m < \omega$ such that $h(x_n) = h(x)_m$. A standard genericity argument 
  shows that, if $G$ is $\bb{A}(T)$-generic over $V$ and 
  $g^* = \bigcup G$, then, in $V[G]$, $\rho_{g^*}$ is a surjection from 
  $T$ to $(S_\infty)^V$. Since $|T| = \omega_1^V$ and 
  $|(S_\infty)^V| = |(2^\omega)^V|$, the desired conclusion follows.
\end{proof}

\section{Guessing models with an almost Kurepa Suslin tree} \label{sect: thma_sect}

In this section, we present our proof of Theorem A, constructing a model of $\ZFC$ 
in which $\GMP$ holds and there exists an almost Kurepa Suslin tree. 
In order to motivate the construction, let us first recall the construction from 
\cite{ks} of a model in which there exists an almost Kurepa Suslin tree but there 
does not exist a Kurepa tree. There, the authors begin in a model of $\ZFC$ with an 
inaccessible cardinal $\kappa$. They then force with the product 
$\mathrm{Coll}(\omega_1, {<}\kappa) \times \P$, where $\P$ is a countable support 
product of $\kappa$-many copies of $\bb{A}(T)$, where $T$ is a fixed free Suslin tree in $V$.
We note that collapsing an inaccessible cardinal is necessary here: if there are no 
Kurepa trees in a model $W$ of $\ZFC$, then $(\omega_2)^W$ is inaccessible in $\mathrm{L}$.
Note that the model $V[\mathrm{Coll}(\omega_1, {<}\kappa) \times \P]$ satisfies $\CH$; 
and therefore contains weak Kurepa trees. If we want to arrange a model of 
$\GMP$, or even of its consequence $\neg \wKH$, then we must simultaneously increase the value 
of the continuum as we are collapsing the large cardinal $\kappa$. This requirement 
leads us naturally to consider Mitchell forcing. Since we also want our final model to 
contain an almost Suslin Kurepa tree, we will work with a variation of the classical 
Mitchell forcing obtained by replacing $\mathrm{Add}(\omega_1,1)$ from the classical Mitchell 
forcing with the forcing $\bb{A}(T)$ from the end of the previous section.

\subsection{A Mitchell variation}
Fix for this section an $\omega_1$-tree $T$ and a cardinal $\kappa$. Let 
$\P = \P_\kappa = \Add(\omega, \kappa)$; for $\gamma < \kappa$, let $\P_\gamma = \Add(\omega, \gamma)$. Define a forcing poset $\M = \M^T_\kappa$ as follows, recalling that $\Sigma(\kappa)$ denotes the 
set of successor ordinals less than $\kappa$:
\begin{itemize}
  \item Conditions of $\M$ are all pairs $(p,q)$ such that 
  \begin{itemize}
    \item $p \in \P$;
    \item $q$ is a function whose domain is a countable subset of 
    $\Sigma(\kappa)$;
    \item for all $\alpha \in \dom(q)$, $q(\alpha)$ is a $\P_\alpha$-name 
    for an element of $\bb{A}(T)^{V[\P_\alpha]}$.
  \end{itemize}
  \item Given $(p_0,q_0),(p_1,q_1) \in \M$, we set $(p_1,q_1) \leq_{\M} 
  (p_0,q_0)$ if and only if
  \begin{itemize}
    \item $p_1 \leq_{\P} p_0$;
    \item $\dom(q_1) \supseteq \dom(q_0)$;
    \item for all $\alpha \in \dom(q_0)$, $p_1 \restriction \alpha 
    \Vdash_{\P_\alpha} ``q_1(\alpha) \leq_{\bb{A}(T)} q_0(\alpha)"$.
  \end{itemize}
\end{itemize}
For $\delta < \kappa$, let $\M_\delta$ denote the 
set of $(p,q) \in \M$ such that $p \in \P_\delta$ and $\dom(q) \subseteq \delta$.
Note that $\M_\delta$ is a regular suborder of $\M$. Given 
$(p,q) \in \M$, we let $(p,q) \restriction \delta$ denote 
$(p \restriction \delta, q \restriction \delta) \in \M_\delta$. It is easily 
verified that the map $(p,q) \mapsto (p,q) \restriction \delta$ is a projection 
from $\M$ to $\M_\delta$.

We let $\Q$ denote the \emph{term forcing} associated with $\M$. More precisely, 
conditions in $\Q$ are all functions $q$ such that $(1_{\P},q) \in \M$. If 
$q,q' \in \Q$, then we let $q' \leq_{\Q} q$ if and only if $(1_{\P}, q') \leq_{\M} (1_{\P}, q)$.

\begin{remark} \label{remark: term_forcing}
  As with the classical Mitchell forcing (recall Subsection \ref{mitchell_section}), 
  there is a natural projection from $\P \times \Q$ to $\M$. Many arguments involving 
  Mitchell forcing, including those in Section \ref{sect: thmb_sect} below, proceed 
  via considerations of this product $\P \times \Q$, which is nicely behaved due to 
  the fact that in the classical setting $\Q$ is $\omega_1$-closed. In the setting of 
  this section, though, $\Q$ is certainly not $\omega_1$-closed, and even though 
  we will see that $\M$ is nicely behaved, it is not even clear to us whether 
  forcing with $\Q$ over $V$ preserves $\omega_1$. (For an example of a natural 
  two-step iteration $\Add(\omega,1) \ast \dot{\Q}$ such that $\dot{\Q}$ is forced 
  to be $\omega_1$-distributive (and even totally proper) and yet forcing with the 
  associated term forcing over $V$ collapses $\omega_1$, see Appendix \ref{appendix_a}
  below.) 
  For this reason, the arguments in this section do not proceed via the familiar 
  product analysis of Mitchell-type forcings.
\end{remark}

\begin{lemma} \label{lemma: constant_height}
  Let $\Q' = \{q \in \Q \mid \exists \alpha < \omega_1 \forall \gamma \in 
  \dom(q) \ \Vdash_{\P_\gamma} ``\top(q(\gamma)) = \alpha"\}$. Then 
  $\Q'$ is dense in $\Q$.
\end{lemma}

\begin{proof}
  Fix $q_0 \in \Q$. For each $\gamma \in \dom(q_0)$, let 
  \[
    \alpha_\gamma = \sup\{\beta < \omega_1 \mid \exists p \in \P_\gamma 
    \ [p \Vdash \top(q_0(\gamma)) = \beta]\},
  \]
  and let $\alpha = \sup\{\alpha_\gamma \mid \gamma \in \dom(q_0)\}$. Since each $\P_\gamma$ has the 
  ccc and $\dom(q_0)$ is countable, it follows that $\alpha < \omega_1$ and, for all 
  $\gamma \in \dom(q_0)$, $\Vdash_{\P_\gamma} ``\top(q_0(\gamma)) 
  \leq \alpha"$. Let $q$ be a function such that $\dom(q) = \dom(q_0)$ and, 
  for all $\gamma \in \dom(q)$, $q(\gamma)$ is a $\P_\gamma$-name such that
  \[
    \Vdash_{\P_\gamma} ``q(\gamma) \leq q_0(\gamma) \text{ and } 
    \top(q(\gamma)) = \alpha."
  \]
  (To see that such a name can be found, repeatedly apply instances of Lemma \ref{lemma: ksProp515}.)
  Then $q \in \Q'$ and $q \leq_{\Q} q_0$.
\end{proof}

%

Going forward, we will typically assume without comment that the conditions in $\Q$ 
that we are considering come from $\Q'$, and, if $q \in \Q'$, then we will let 
$\top(q)$ denote the unique $\alpha < \omega_1$ witnessing that $q \in \Q'$.
Similarly, we will typically assume without comment that all conditions $(p,q) 
\in \M$ that we are considering are such that $q \in \Q'$. 

We now introduce some definitions and lemmas to connect this modified Mitchell 
forcing with some of the concepts and results from the end of the previous section.

\begin{definition}[Mitchellized consistency]
  Suppose that $(p,q) \in \M$, $A \subseteq \dom(q)$ is finite, $\beta = 
  \top(q)$, $\alpha < \alpha' \leq \beta$, and $X \subseteq T_{\alpha'}$ 
  has unique drop-downs to $\alpha$. We say that $X \restriction \alpha$ 
  and $X$ are \emph{$((p,q), A)$-consistent} if, for all 
  $\gamma \in A$, 
  \[
    p \restriction \gamma \Vdash_{\P_\gamma} ``X \restriction \alpha 
    \text{ and } X \text{ are } q(\gamma)\text{-consistent}".
  \]
  We simply say that $X \restriction \alpha$ and $X$ are 
  \emph{$(q,A)$-consistent} if they are $((1_{\P},q),A)$-consistent.
  
  As in Definition \ref{def: con}, we similarly define the notion of 
  $((p,q),A)$-consistency for $\vec{x} \restriction \alpha$ and 
  $\vec{x}$, where $\vec{x}$ is an injective tuple from 
  $T_{\alpha'}$ with unique drop-downs to $\alpha$.
\end{definition}

\begin{definition}[Mitchellized separation]
  Suppose that $(p,q) \in \M$, $A \subseteq \dom(q)$ is finite, 
  $\alpha \leq \beta = \top(q)$, 
  $0 < n < \omega$, and $\vec{x}$ is an injective $n$-tuple from $T_\alpha$. 
  \begin{itemize}
    \item We say that $p$ \emph{decides $q$ on $(A,\vec{x})$} if, 
    for all $\gamma \in A$ and $m < n$, $p \restriction \gamma$ decides the value of 
    $q(\gamma)(x_m)$ and $q(\gamma)^{-1}(x_m)$, say as $g_\gamma(x_m)$ and 
    $g^{-1}_\gamma(x_m)$.
    \item We say that $(p,q)$ is \emph{$A$-separated on $\vec{x}$} if
    \begin{itemize}
      \item $p$ decides $q$ on $(A,\vec{x})$, with $g_\gamma(x_m)$ and 
      $g^{-1}_\gamma(x_m)$ for $\gamma \in A$ and $m < n$ as in the previous point;
      \item for all $\gamma \in A$ and $m < n$, $g_\gamma(x_m) \neq x_m$; and
      \item for all $m < n$, there exists at most triple $(k,j,\gamma)$ such that $k < m$, 
      $j \in \{-1,1\}$, $\gamma \in A$, and $g^j_\gamma(x_m) = x_k$.
    \end{itemize}
    \item We say that $q$ is \emph{$A$-separated on $\vec{x}$ below $p$} if, for all 
    $p' \leq p$, if $p'$ decides $q$ on $(A,\vec{x})$, then $(p',q)$ is 
    $A$-separated on $\vec{x}$. We will simply say that $q$ is \emph{$A$-separated 
    on $\vec{x}$} if it is $A$-separated on $\vec{x}$ below $1_{\P}$.
\end{itemize}
  If $X$ is a finite subset of $T_\alpha$, we say that $p$ decides $q$ on 
  $(A,X)$ if for some (equivalently, any) enumeration $\vec{x}$ of $X$, 
  $p$ decides $q$ on $(A,\vec{x})$, and we say that $(p,q)$ is 
  $A$-separated on $X$ if there is an enumeration $\vec{x}$ of $X$ 
  such that $(p,q)$ is $A$-separated on $\vec{x}$. We say that $q$ is 
  $A$-separated on $X$ below $p$ if, for all $p' \leq p$, if $p'$ decides $q$ 
  on $(A,X)$, then $(p',q)$ is $A$-separated on $X$.
\end{definition}

Note that, if $(p',q') \leq_{\M} (p,q)$ and $(p,q)$ is $A$-separated on 
$\vec{x}$ for some appropriate choice of $A$ and $\vec{x}$, then $(p',q')$ 
is also $A$-separated on $\vec{x}$. Similarly, if $q$ is $A$-separated on $\vec{x}$ below $p$, then $q'$ is $A$-separated on $\vec{x}$ below $p'$. The following proposition is an immediate consequence of the definitions and Proposition 
\ref{prop: ksLem57}.

\begin{proposition} [Mitchellized persistence]\label{prop: persistent_separation}
  Suppose that 
  \begin{itemize}
    \item $(p,q) \in \M$ with $\top(q) = \beta$;
    \item $\alpha < \alpha' \leq \beta$;
    \item $A \in [\dom(q)]^{<\omega}$ and $X \in [T_{\alpha'}]^{<\omega}$ 
    has unique dropdowns to $\alpha$.
  \end{itemize}
  If $q$ is $A$-separated on $X \restriction \alpha$ below $p$, then 
  $q$ is also $A$-separated on $X$ below $p$. \qed
\end{proposition}

\subsection{Technical lemmas}
In this subsection, we present some technical lemmas concerning the construction of conditions 
in $\M$ having certain desired properties. Since these lemmas are motivated primarily by 
their eventual deployment in the proof of Theorem A, the reader may wish to skip this subsection 
on first read and return to it when the relevant lemmas are invoked in later subsections.
\begin{lemma} \label{lemma: separated_extension}
  Suppose that
  \begin{itemize}
    \item $q \in \Q$, with $\top(q) = \alpha$;
    \item $X$ is a finite subset of $T_{\alpha+1}$ with unique drop-downs to $\alpha$;
    \item $A \in [\dom(q)]^{<\omega}$ and $\delta \in \kappa \setminus A$;
    \item $q$ is $A$-separated on $X \restriction \alpha$.
  \end{itemize}
  Then there is $q' \leq_{\Q} q$ such that
  \begin{enumerate}
    \item $\top(q') = \alpha+1$;
    \item $\delta \in \dom(q')$;
    \item $X \restriction \alpha$ and $X$ are $(q',A)$-consistent;
    \item $q'$ is $(A \cup \{\delta\})$-separated on $X$.
  \end{enumerate}
\end{lemma}

\begin{proof}
  First note that we can assume that $\delta \in \dom(q)$, as otherwise we first 
  extend $q$ to the condition $q^* \leq_{\Q} q$ defined by setting 
  $\dom(q^*) = \dom(q) \cup \{\delta\}$ and $q^* \restriction \dom(q) = q$, and 
  letting $q^*(\delta)$ be an arbitrary $\P_\delta$-name for an element of 
  $\mathrm{Aut}(T_{\leq \alpha})$.
  
  Now let $q' \leq_{\Q} q$ be such that
  \begin{itemize}
    \item $\dom(q') = \dom(q)$;
    \item for all $\gamma \in A$, $q'(\gamma)$ is a $\P_\gamma$-name for an element of 
    $\mathrm{Aut}(T_{\leq\alpha+1})$ extending $q(\gamma)$ such that
    \[
      \Vdash_{\P_\gamma} ``X \restriction \alpha \text{ and } 
      X \text{ are } q'(\gamma)\text{-consistent}"
    \]
    (this is possible by Lemma \ref{lemma: ksProp515} applied in $V[\P_\gamma]$);
    \item $q'(\delta)$ is a $\P_\gamma$-name for an element of 
    $\mathrm{Aut}(T_{\leq\alpha+1})$ extending $q(\delta)$ such that, for all 
    $x \in X$, we have
    \[
      \Vdash_{\P_\delta} ``q'(\delta)(x) \notin X."
    \]
    (this is possible by Proposition \ref{prop: 54});
    \item for $\gamma \in \dom(q) \setminus (A \cup \{\delta\})$, 
    $q'(\gamma)$ is an arbitrary $\P_\gamma$-name for an element of 
    $\mathrm{Aut}(T_{\leq\alpha+1})$ extending $q(\gamma)$.
  \end{itemize}
  Then $q'$ is as desired, with requirement (4) following from a combination of 
  Proposition \ref{prop: persistent_separation} and the choice of 
  $q'(\delta)$.
\end{proof}

\begin{lemma}[Mitchellized Key Property] \label{lemma: mitchell_key}
  Suppose that
  \begin{itemize}
    \item $(p,q) \in \M$ with $\top(q) = \alpha$;
    \item $0 < n < \omega$ and $\vec{x}$ is an injective $n$-tuple from 
    $T_\alpha$;
    \item $A \in [\dom(q)]^{<\omega}$;
    \item $(p,q)$ is $A$-separated on $\vec{x}$;
    \item $(p',q') \leq (p,q)$ with $\top(q') = \beta$;
    \item $t \in [T_\beta]^{<\omega}$.
  \end{itemize}
  Then there are $p'' \leq_{\P} p'$ and an $n$-tuple $\vec{y}$ from 
  $T_\beta \setminus t$ such that $\vec{y} \in T_{\vec{x}}$ and such that 
  $\vec{x}$ and $\vec{y}$ are $((p'',q'),A)$-consistent.
\end{lemma}

\begin{proof}
  If $\beta = \alpha$, then there is nothing to prove, so assume that $\beta > \alpha$.
  Let $G$ be $\P$-generic over $V$ with $p' \in G$. For all $\gamma \in A$, 
  let $g_\gamma$ be the evaluation of $q'(\gamma)$ in $V[G]$. Then 
  $\mc{G} = \{g_\gamma \mid \gamma \in A\}$ is a finite indexed set of automorphisms 
  of $T_{\leq \beta+1}$ and $\mc{G}$ is separated on 
  $\vec{x}$. By Lemma \ref{lemma: ksProp511}, we can find an $n$-tuple 
  $\vec{y}$ from $T_{\beta} \setminus t$ such that $\vec{y} \in T_{\vec{x}}$  
  and, for all $\gamma \in A$, $\vec{x}$ and $\vec{y}$ are $g_\gamma$-consistent. We can 
  then find $p'' \in G$ such that $p'' \leq_{\P} p'$ and, for all $\gamma \in A$, 
  \[
    p'' \restriction \gamma \Vdash_{\P_\gamma} ``\vec{x} \text{ and } \vec{y} 
    \text{ are } q'(\gamma)\text{-consistent}".
  \]
  Now $p''$ and $\vec{y}$ are as desired.
\end{proof}

\begin{lemma}[Mitchellized 1-Key Property] \label{lemma: mitchell_1_key}
  Suppose that
  \begin{itemize}
    \item $(p,q) \in \M$ with $\top(q) = \alpha$;
    \item $0 < n < \omega$ and $\vec{x}$ is an injective $n$-tuple from $T_\alpha$;
    \item $A \in [\dom(q)]^{<\omega}$;
    \item $(p,q)$ is $A$-separated on $\vec{x}$;
    \item $(p',q') \leq (p,q)$ with $\top(q') = \beta$;
    \item $m < n$, $y^* \in T_\beta$, and $x_m <_T y^*$.
  \end{itemize}
  Then there are $p'' \leq_{\P} p'$ and an $n$-tuple $\vec{y}$ from $T_\beta$ 
  such that $y_m = y^*$ and $\vec{y} \in T_{\vec{x}}$, and such that $\vec{x}$ 
  and $\vec{y}$ are $((p'',q'),A)$-consistent.
\end{lemma}

\begin{proof}
  This is proven in exactly the same way as Lemma \ref{lemma: mitchell_key}, 
  using Lemma \ref{lemma: ksProp512} in place of Lemma \ref{lemma: ksProp511}.
\end{proof}

\begin{lemma} \label{lemma: one_step}
  Suppose that 
  \begin{itemize}
    \item $\delta < \kappa$ and $\alpha < \omega_1$;
    \item $(p_0,q_0), (p_1,q_1) \in \M$ are such that
    \begin{itemize}
      \item $\top(q_0) = \top(q_1) = \alpha$;
      \item $\dom(q_0) = \dom(q_1)$;
      \item $(p_0,q_0) \restriction \delta = (p_1,q_1) \restriction \delta$;
    \end{itemize}
    \item $X \in [T_{\alpha+1}]^{<\omega}$ and $y \in T_{\alpha+1}$ are such that 
    $X \cup \{y\}$ has unique drop-downs to $\alpha$;
    \item $A \in [\dom(q_0)]^{<\omega}$, and $(p_0,q_0)$ and $(p_1,q_1)$ are both $A$-separated on 
    $X \restriction \alpha$.
  \end{itemize}
  Then there exist $(\hat{p}_0, \hat{q}_0), (\hat{p}_1, \hat{q}_1) \in \M$ 
  and $Y \in [T_{\alpha+1}]^{<\omega}$ such that
  \begin{itemize}
    \item $X \cup \{y\} \subseteq Y$;
    \item for each $i < 2$, we have 
    \begin{itemize}
      \item $(\hat{p}_i, \hat{q}_i) \leq (p_i,q_i)$;
      \item $\top(\hat{q}_i) = \alpha + 1$;
      \item $\dom(\hat{q}_i) = \dom(q_i)$;
      \item $(\hat{p}_i,\hat{q}_i)$ is $A$-separated on $Y$;
      \item $X \restriction \alpha$ and $X$ are 
      $((\hat{p}_i, \hat{q}_i),A)$-consistent;
      \item for all $\gamma \in A$, $x \in X \cup \{y\}$, and $j \in \{-1,1\}$,
      \[
        \hat{p}_i \restriction \gamma \Vdash_{\P_\gamma} ``
        \hat{q}_i(\gamma)^j(x) \in Y";
      \]
    \end{itemize}          
    \item $(\hat{p}_0, \hat{q}_0) \restriction \delta = (\hat{p}_1, \hat{q}_1) 
    \restriction \delta$.
  \end{itemize}
\end{lemma}

\begin{proof}
  By extending $p_0$ if necessary (and correspondingly extending $p_1 \restriction \delta$ 
  to maintain the fact that $p_0 \restriction \delta = p_1 \restriction \delta$), we can 
  assume that $p_0$ decides $q_0$ on $(A, \{y \restriction \alpha\})$. Similarly, we can 
  assume that $p_1$ decides $q_1$ on $(A, \{y \restriction \alpha\})$. For each 
  $i < 2$, $\gamma \in A$, $x \in X \cup \{y\}$, and $j \in \{-1,1\}$, let 
  $z^j_{i,\gamma,x} \in T_\alpha$ be such that
  \[
    p_i \restriction \gamma \Vdash_{\P_\gamma} ``q_i(\gamma)^j(x \restriction \alpha) = 
    z^j_{i,\gamma,x}".
  \]
  Note that, if $\gamma \in A \cap \delta$, $x \in X \cup \{y\}$, and $j \in \{-1,1\}$, 
  then $z^j_{0,\gamma,x} = z^j_{1,\gamma,x}$. Let $W_0$ be the set of all 
  $(i,\gamma,x,j) \in 2 \times A \times X \times \{-1,1\}$ such that $z^j_{i,\gamma,x} 
  \in X \restriction \alpha$, and let $W_1 = (2 \times A \times (X \cup \{y\}) \times 
  \{-1,1\}) \setminus W_0$. For each $(i,\gamma,x,j) \in W_1$, choose 
  $\hat{z}^j_{i,\gamma,x} \in T_{\alpha+1}$ in such a way that
  \begin{itemize}
    \item $\hat{z}^j_{i,\gamma,x} \restriction \alpha = z^j_{i,\gamma,x}$;
    \item $\hat{z}^j_{i,\gamma,x} \notin X \cup \{y\}$;
    \item if $(0,\gamma,x,j) \in W_1$ and $\gamma < \delta$, then 
    $\hat{z}^j_{0,\gamma,x} = \hat{z}^j_{1,\gamma,x}$ (note that the hypothesis implies 
    that $(1,\gamma,x,j) \in W_1$ as well);
    \item if $(i,\gamma,x,j)$ and $(i',\gamma',x',j')$ are distinct elements of $W_1$ such 
    that $(\gamma,x,j) \neq (\gamma',x',j')$ or $\gamma \geq \delta$, 
    (i.e., the pair does not fall into the scope of the previous bullet point), then 
    $\hat{z}^j_{i,\gamma,x} \neq \hat{z}^{j'}_{i',\gamma',x'}$.
  \end{itemize}
  Let $Y = X \cup \{y\} \cup \{\hat{z}^j_{i,\gamma,x} \mid (i,\gamma,x,j) \in W_1\}$. For 
  each $\gamma \in A$ and $i < 2$, let $\hat{q}_i(\gamma)$ be a $\P_\gamma$-name for an  
  element of $\mathrm{Aut}(T_{\leq \alpha+1})$ that is forced by $p_i \restriction \gamma$ 
  to have the following properties, with the additional requirement that 
  $\hat{q}_0(\gamma) = \hat{q}_1(\gamma)$ if $\gamma \in A \cap \delta$:
  \begin{itemize}
    \item $\hat{q}_i(\gamma) \supseteq q_i(\gamma)$;
    \item $X \restriction \alpha$ and $X$ are $\hat{q}_i(\gamma)$-consistent;
    \item for all $(x,j) \in (X \cup \{y\}) \times \{-1,1\}$ such that $(i,\gamma,x,j) \in W_1$, 
    we have
    \begin{itemize}
      \item $\hat{q}_i(\gamma)^j(x) = \hat{z}^j_{i,\gamma,x}$;
      \item $\hat{q}_i(\gamma)^j(\hat{z}^j_{i,\gamma,x}) \notin Y$;
    \end{itemize}
    \item for all $(i',\gamma',x',j') \in W_1$ such that $(i',\gamma') \neq (i,\gamma)$ 
    and it is not the case that $\gamma' = \gamma \in A \cap \delta$ (i.e., 
    $\hat{z}^{j'}_{i',\gamma',x'}$ does not fall within the scope of the previous bullet point), 
    and for all $j \in \{-1,1\}$, we have $\hat{q}_i(\gamma)^j(\hat{z}^{j'}_{i',\gamma',x'}) 
    \notin Y$.
  \end{itemize}
  It is possible to find such a $\hat{q}_i(\gamma)$ by Proposition \ref{prop: ext_prop}. For all 
  $i < 2$ and $\gamma \in \dom(q_i) \setminus A$, let $\hat{q}_i(\gamma)$ be an arbitrary 
  $\P_\gamma$-name for an element of $\mathrm{Aut}(T_{\leq \alpha+1})$ extending 
  $q_i(\gamma)$, again subject to the requirement that $\hat{q}_0(\gamma) = \hat{q}_1(\gamma)$ 
  for all $\gamma \in \dom(q_i) \cap \delta$. Finally, for each $i < 2$, find $\hat{p}_i \leq_{\P} 
  p_i$ such that
  \begin{itemize}
    \item $\hat{p}_i$ decides $\hat{q}_i$ on $(A,Y)$;
    \item $\hat{p}_0 \restriction \delta = \hat{p}_1 \restriction \delta$.
  \end{itemize}
  Then $(\hat{p}_0,\hat{q}_0)$,$(\hat{p}_1,\hat{q}_1)$, and $Y$ are as desired. 
  For instance, to see that $(\hat{p}_0, \hat{q}_0)$ is $A$-separated on 
  $Y$, first let $\vec{x}$ be an injective enumeration of $X$ such that 
  $\vec{x} \restriction \alpha$ witnesses that $(p_0,q_0)$ is 
  $A$-separated on $X \restriction \alpha$, let $\vec{z}$ be an arbitrary injective 
  enumeration of $Y \setminus (X \cup \{y\})$, and let $\vec{y} = 
  \vec{x} \ {}^\frown \langle y \rangle ^\frown \vec{z}$. The construction and choice 
  of $\vec{x}$ then easily yield the conclusion that $(\hat{p}_0, \hat{q}_0)$ 
  is $A$-separated on $\vec{y}$. 
\end{proof}

\subsection{Dense subsets of derived trees}
This subsection contains two crucial lemmas proving that certain subsets of derived trees
of $T$ are dense. The first plays a key role in the eventual proof that $\M$ and its quotients have the $\omega_1$-approximation property, which 
will later lead to the fact that, if $\kappa$ is supercompact, then $\GMP$ holds in $V[\M]$. The second plays a similar role in the proof that, if $T$ is a free 
Suslin tree, then it remains a Suslin tree in $V[\M]$.

\begin{lemma} \label{lemma: dense}
  Suppose that
  \begin{itemize}
    \item $(p,q) \in \M$ and $\alpha = \top(q)$;
    \item $\delta < \kappa$;
    \item $A$ is a finite subset of $\dom(q)$, $\vec{x}$ is a finite sequence from 
    $T_\alpha$, and $(p,q)$ is $A$-separated on $\vec{x}$;
    \item $\lambda$ is a cardinal and $\dot{b}$ is an $\M$-name for a subset of 
    $\lambda$ such that $\Vdash_{\M} \dot{b} \notin V[\M_\delta]$.
  \end{itemize}
  Let $D$ be the set of all $\vec{z} \in T_{\vec{x}}$ for which there exist 
  $(p_0,q_0),(p_1,q_1) \in \M$ and $\varepsilon < \lambda$ such that
  \begin{itemize}
    \item for $i < 2$, we have:
    \begin{itemize}
      \item $(p_i,q_i) \leq (p,q)$;
      \item $\top(q_i) = \height_T(\vec{z})$;
      \item $\vec{x}$ and $\vec{z}$ are $((p_i,q_i),A)$-consistent;
    \end{itemize}
    \item $(p_0,q_0) \restriction \delta = (p_1,q_1) \restriction \delta$;
    \item $(p_0,q_0) \Vdash ``\varepsilon \notin \dot{b}"$ and $(p_1,q_1) 
    \Vdash ``\varepsilon \in \dot{b}"$.
  \end{itemize}
  Then $D$ is a dense open subset of $T_{\vec{x}}$.
\end{lemma}

\begin{proof}
  The fact that $D$ is open follows easily from its definition. To show that 
  $D$ is dense, fix $\vec{y} \in T_{\vec{x}}$, and let $\alpha' = \height_T(\vec{y})$. 
  We will find $\vec{z} \in D \cap T_{\vec{y}}$.
  By repeatedly applying Extension (Lemma \ref{lemma: ksProp515}), we can find 
  $q'$ such that
  \begin{itemize}
    \item $(p,q') \leq_{\M} (p,q)$;
    \item $\top(q') = \alpha'$;
    \item $\vec{x}$ and $\vec{y}$ are $((p,q'),A)$-consistent.
  \end{itemize}
  Since $\Vdash_{\M} ``\dot{b} \notin V[\M_\delta]"$, we can find 
  $(p'_0, q'_0), (p'_1,q'_1) \leq_{\M} (p,q')$ and $\varepsilon < \lambda$ such that
  \begin{enumerate}
    \item $(p'_0,q'_0) \restriction \delta = (p'_1,q'_1) \restriction \delta$;
    \item $(p'_0,q'_0) \Vdash_{\M} ``\varepsilon \notin \dot{b}"$;
    \item $(p'_1,q'_1) \Vdash_{\M} ``\varepsilon \in \dot{b}"$.
  \end{enumerate}
  By extending the conditions if necessary, we may assume that there is 
  a countable ordinal $\beta > \alpha'$ such that 
  $\top(q'_0) = \top(q'_1) = \beta$. By extending 
  $p'_0$ if necessary (and also correspondingly extending 
  $p'_1 \restriction \delta$ to ensure item (1) above) and applying the 
  Mitchellized Key Property (Lemma \ref{lemma: mitchell_key}), we can fix 
  $\vec{y}_0 \in T_{\vec{y}}$ such that 
  $\height_T(\vec{y}_0) = \beta$ and such that $\vec{y}$ and $\vec{y}_0$ are 
  $((p'_0,q'_0),A)$-consistent. 
  Let $\vec{y}_{00} = \vec{y}_0 \restriction (\alpha' + 1)$. 
  By extending $p'_1$ (and $p'_0 \restriction \delta$ if necessary) and 
  applying the Mitchellized Key Property, we can fix $\vec{y}_{10} \in T_{\vec{y}}$ 
  such that
  \begin{itemize}
    \item $\height_T(\vec{y}_{10}) = \alpha' + 1$;
    \item $\vec{y}_{00}$ and $\vec{y}_{10}$ have no entries in common;
    \item $\vec{y}$ and $\vec{y}_{10}$ are $((p'_1,q'_1),A)$-consistent.
  \end{itemize}
  By the same reasoning as above, we can fix $\vec{y}_1 \in T_{\vec{y}_{10}}$ 
  such that $\height_T(\vec{y}_1) = \beta$ and such that $\vec{y}_{10}$ 
  and $\vec{y}_1$ are $((p'_1,q'_1),A)$-consistent.
  
  Now repeatedly apply Lemma \ref{lemma: ksLem539} followed by Extension
  (Lemma \ref{lemma: ksProp515}) to find $(p'_2,q'_2) \leq_\M (p,q')$ such that
  \begin{itemize}
    \item $\height(q'_2) = \beta$;
    \item $(p'_2, q'_2) \restriction \delta = (p'_0,q'_0) \restriction \delta$;
    \item for each $i < 2$, $\vec{y}$ and $\vec{y}_i$ are 
    $((p'_2,q'_2),A)$-consistent.
  \end{itemize}
  Now find $(p_2,q_2) \leq_{\M} (p'_2,q'_2)$ deciding the truth value of ``$\varepsilon \in 
  \dot{b}$" For concreteness, suppose that $(p_2,q_2) \Vdash_{\M} ``\varepsilon \in 
  \dot{b}"$; the other case is symmetric, with $(p'_1,q'_1)$ playing the role 
  of $(p'_0,q'_0)$ in the following argument. Let $\beta^* = \top(q_2)$. 
  By applying the Mitchellized Key Property and extending $p_2$ if necessary, 
  we can fix $\vec{z} \in T_{\vec{y}_0}$ 
  such that $\height_T(\vec{z}) = \beta^*$ and such that $\vec{y}_0$ and $\vec{z}$ 
  are $((p_2,q_2),A)$-consistent. Finally, repeatedly apply Extension to find $(p_0,q_0) \leq_{\M} (p'_0,q'_0)$ such that 
  \begin{itemize}
    \item $\top(q_0) = \beta^*$;
    \item $(p_0,q_0) \restriction \delta = (p_2,q_2) \restriction \delta$;
    \item $\vec{y}_0$ and $\vec{z}$ are $((p_0,q_0),A)$-consistent.
  \end{itemize}
  Then $\vec{z} \in D$, as witnessed by $(p_0,q_0)$, $(p_2,q_2)$, and $\varepsilon$.
\end{proof}

\begin{lemma} \label{lemma: dense_1}
  Suppose that
  \begin{itemize}
    \item $(p,q) \in \M$ and $\alpha = \top(q)$;
    \item $A$ is a finite subset of $\dom(q)$, $\vec{x}$ is an injective 
    $n$-tuple from $T_\alpha$ for some $0 < n < \omega$, and 
    $(p,q)$ is $A$-separated on $\vec{x}$;
    \item $m < n$ and $\dot{E}$ is an $\M$-name for a dense open subset of $T$.
  \end{itemize}
  Let $D$ be the set of all $\vec{z} \in T_{\vec{x}}$ for which there exists 
  $(p^*,q^*) \leq_{\M} (p,q)$ such that
  \begin{itemize}
    \item $\top(q^*) = \height_T(\vec{z})$;
    \item $(p^*,q^*) \Vdash_{\M} ``z_m \in \dot{E}"$;
    \item $\vec{x}$ and $\vec{z}$ are $((p^*,q^*),A)$-consistent.
  \end{itemize}
  Then $D$ is a dense open subset of $T_{\vec{x}}$.
\end{lemma}

\begin{proof}
  The fact that $D$ is open follows easily from the definition. To show that $D$ 
  is dense, fix $\vec{y} \in T_{\vec{x}}$, and let $\alpha' = \height_T(\vec{y})$. 
  Fix $q'$ exactly as in the proof of Lemma \ref{lemma: dense}. Find $p' 
  \leq_{\P} p$ such that $p'$ decides $q'$ on $(A, \vec{y})$. Note that, 
  since $(p,q)$ is $A$-separated on $\vec{x}$, it follows from Mitchellized 
  Persistence (Proposition \ref{prop: persistent_separation}) that $(p',q')$ is $A$-separated on 
  $\vec{y}$. Now find 
  $(p'',q'') \leq_{\M} (p',q')$ and $t \in T$ such that $y_m \leq_T t$ and 
  $(p'',q'') \Vdash_{\M} ``t \in \dot{E}"$. By strengthening $q''$ if 
  necessary, we may 
  assume that $\top(q'') \geq \height(t)$, and since $\dot{E}$ is forced 
  to be open, we can extend $t$ to assume that we in fact have 
  $\top(q'') = \height(t)$; let $\beta$ denote this height. 
  Now apply the Mitchellized 1-Key Property (Lemma \ref{lemma: mitchell_1_key}) 
  to find $p^* \leq_{\P} p''$ and 
  an $n$-tuple $\vec{z}$ from $T_\beta$ such that $z_m = t$ and
  $\vec{z} \in T_{\vec{y}}$, and such that $\vec{y}$ and $\vec{z}$ are 
  $((p^*,q''),A)$-consistent. Then $\vec{z} \in D$, as witnessed 
  by $(p^*, q'')$.
\end{proof}

\subsection{Approximation and the proof of Theorem A}
This subsection finally contains the proof of Theorem A. We first recall what 
it means for a forcing poset to satisfy the \emph{$\mu$-approximation property}.

\begin{definition}
  Let $\mu$ be an uncountable cardinal, and let $V \subseteq W$ be two models of 
  $\ZFC$. We say that $(V,W)$ has the \emph{$\mu$-approximation property} 
  if, whenever $x \in W$ is a set of ordinals such that $x \cap z \in V$ for all 
  $z \in ([\mathrm{On}]^{<\mu})^V$, it follows that $x \in V$. Working in a model 
  $V$ of $\ZFC$, we say that a forcing poset $\bb{R}$ has the 
  $\mu$-approximation property if it is forced by $1_{\mathbb{R}}$ that 
  $(V,V[\bb{R}])$ has the $\mu$-approximation property.
\end{definition}

We begin with the following 
preliminary result, establishing the crucial properties of the forcing extension by $\M$.
In what follows, we let $\dot{G}$ be the canonical $\M$-name for the generic filter. 
For all $\delta < \kappa$, let $\dot{G}_\delta$ be the canonical name for the generic 
filter over $\M_\delta$ induced by $\dot{G}$. Let $\dot{\M}_{\delta,\kappa}$ 
be the canonical $\M_\delta$-name for the quotient forcing $\M/\dot{G}_\delta$.
\begin{theorem} \label{thm: approx}
  Suppose that $T$ is a free Suslin tree in $V$. Then
  \begin{enumerate}
    \item for all limit ordinals $\delta < \kappa$ (including $\delta = 0$),
    \[
      \Vdash_{\M_\delta} ``\dot{\M}_{\delta,\kappa} \text{ has the } 
      \omega_1\text{-approximation property}";
    \]
    \item $\Vdash_{\M} ``T \emph{ is Suslin}"$.
  \end{enumerate}
\end{theorem}

\begin{proof}
  We will work towards proving both statements simultaneously.
  Fix a condition $(p,q) \in \M$, a limit ordinal $\delta < \kappa$, 
  a cardinal $\lambda$, an $\M$-name $\dot{b}$ such that 
  \[
    (p,q) \Vdash_{\M} ``\dot{b} \subseteq \lambda \wedge \dot{b} \notin 
    V[\dot{G}_\delta]",
  \]
  and an $\M$-name $\dot{E}$ such that 
  \[
    (p,q) \Vdash_{\M} ``\dot{E} \text{ is a dense open subset of } T".
  \]
  For notational simplicity, assume without loss of generality that $p = 1_{\P}$ 
  (for the general case, simply work below $p$ in $\P$).
  Fix a sufficiently large regular cardinal $\theta$, a well-ordering $\vartriangleleft$
  of $H(\theta)$, and a countable elementary submodel $N \prec (H(\theta), \in, 
  \vartriangleleft)$ containing everything relevant (including $q$, $\dot{b}$, 
  and $\dot{E}$). Let $\beta = N \cap \omega_1$. 
  We will find $q^* \leq_{\Q} q$ such that
  \begin{itemize}
    \item $(1_{\P},q^*) \Vdash_{\M} ``\dot{b} \cap N \notin V[\dot{G}_\delta]"$;
    \item $(1_{\P},q^*) \Vdash_{\M} ``T_\beta \subseteq \dot{E}"$.
  \end{itemize}
  This will clearly suffice to prove the theorem.
  
  Let $\P_N$ denote $\P \cap N$. Let $\langle x_n \mid n < \omega \rangle$ 
  enumerate $T_\beta$, let $\langle \gamma_n \mid n < \omega \rangle$ enumerate 
  $\kappa \cap N$, let $\langle \beta_n \mid n < \omega \rangle$ be an increasing 
  sequence of ordinals cofinal in $\beta$, and let $\langle p_n \mid n < \omega \rangle$ 
  enumerate $\P_N$ so that every element of $\P_N$ equals $p_n$ for infinitely many 
  $n < \omega$. For all $n < \omega$, let $A_n = \{\gamma_m \mid m \leq n\}$.
  
  We will now recursively build objects $\langle q_n, \hat{p}_{n,0}, \hat{p}_{n,1}, X_n, 
  \varepsilon_n \mid n < \omega \rangle$ satisfying the following requirements:
  \begin{enumerate}
    \item $\langle q_n \mid n < \omega \rangle$ is a $\leq_{\Q}$-decreasing sequence 
    of conditions from $\Q \cap N$ with $q_0 \leq_{\Q} q$;
    \item for all $n < \omega$, we have $A_n \subseteq \dom(q_n)$ and 
    $\top(q_n) \geq \beta_n$;
    \item for all $n < \omega$, $\hat{p}_{n,0}$ and $\hat{p}_{n,1}$ are elements of $\P_N$ 
    such that $\hat{p}_{n,0},\hat{p}_{n,1} \leq_{\P} p_n$ and $\hat{p}_{n,0} \restriction \delta 
    = \hat{p}_{n,1} \restriction \delta$;
    \item for all $m < n < \omega$ and $i < 2$, 
    \[
      (\hat{p}_{n,i},q_n) \Vdash_{\M} ``x_m \restriction \top(q_n) \in \dot{E}";
    \]
    \item for all $n < \omega$, $\varepsilon_n$ is an element of $N \cap \lambda$ 
    such that $(\hat{p}_{n,0},q_n) \Vdash_{\M} ``\varepsilon_n \notin \dot{b}"$ and 
    $(\hat{p}_{n,1},q_n) \Vdash_{\M} ``\varepsilon_n \in \dot{b}"$;
    \item for all $n < \omega$, $X_n$ is a finite subset of $T_\beta$ such that
    \begin{enumerate}
      \item $\{x_m \mid m \leq n\} \subseteq X_n \subseteq X_{n+1}$;
      \item $X_n$ has unique drop-downs to $\top(q_n)$;
      \item for each $i < 2$, $\hat{p}_{n,i}$ decides $q_n$ on $(A_n, X_n 
      \restriction \top(q_n))$;
      \item for each $i < 2$, $\gamma \in A_n$, $m \leq n$, and $j \in \{-1,1\}$, 
      there is $y \in X_n$ such that
      \[
        \hat{p}_{n,i} \Vdash_{\P} ``q_n(\gamma)^j(x_m \restriction \top(q_n)) 
        = y \restriction \top(q_n)";
      \]
    \end{enumerate}
    \item for all $n < \omega$, $q_n$ is $A_n$-separated on $X_n \restriction 
    \top(q_n)$;
    \item for all $n < \omega$, $X_n \restriction \top(q_n) \text{ and }
    X_n \restriction \top(q_{n+1})$ are $(q_{n+1},A_n)$-consistent.
  \end{enumerate}
  For ease of notation, let $q_{-1} = q$ and $X_{-1} = \emptyset$. Fix $-1 \leq n < \omega$ and 
  suppose that we have constructed $q_n$ and $X_n$. First, apply Lemma \ref{lemma: separated_extension} 
  inside $N$ to find $q^*_{n+1} \leq_{\Q} q_n$ such that
  \begin{itemize}
    \item $q^*_{n+1} \in N$;
    \item $\top(q^*_{n+1}) = \top(q_n) + 1$;
    \item $A_{n+1} \subseteq \dom(q^*_{n+1})$;
    \item $q^*_{n+1}$ is $A_{n+1}$-separated on $X_n \restriction \top(q^*_{n+1})$;
    \item for all $\gamma \in A_n$, $X_n \restriction \top(q_n)$ and 
    $X_n \restriction \top(q^*_{n+1})$ are $(q^*_{n+1},A_n)$-consistent.
  \end{itemize}
  Fix $p_{n+1}^* \leq_{\P} p_{n+1}$ with $p_{n+1}^* \in N$ 
  and an injective enumeration $\vec{x^n}$ of $X_n$ such that $(p_{n+1}^*, q^*_{n+1})$ is 
  $A_{n+1}$-separated on $\vec{x^n} \restriction \top(q^*_{n+1})$.
  
  \begin{claim} \label{claim: dense_claim}
     Fix $0 \leq m \leq n$ and $(\bar{p},\bar{q}) \leq_{\M} (p^*_{n+1}, q^*_{n+1})$ such 
     that $(\bar{p},\bar{q}) \in N$, $\bar{p}$ decides $\bar{q}$ on $(A_{n+1}, \vec{x^n} 
     \restriction \top(\bar{q}))$, and such that $\vec{x^n} \restriction 
     \top(q^*_{n+1})$ and $\vec{x^n} \restriction \top(\bar{q})$ are 
     $((\bar{p},\bar{q}),A_{n+1})$-consistent. Then there is $(\bar{\bar{p}}, 
     \bar{\bar{q}}) \leq_{\M} (\bar{p}, \bar{q})$ such that
     \begin{itemize}
       \item $(\bar{\bar{p}}, \bar{\bar{q}}) \in N$;
       \item $(\bar{\bar{p}}, \bar{\bar{q}}) \Vdash_{\M} ``x_m \restriction \top(\bar{\bar{q}}) 
       \in \dot{E}"$;
       \item $\vec{x} \restriction \top(\bar{q})$ and $\vec{x} \restriction 
       \top(\bar{\bar{q}})$ are $((\bar{\bar{p}}, \bar{\bar{q}}),A_{n+1})$-consistent.
     \end{itemize} 
  \end{claim}
  
  \begin{proof}
    Fix $m^* < \omega$ such that $x_m  = x^n_{m^*}$. Let $D$ be the set of 
    all $\vec{z} \in T_{\vec{x^n} \restriction \top(\bar{q})}$ for which there exists 
    $(\bar{\bar{p}}, \bar{\bar{q}}) \leq_{\M} (\bar{p}, \bar{q})$ such that 
    \begin{itemize}
      \item $\top(\bar{\bar{q}}) = \height_T(\vec{z})$;
      \item $(\bar{\bar{p}}, \bar{\bar{q}}) \Vdash_{\M} ``z_{m^*} \in \dot{E}"$;
      \item $\vec{x} \restriction \top(\bar{q})$ and $\vec{z}$ are 
      $((\bar{\bar{p}}, \bar{\bar{q}}),A_{n+1})$-consistent.
    \end{itemize}
    By Lemma \ref{lemma: dense_1}, $D$ is a dense open subset of 
    $T_{\vec{x^n} \restriction \height(\bar{q})}$. Since $T$ is a free Suslin tree, there is 
    $\eta < \omega_1$ such that every $\vec{z} \in T_{\vec{x^n} \restriction \height(\bar{q})}$ of 
    height at least $\eta$ is in $D$. Since $D \in N$, the least such $\eta$ is in $N$, and hence 
    is less than $\beta$. In particular, there is some $\eta^* \in [\top(\bar{q}), \beta)$ such that 
    $\vec{x^n} \restriction \eta^* \in D$. Then any condition $(\bar{\bar{p}}, \bar{\bar{q}}) \in 
    N$ witnessing that $\vec{x^n} \restriction \eta^* \in D$ witnesses the conclusion of the claim.
  \end{proof}
  By repeatedly applying Claim \ref{claim: dense_claim}, fix a condition $(\bar{p}_{n+1}, 
  \bar{q}_{n+1}) \leq_{\M} (p^*_{n+1}, q^*_{n+1})$ such that 
  \begin{itemize}
    \item $(\bar{p}_{n+1}, \bar{q}_{n+1}) \in N$;
    \item $\bar{p}_{n+1}$ decides $\bar{q}_{n+1}$ on $(A_{n+1}, \vec{x^n} \restriction 
    \top(\bar{q}_{n+1}))$;
    \item $\vec{x^n} \restriction \top(q^*_{n+1})$ and $\vec{x^n} 
    \restriction \top(\bar{q}_{n+1})$ are $((\bar{p}_{n+1}, \bar{q}_{n+1}),
    A_{n+1})$-consistent;
    \item for all $m \leq n$, we have $(\bar{p}_{n+1}, \bar{q}_{n+1}) \Vdash_{\M} 
  ``x_m \restriction \top(\bar{q}_{n+1}) \in \dot{E}"$.
  \end{itemize}
  
  Now let $D'$ be the set of all $\vec{z} \in T_{\vec{x^n} \restriction \top(\bar{q}_{n+1})}$ for which 
  there exist $(\bar{p}_{n+1,0}, \bar{q}_{n+1,0})$ and $(\bar{p}_{n+1,1}, \bar{q}_{n+1,1})$ 
  in $\M$ and $\varepsilon < \lambda$ such that
  \begin{itemize}
    \item for each $i < 2$, we have $(\bar{p}_{n+1,i}, \bar{q}_{n+1,i}) \leq_{\M} 
    (\bar{p}_{n+1}, \bar{q}_{n+1})$;
    \item $\top(\bar{q}_{n+1,0}) = \top(\bar{q}_{n+1,1}) = \height_T(\vec{z})$;
    \item $(\bar{p}_{n+1,0},\bar{q}_{n+1,0}) \restriction \delta = (\bar{p}_{n+1,1}, 
    \bar{q}_{n+1,1}) \restriction \delta$;
    \item for each $i < 2$, $\vec{x^n} \restriction \top(\bar{q}_{n+1})$ and 
    $\vec{z}$ are $((\bar{p}_{n+1,i}, \bar{q}_{n+1,i}),A_{n+1})$-consistent;
    \item $(\bar{p}_{n+1,0}, \bar{q}_{n+1,0}) \Vdash_{\M} ``\varepsilon \notin \dot{b}"$ and 
    $(\bar{p}_{n+1,1}, \bar{q}_{n+1,1}) \Vdash_{\M} ``\varepsilon \in \dot{b}"$.
  \end{itemize}
  By Lemma \ref{lemma: dense}, $D'$ is a dense open subset of $T_{\vec{x^n} \restriction 
  \top(\bar{q}_{n+1})}$, and $D' \in N$. Therefore, as in the proof of Claim 
  \ref{claim: dense_claim}, we can find $\alpha_{n+1} < \beta$ such that
  \begin{itemize}
    \item $\alpha_{n+1} \geq \max\{\beta_{n+1}, \top(\bar{q}_{n+1})\}$;
    \item $X_n \cup \{x_{n+1}\}$ has unique drop-downs to $\alpha_{n+1}$;
    \item $\vec{x^n} \restriction \alpha_{n+1} \in D'$.
  \end{itemize}
  Since $D'$ and $\vec{x^n} \restriction \alpha_{n+1}$ are in $N$, we can fix 
  $(\bar{p}_{n+1,0}, \bar{q}_{n+1,0})$, $(\bar{p}_{n+1,1}, \bar{q}_{n+1,1})$, and 
  $\varepsilon_{n+1}$ in $N$ witnessing that $\vec{x^n} \restriction \alpha_{n+1}$ 
  is in $D'$. By extending the conditions if necessary, we can assume that
  \begin{itemize}
    \item $\dom(\bar{q}_{n+1,0}) = \dom(\bar{q}_{n+1,1})$;
    \item $\delta \in \dom(\bar{p}_{n+1,0}) \cap \dom(\bar{p}_{n+1,1})$ and 
    $\bar{p}_{n+1,0}(\delta) \perp \bar{p}_{n+1,1}(\delta)$;
    \item for each $i < 2$, $\bar{p}_{n+1, i}$ decides $\bar{q}_{n+1,i}$ on 
    $(A_{n+1}, X_n \cup \{x_{n+1}\})$.
  \end{itemize}
  By Lemma \ref{lemma: one_step} applied in $N$, we can find $(\hat{p}_{n+1,0}, 
  \hat{q}_{n+1,0}),(\hat{p}_{n+1,1},\hat{q}_{n+1,1}) \in \M$ and a finite set 
  $Y$, all in $N$, such that, letting $\alpha^*_{n+1} = \alpha_{n+1}+1$, 
  the following hold:
  \begin{itemize}
    \item for each $i < 2$, we have $\dom(\hat{q}_{n+1,i}) = 
    \dom(\bar{q}_{n+1,i})$ and $(\hat{p}_{n+1,i},\hat{q}_{n+1,i}) \leq_{\M} 
    (\bar{p}_{n+1,i}, \bar{q}_{n+1,i})$;
    \item for each $i < 2$, we have $\top(\hat{q}_{n+1,i}) = \alpha^*_{n+1}$;
    \item $(\hat{p}_{n+1,0},\hat{q}_{n+1,0}) \restriction \delta = 
    (\hat{p}_{n+1,1},\hat{q}_{n+1,1}) \restriction \delta$;
    \item $(X_n \cup \{x_{n+1}\}) \restriction \alpha^*_{n+1} \subseteq Y \subseteq 
    T_{\alpha^*_{n+1}}$;
    \item for each $i < 2$, $(\hat{p}_{n+1,i}, \hat{q}_{n+1,i})$ is 
    $A_{n+1}$-separated on $Y$;
    \item for each $i < 2$, $X_n \restriction \alpha_{n+1}$ and $X_n \restriction 
    \alpha^*_{n+1}$ are $(\hat{p}_{n+1,i}, \hat{q}_{n+1,i})$-consistent on $A_{n+1}$;
    \item for all $i < 2$, $\gamma \in A_{n+1}$, $m \leq n+1$, and $j \in \{-1,1\}$, 
    there is $y \in Y$ such that 
    \[
      \hat{p}_{n+1,i} \restriction \gamma \Vdash_{\P_\gamma} 
      ``\hat{q}_{n+1,i}(\gamma)^j(x_m \restriction \alpha^*_{n+1}) = y".
    \]
  \end{itemize}
  For each $y \in Y \setminus \left( (X_n \cup \{x_{n+1}\}) \restriction \alpha^*_{n+1} 
  \right)$, choose a $z_y \in T_\beta$ such that $z_y \restriction \alpha^*_{n+1} = y$, 
  and set
  \[
    X_{n+1} = X_n \cup \{x_{n+1}\} \cup \left\{ z_y \ \middle| \ y \in 
    Y \setminus \left( (X_n \cup \{x_{n+1}\}) \restriction \alpha^*_{n+1} 
    \right) \right\}.
  \]
  
  Next, working in $N$, fix $r \in \Q$ with $\top(r) = \alpha^*_{n+1}$ and 
  $\dom(r) = \dom(\hat{q}_{n+1,0})$ such that $r \leq_{\Q} q^*_{n+1}$, $r$ is 
  $A_{n+1}$-separated on $X_{n+1} \restriction \alpha^*_{n+1}$, and such 
  that $X_n \restriction \top(q^*_{n+1})$ and $X_n \restriction 
  \alpha^*_{n+1}$ are $(r,A_n)$-consistent.
  To see that such an $r$ can be found, first repeatedly apply Proposition 
  \ref{prop: 54} to find $r_0 \leq_{\Q} q^*_{n+1}$ such that 
  $\top(r_0) = \alpha_{n+1}$, $\dom(r_0) = \dom(\hat{q}_{n+1,0})$ and 
  such that $X_n \restriction \top(q^*_{n+1})$ and $X_n \restriction 
  \alpha_{n+1}$ are $(r_0,A_{n+1})$-consistent.
  Then again apply Proposition \ref{prop: 54} 
  to find $r \leq_{\Q} r_0$ such that $\top(r) = \alpha^*_{n+1}$, 
  $\dom(r) = \dom(r_0)$ and, for all $\gamma \in A_{n+1}$, we have:
  \begin{itemize}
    \item $\Vdash_{\P_\gamma}``X_n \restriction \alpha_{n+1} \text{ and } 
    X_n \restriction \alpha^*_{n+1} \text{ are } r(\gamma)\text{-consistent}"$;
    \item for all $y \in (X_{n+1} \setminus X_n) \restriction \alpha^*_{n+1}$ 
    and all $j \in \{-1,1\}$, we have $\Vdash_{\P_\gamma}``r(\gamma)^j(y) \notin 
    X_{n+1} \restriction \alpha^*_{n+1}$.
  \end{itemize}
  Then, recalling that $q^*_{n+1}$ was chosen to be $A_{n+1}$-separated 
  on $X_n \restriction \top(q^*_{n+1})$, it is readily seen that 
  this $r$ is as desired.
  
  Finally, define $q_{n+1} \in \Q$ with 
  $\top(q_{n+1}) = \alpha^*_{n+1}$ and $\dom(q_{n+1}) = \dom(\hat{q}_{n+1,0})$ by letting, for 
  all $\gamma \in \dom(\hat{q}_{n+1,0})$, $q_{n+1}(\gamma)$ be a $\P_\gamma$-name such that
  \begin{itemize}
    \item $\hat{p}_{n+1,0} \restriction \gamma \Vdash ``q_{n+1}(\gamma) = \hat{q}_{n+1,0}(\gamma)"$;
    \item $\hat{p}_{n+1,1} \restriction \gamma \Vdash ``q_{n+1}(\gamma) = \hat{q}_{n+1,1}(\gamma)$ 
    (note that this is not in conflict with the previous point since, if 
    $\hat{q}_{n+1,0}(\gamma) \neq \hat{q}_{n+1,1}(\gamma)$, then $\gamma > \delta$, and hence 
    $\hat{p}_{n+1,0} \restriction \gamma \perp \hat{p}_{n+1,1} \restriction \gamma$);
    \item if $p^* \in \P_\gamma$ is incompatible with both 
    $\hat{p}_{n+1,0} \restriction \gamma$ and $\hat{p}_{n+1,1} \restriction \gamma$, then 
    $p^* \Vdash ``q_{n+1}(\gamma) = r(\gamma)"$.
  \end{itemize}
  It is readily verified that $q_{n+1}$, $\hat{p}_{n+1,0}$, $\hat{p}_{n+1,1}$, $X_{n+1}$, and 
  $\varepsilon_{n+1}$ satisfy the requirements of the construction.
  
  \begin{claim}
    $\langle q_n \mid n < \omega \rangle$ has a lower bound in $\Q$.
  \end{claim}
  
  \begin{proof}
    Recall that $\alpha^*_n = \top(q_n)$ for all $n < \omega$, and that $\langle 
    \alpha^*_n \mid n < \omega \rangle$ is cofinal in $\beta$. Note also that 
    $\bigcup \{\dom(q_n) \mid n < \omega\} = N \cap \kappa$. For each $\gamma \in N \cap 
    \kappa$, let $n_\gamma$ be the least $n < \omega$ such that $\gamma \in \dom(q_n)$. 
    For each $\gamma \in N \cap \kappa$, let $q^-(\gamma)$ be a $\P_\gamma$-name forced 
    to be equal to $\bigcup \{q_n(\gamma) \mid n_\gamma \leq n < \omega\}$. In particular, 
    $q^-(\gamma)$ is forced to be an automorphism of $T_{<\beta}$.
    
    \begin{subclaim} \label{subclaim: extension}
      For each $\gamma \in N \cap \kappa$, each $x \in T_\beta$, and each $j \in \{-1,1\}$, 
      \[
        \Vdash_{\P_\gamma} ``\exists y \in T_\beta \forall \alpha < \beta \
        [q^-(\gamma)^j(x \restriction \alpha) = y \restriction \alpha]".
      \]
    \end{subclaim}
    
    \begin{proof}
      Suppose for the sake of contradiction that there exists $\gamma \in N \cap \kappa$, 
      $n^* < \omega$, $j \in \{-1,1\}$, and $p^* \in \P_\gamma$ such that
      \[
        p^* \Vdash_{\P_\gamma} ``\neg (\exists y \in T_\beta \forall \alpha < \beta \
        [q^-(\gamma)^j(x_{n^*} \restriction \alpha) = y \restriction \alpha])".
      \]
      Find $n < \omega$ such that $p^* \cap N = p_n$, $n \geq n^*$, and $\gamma \in 
      A_n$. Recall, then, that there is $y \in X_n$ such that 
      $\hat{p}_{n,0} \restriction \gamma \Vdash_{\P_\gamma} ``q_n(\gamma)^j(x_{n^*} 
      \restriction \alpha^*_n) = y \restriction \alpha^*_n"$. Moreover, for all $n' > n$, 
      we have
      \[
        \Vdash_{\P_\gamma} ``X_n \restriction \alpha^*_n \text{ and } X_n \restriction 
        \alpha^*_{n'} \text{ are } q_{n'}(\gamma)\text{-consistent}".
      \]
      It follows that, for all $n' > n$, we have
      \[
        \hat{p}_{n,0} \restriction \gamma \Vdash_{\P_\gamma} ``q_{n'}(\gamma)^j(x_{n^*} 
        \restriction \alpha^*_{n'}) = y \restriction \alpha^*_{n'}",
      \]
      and hence
      \[
        \hat{p}_{n,0} \restriction \gamma \Vdash_{\P_\gamma} ``\forall \alpha < \beta \
        [q^-(\gamma)^j(x_{n^*} \restriction \alpha) = y \restriction \alpha]".
      \]
      But $\hat{p}_{n,0} \restriction \gamma \in N$ and $\hat{p}_{n,0} \restriction \gamma 
      \leq p^* \cap N$. Thus, $\hat{p}_{n,0} \restriction \gamma$ and $p^*$ are compatible 
      in $\P_\gamma$, contradicting the fact that they force contradictory statements.
    \end{proof}
    We now define a lower bound $q^*$ for $\langle q_n \mid n < \omega \rangle$ with 
    $\top(q^*) = \beta$ and $\dom(q^*) = N \cap \kappa$ by letting, for each $\gamma 
    \in N \cap \kappa$, $q^*(\gamma)$ be a $\P_\gamma$-name for the unique element of 
    $\mathrm{Aut}(T_{\leq \beta})$ extending $q^-(\gamma)$. Such a name must
    exist by Subclaim \ref{subclaim: extension} and Proposition \ref{prop: limit_extension}.
    It is readily verified that $q^*$ thus defined is indeed a lower bound for 
    $\langle q_n \mid n < \omega \rangle$.
  \end{proof}
  Fix a lower bound $q^*$ for $\langle q_n \mid n < \omega \rangle$ in $\Q$. We first 
  verify that $(1_{\P},q^*) \Vdash_{\M} ``\dot{b} \cap N \notin V[\dot{G}_\delta]"$. To this 
  end, let $w = N \cap \lambda$, and suppose for the sake of contradiction that 
  there are $(p^{**},q^{**}) \leq_{\M} (1_{\P},q^*)$ and an $\M_\delta$-name $\dot{c}$ 
  such that $(p^{**},q^{**}) \Vdash_{\M} ``\dot{b} \cap w = \dot{c}"$. Fix 
  $n < \omega$ such that $p^{**} \cap N = p_n$. Note that both $\hat{p}_{n,0}$ and 
  $\hat{p}_{n,1}$ are compatible with $p^{**}$, and recall that $\hat{p}_{n,0} 
  \restriction \delta = \hat{p}_{n,1} \restriction \delta$. Let $r_0$ denote 
  $(p^{**} \wedge \hat{p}_{n,0}) \restriction \delta$, i.e., $r_0$ is the 
  greatest common lower bound of $p^{**} \restriction \delta$ and $\hat{p}_{n,0} 
  \restriction \delta$ in $\P_\delta$. Find $(r_1,s_1) \leq_{\M_\delta} (r_0, 
  q^{**} \restriction \delta)$ deciding the truth value of the statement 
  $``\varepsilon_n \in \dot{c}"$. Suppose without loss of generality that 
  $(r_1,s_1) \Vdash_{\M_\delta} ``\varepsilon_n \in \dot{c}"$ (the other case is 
  symmetric). Define $(\hat{p}, \hat{q}) \in \M$ by setting 
  \begin{itemize}
    \item $(\hat{p}, \hat{q}) \restriction \delta = (r_1,s_1)$;
    \item $(\hat{p}, \hat{q}) \restriction [\delta,\kappa) = 
    (p^{**} \wedge \hat{p}_{n,0}, q^{**}) \restriction [\delta,\kappa)$.
  \end{itemize}
  Then $(\hat{p}, \hat{q}) \Vdash_{\M} ``\varepsilon_n \in \dot{c}"$, since 
  $(r_1,s_1)$ forces the same. We also have $(\hat{p}, \hat{q}) \leq_{\M} 
  (p^{**},q^{**})$, so $(\hat{p}, \hat{q}) \Vdash_{\M} ``\dot{b} \cap w = \dot{c}"$, 
  and hence $(\hat{p}, \hat{q}) \Vdash_{\M} ``\varepsilon_n \in \dot{b}"$. On the 
  other hand, we have $(\hat{p}, \hat{q}) \leq_{\M} (\hat{p}_{n,0}, q_n)$, and 
  hence $(\hat{p}, \hat{q}) \Vdash_{\M} ``\varepsilon_n \notin \dot{b}"$. This is a 
  contradiction, completing the verification that 
  $(1_{\P},q^*) \Vdash_{\M} ``\dot{b} \cap N \notin V[\dot{G}_\delta]"$.
  
  We finally verify that $(1_{\P},q^*) \Vdash_{\M} ``T_\beta \subseteq \dot{E}"$. 
  Suppose for the sake of contradiction that there are $(p^{**}, q^{**}) 
  \leq_{\M} (1_{\P},q^*)$ and $m < \omega$ such that $(p^{**},q^{**}) \Vdash_{\M} 
  ``x_m \notin \dot{E}"$. Fix $n < \omega$ such that $m < n$ and $p^{**} \cap N 
  = p_n$. By construction, we know that $(\hat{p}_{n,0}, q_n) \Vdash_{\M} 
  ``x_m \restriction \height(q_n) \in \dot{E}"$, and hence 
  $(\hat{p}_{n,0}, q_n) \Vdash_{\M} ``x_m \in \dot{E}"$. Let 
  $r = \hat{p}_{n,0} \wedge p^{**}$. Then $(r,q^{**})$ forces both 
  $``x_m \in \dot{E}"$ and $``x_m \notin \dot{E}"$. This contradiction concludes 
  the proof of the theorem.
\end{proof}

\begin{theorem} \label{theorem: gmp_thm}
  Suppose that $T$ is a free Suslin tree and $\kappa$ is supercompact. Then, in 
  $V[\bb{M}]$, the following statements hold:
  \begin{enumerate}
    \item $\kappa = \omega_2$;
    \item $T$ is an almost Kurepa Suslin tree;
    \item $\mathsf{GMP}$.
  \end{enumerate}
\end{theorem}

\begin{proof}
  Given the previous results from this section, proofs of all three assertions 
  follow standard paths; we provide details for completeness.

  Let $G$ be $\M$-generic over $V$. By a standard counting argument using the 
  inaccessibility of $\kappa$, $\M$ has the $\kappa$-cc in $V$, and hence 
  $\kappa$ remains a cardinal in $V[G]$. Moreover, by Theorem \ref{thm: approx}, 
  $\M$ has the $\omega_1$-approximation property in $V$, and therefore 
  $(\omega_1)^V = (\omega_1)^{V[G]}$. Thus, to prove (1), it suffices to prove 
  that $V[G] \models ``|\gamma| = \omega_1"$ for all uncountable $\gamma < \kappa$. 
  
  To this end, fix an uncountable $\gamma < \kappa$. By increasing it if 
  necessary, assume that $\gamma \in \Sigma(\kappa)$. It is readily shown that 
  the map $\pi_\gamma : \M \ra \P_\gamma \ast \dot{\bb{A}}(T)$ defined by setting 
  $\pi_\gamma(p,q) = (p \restriction \gamma, q(\gamma))$ for all 
  $(p,q) \in \M$ is a projection.\footnote{Here we are adopting the convention that 
  $q(\gamma) = 1_{\dot{\bb{A}}(T)}$ if $\gamma \notin \dom(q)$.} Therefore, 
  $G$ induces a $\P_\gamma \ast \dot{\bb{A}}(T)$-generic filter 
  $G_{0,\gamma} \ast G_{1,\gamma}$ over $V$. In $V[G_{0,\gamma}]$, we have 
  $2^\omega \geq |\gamma|$. Therefore, by Proposition \ref{prop: collapse} 
  applied in $V[G_{0,\gamma}]$, we have $|\gamma| = \omega_1$ in 
  $V[G_{0,\gamma} \ast G_{1,\gamma}]$ and hence in $V[G]$ as well.
  
  To prove (2), first apply Theorem \ref{thm: approx} to conclude that 
  $T$ remains a Suslin tree in $V[G]$. For each $(p,q) \in G$ and 
  $\gamma \in \dom(q)$, let $q(\gamma)_G$ denote the interpretation of 
  $q(\gamma)$ in $V[G]$. For each $\gamma \in \Sigma(\kappa)$, let 
  $g_\gamma = \bigcup \{q(\gamma)_G \mid (p,q) \in G \text{ and } 
  \gamma \in \dom(q)\}$. Then $g_\gamma$ is an automorphism of $T$. 
  Moreover, a routine density argument shows that the sequence of 
  automorphisms $\langle g_\gamma \mid \gamma \in \Sigma(\kappa) \rangle$ 
  is pairwise \emph{almost disjoint}, i.e., for all 
  distinct $\gamma,\delta \in \Sigma(\kappa)$ and all $t \in T$, there 
  is $t' \in T_t$ such that $g_\gamma(t') \neq g_\delta(t')$.
  
  If we now force with $T$ over $V[G]$, we add a generic cofinal branch 
  $b$ through $T$. In $V[G][b]$, for all $\gamma \in \Sigma(\kappa)$ let 
  $b_\gamma = \{g_\gamma(t) \mid t \in b\}$. Then each $b_\gamma$ is a cofinal 
  branch through $T$ and, by the fact that $\langle g_\gamma \mid 
  \gamma \in \Sigma(\kappa) \rangle$ is pairwise almost disjoint and a 
  routine density argument, we have $b_\gamma \neq b_\delta$ for all 
  distinct $\gamma,\delta \in \Sigma(\kappa)$. Since $T$ has the ccc in 
  $V[G]$, and therefore $\kappa = \omega_2^{V[G][b]}$, it follows that 
  $T$ is a Kurepa tree in $V[G][b]$ and is therefore an almost Kurepa 
  Suslin tree in $V[G]$.
  
  It remains to verify that $\GMP$ holds in $V[G]$. To this end, working 
  in $V[G]$, fix a regular uncountable cardinal $\theta \geq \kappa$ and a 
  club $C$ in $\power_\kappa H(\theta)$. We must find an $\omega_1$-guessing model in 
  $C$. Recall that, given a function 
  $f:[H(\theta)]^2 \ra \power_\kappa H(\theta)$, we let $C_f$ denote the set
  \[
    \{X \in \power_\kappa H(\theta) \mid \forall u \in [X]^2 \ f(u) \subseteq X\}.
  \]
  Using Menas's characterization of two-cardinal club filters \cite{menas}, we can fix 
  a function $f: [H(\theta)]^2 \ra \power_\kappa H(\theta)$ such that 
  $C_f \subseteq C$. 
  
  Back in $V$, fix an elementary embedding $j:V \ra M$ witnessing that $\kappa$ is 
  $|H(\theta)|$-supercompact. Note that, in $M$, $j(\M)$ is defined in the same 
  way that $\M$ is defined in $V$, but with length $j(\kappa)$, and that $j(\M)_\kappa = \M$. 
  Moreover, $j \restriction \M$ 
  is the identity map and, by Theorem \ref{thm: approx}, in $M[G]$,
  $j(\M)/G = j(\M)_{\kappa,j(\kappa)}$ has the $\omega_1$-approximation property. 
  Let $H$ be $j(\M)/G$-generic over $V[G]$. Then, in $V[G \ast H]$, we can lift $j$ to 
  $j:V[G] \ra M[G \ast H]$. Note that, in $M[G \ast H]$, $j(C)$ is a club in 
  $\power_{j(\kappa)} H(j(\theta))$ and $C_{j(f)} = j(C_f) \subseteq j(C)$.
  
  Let $N = j``H(\theta)^{V[G]}$. By the closure of $M$ and the fact that $N$ is 
  definable from $j``H(\theta)^V$ and $G \ast H$, it follows that $N \in M[G \ast H]$. 
  Moreover, $M[G \ast H] \models N \in \power_{j(\kappa)} H(j(\theta))$. We claim 
  that $N \in C_{j(f)}$. To this end, fix $u \in [N]^2$, and find 
  $a,b \in H(\theta)^{V[G]}$ such that $u = j(\{a,b\})$. Then 
  $j(f)(u) = j(f(\{a,b\}))$. Since $f(\{a,b\}) \in \power_\kappa H(\theta)^V[G]$ and 
  $\mathrm{crit}(j) = \kappa$, we have $j(f(\{a,b\})) = j``f(\{a,b\}) \subseteq N$. 
  Thus, $N \in C_{j(f)}$, and hence $N \in j(C)$.
  
  We next argue that, in $M[G \ast H]$, $N$ is an $\omega_1$-guessing model. To this end, fix in $M[G \ast H]$ a 
  $z \in N$ and a $d \subseteq z$ such that, for every $x \in \power_{\omega_1} z \cap N$, 
  we have $d \cap x \in N$. Let $z' \in H(\theta)^{V[G]}$ be such that $j(z') = z$, and 
  let $e' = \{a \in H(\theta)^{V[G]} \mid j(a) \in d\}$. Then $e' \subseteq z'$, and 
  $e' \in M[G \ast H]$.
  
  \begin{claim}
    For all $x' \in (\power_{\omega_1} z')^{V[G]}$, we have $e' \cap x' \in M[G]$.
  \end{claim}
  
  \begin{proof}
    Fix $x' \in (\power_{\omega_1} z')^{V[G]}$. Let $x = j(x') = j``x' \in N$.
    By the choice of $d$, we know that $x \cap d \in N$. We can therefore find 
    $y' \in H(\theta)^{V[G]}$ such that $x \cap d = j(y') = j``y'$. But then, 
    unraveling the definitions, we have $e' \cap x' = y' \in V[G]$. 
    Since $H(\theta)^{M[G]} = H(\theta)^{V[G]}$, the desired conclusion follows.
  \end{proof}
  
  Since $j(\M)/G$ has the $\omega_1$-approximation property in $M[G]$, it follows that 
  $e' \in M[G]$; in particular, $e' \in H(\theta)^{V[G]}$. Let $e = j(e')$. Then 
  $e \in N$ and $e \cap N = j``e' = d \cap N$. Thus, $N \in j(C)$ is indeed an $\omega_1$-guessing model 
  in $M[G \ast H]$. By elementarity, in $V[G]$ there exists an $\omega_1$-guessing model in $C$, and hence 
  $\GMP$ holds in $V[G]$.
\end{proof}

The following corollary will now complete the proof of Theorem A.

\begin{corollary} \label{cor: thm_a_cor}
  Assume $\mathrm{Con}(\ZFC + \text{there exists a supercompact cardinal})$. Then it is 
  consistent that $\GMP$ holds but is destructible by a ccc forcing of cardinality $\omega_1$.
\end{corollary}

\begin{proof}
  Start with a model $V$ of $\ZFC$ containing a supercompact cardinal $\kappa$. Since, for example, 
  the classical construction of a Suslin tree from $\diamondsuit$ yields a free Suslin tree, 
  we may assume that there exists a free Suslin tree $T$ in $V$. Then, by Theorem \ref{theorem: gmp_thm}, 
  in $V[\M]$, $\mathsf{GMP}$ holds and $T$ is an almost Kurepa Suslin tree. In particular, 
  $T$ is a ccc forcing of cardinality $\omega_1$ such that, after forcing with $T$, there exists 
  a Kurepa tree. Since $\GMP$ implies the failure of the Kurepa Hypothesis, it follows that 
  $\Vdash_T \neg \GMP$.
\end{proof}

We end this section with a sketch of the proof of Corollary \ref{cor: cor_11}.

\begin{proof}[Sketch of proof of Corollary \ref{cor: cor_11}]
  Suppose that $\kappa$ is an inaccessible cardinal. As in the proof of 
  Corollary \ref{cor: thm_a_cor}, we can fix a free Suslin tree $T$ 
  in $V$, and let $\M = \M^T_\kappa$. We argue that $V[\M]$ 
  satisfies the conclusion of the corollary. Given the proof of Theorem A 
  above, it suffices to show that $\neg \wKH$ holds in $V[\M]$. To this end, 
  let $G$ be $\M$-generic over $V$, and let $S \in V[G]$ be a tree of height 
  and size $\omega_1$. We must show that, in $V[G]$, $S$ has at most 
  $\omega_1$-many cofinal branches.
  
  Since $\kappa$ is inaccessible, we know that $\M$ has the 
  $\kappa$-cc in $V$, and $\kappa = \omega_2$ in $V[G]$. We can therefore 
  find a limit ordinal $\delta < \kappa$ such that $S \in V[G_\delta]$. 
  Recall that, given a tree $R$, $[R]$ denotes the set of cofinal branches 
  through $R$. Since $(2^{\omega_1})^{V[G_\delta]} < \kappa$, we have 
  $([S])^{V[G_\delta]} < \kappa$. By Theorem \ref{thm: approx}, 
  $\M_{\delta,\kappa}$ has the $\omega_1$-approximation property in 
  $V[G_\delta]$, and hence $[S]^{V[G]} = [S]^{V[G_\delta]}$. Since 
  $\kappa = \omega_2$ in $V[G]$, it follows that $S$ has at most 
  $\omega_1$-many cofinal branches in $V[G]$, as desired.
\end{proof}

\section{Weak Kurepa trees} \label{sect: thmb_sect}

In \cite{cox_krueger_indestructible}, Cox and Krueger show that $\GMP$ implies that there are no weak Kurepa trees at $\omega_1$. In contrast to their result, we show in \cite{kurepa_paper} that $\GMP^{\omega_2}(\omega_2)$ does not imply that there are no weak Kurepa trees at $\omega_1$. In fact, we construct a model in which $\GMP^{\omega_2}(\omega_2)$ holds and yet there exists not only a \emph{weak} Kurepa tree, but even a Kurepa tree at $\omega_1$. In this section we refine our result and show that the existence of a weak Kurepa tree at $\omega_1$ is consistent with $\neg \KH$ and $\GMP^{\omega_2}(\omega_2)$, thus proving Theorem B.

\subsection{Forcing for adding a weak Kurepa tree}

We begin this section by introducing a forcing notion to add a weak 
Kurepa tree, which is a variation on the standard forcing for adding a 
Kurepa tree due to Stewart \cite{Stewart}.

\begin{definition}  
Let $\lambda$ be an uncountable ordinal.\footnote{ In practice, $\lambda$ will typically be a cardinal, but 
we will sometimes need to consider $\lambda$ of the form $j(\kappa)$, where $j:V \rightarrow M$ is an 
elementary embedding with critical point $\kappa$, in which case $\lambda$ is a cardinal in $M$ but 
may not be a cardinal in $V$.} We define a poset $\K_\lambda$ which adds a tree with size and height $\omega_1$ with $|\lambda|$-many cofinal branches. Conditions $q \in \K_\lambda$ are pairs $(T_q, f_q)$ such that
  \begin{itemize}
    \item there is $\eta_q < \omega_1$ such that $T_q$ is a normal, infinitely splitting subtree of 
    ${^{<\eta_q + 1}}\omega_1$ of size at most $\omega_1$;
    \item $f_q$ is a countable partial function from $\lambda$ to $T_q \cap {^{\eta_q}}\omega_1$.
  \end{itemize}
  If $q_0, q_1 \in \K_\lambda$, then $q_1 \leq q_0$ if and only if
  \begin{itemize}
    \item $\eta_{q_1} \geq \eta_{q_0}$;
    \item $T_{q_1} \cap {^{<\eta_{q_0} + 1}}\omega_1 = T_{q_0}$;
    \item $\dom(f_{q_1}) \supseteq \dom(f_{q_0})$;
    \item for all $\alpha \in \dom(f_{q_0})$, $f_{q_1}(\alpha) \supseteq f_{q_0}(\alpha)$.
  \end{itemize}
We also include the pair $(\emptyset, \emptyset)$ in $\K_\lambda$ as $1_{\K_\lambda}$.  
\end{definition} 

It is easy to verify that $\K_\lambda$ is separative, and the proof of Lemma 
\ref{C:wKurepaTree} below will show that it adds a tree of height and 
size $\omega_1$ with $|\lambda|$-many branches. Note that $\K_\lambda$ is not $\omega_2$-cc: in fact, it collapses $2^{\omega_1}$ to $\omega_1$ since we can code subsets of $\omega_1$ in the ground model into the levels of the generic tree added by $\K_\lambda$. Therefore if $\lambda\le 2^{\omega_1}$, then the generic tree added by $\K_\lambda$ will not be a weak Kurepa tree.

\begin{lemma}\label{L:K-Knaster}
Let $\lambda$ be an uncountable ordinal and $\mu$ be a regular cardinal such that $\mu>2^{\omega_1}$ and $\mu>\gamma^\omega$ for all $\gamma<\mu$. Then $\K_\lambda$ is $\mu$-Knaster. In particular $\K_\lambda$ is $(2^{\omega_1})^+$-Knaster.
\end{lemma}

\begin{proof}
Let a set of conditions $\set{q_\alpha=(T_\alpha,f_\alpha)\in\K_\lambda}{\alpha<\mu}$ be given. Since $\mu>2^{\omega_1}$ is regular and there are only $2^{\omega_1}$-many possibilities for $T_\alpha$, there is a tree $T\sub {^{<{\eta+1}}\omega_1}$ of countable height and $I\sub\mu$ of size $\mu$ such that $T_\alpha=T$ for all $\alpha\in I$.

Since $\gamma^\omega<\mu$ for all $\gamma<\mu$, there is $I'\sub I$ of size $\mu$ such that the set $\set{\dom(f_\alpha)}{\alpha\in I'}$ forms a $\Delta$-system with root $a\sub\lambda$. Since $a$ is at most countable, there are at most $2^\omega<\mu$ many functions from $a$ to $T_\eta$ and therefore there is a countable $f$ from $a$ to $T_\eta$ and $J\sub I'$ of size $\mu$ such $f=f_\alpha\cap f_\beta$ for all $\alpha\neq\beta\in J$. Then all conditions in $\set{q_\alpha}{\alpha\in J}$ are compatible.
\end{proof}

\begin{lemma}\label{L:K-closed}
Let $\lambda$ be an uncountable ordinal. Then $\K_\lambda$ is $\omega_1$-closed.
\end{lemma}
  
\begin{proof}
    Let $\langle q_n \mid n < \omega \rangle$ be a decreasing sequence from $\K_\lambda$. To avoid 
    trivialities, assume that the sequence $\langle \eta_{q_n} \mid n < \omega \rangle$ is strictly 
    increasing. Let $\eta := \sup\{\eta_{q_n} \mid n < \omega\}$. We will construct a lower bound 
    $q$ for $\langle q_n \mid n < \omega \rangle$ such that $\eta_q = \eta$. Let $T = 
    \bigcup_{n < \omega} T_{q_n}$. Note that $T$ is a normal tree of height $\eta$. To define $T_q$, first note that there are countably many branches through $T$ that we are obliged to extend 
    because of the functions $\{f_{q_n} \mid n < \omega\}$. Namely, let $a = \bigcup_{n < \omega} 
    \dom(f_{q_n})$ and, for each $\alpha \in a$, let 
    \[
      b_\alpha := \bigcup \{f_{q_n}(\alpha) \mid n < \omega \wedge \alpha \in \dom(f_{q_n})\}.
    \]
    Then each $b_\alpha$ is a cofinal branch through $T$, and we are obliged to put $b_\alpha$ 
    in $T_q$.

Then we need to ensure that $T_q$ will be a normal, infinitely splitting subtree of 
    ${^{<\eta + 1}}\omega_1$; namely for every node in $T$ we need to add a node above it on level $\eta$. For every $t\in T$ let $C_t$ be a cofinal branch in $T$ such that $t\in C_t$. Note that such a cofinal branch exists since $T$ is a normal tree of countable height. Let $c_t=\bigcup C_t$, then $c_t\in {^{\eta}}\omega_1$ and extends $t$.  

 At the end, set $T_q := T \cup \{b_\alpha \mid \alpha \in a\} \cup 
    \{c_t \mid t\in T\}$. Let $f_q$ be such that $\dom(f_q) = a$ and, for all 
    $\alpha \in a$, $f_q(a) = b_\alpha$. 
\end{proof}

\begin{lemma}\label{L:wKT-projection}
Let $\lambda<\kappa$ be uncountable ordinals. Then there is a projection from $\K_\kappa$ to $\K_\lambda$.
\end{lemma}

\begin{proof}
We define $\pi$ from $\K_\kappa$ to $\K_\lambda$ by letting $\pi(T_q,f_q)=(T_q,f_q\rest\lambda)$ for all $q\in\K_\kappa$. It is routine to verify that $\pi$ is order-preserving and that $\pi(1_{\K_\kappa})=1_{\K_\lambda}$.

Let $q\in\K_\kappa$ and $r\in\K_\lambda$ be such that $r\le \pi(q)=(T_q,f_q\rest\lambda)$. Define $r'\le q$ by first letting $T_{r'}= T_r$ (note that $T_r$ is an end-extension of $T_q$). Let $\dom (f_{r'})=\dom(f_{q})\cup\dom(f_r)$ and define 
\begin{itemize}[--]
\item $f_{r'}(\alpha)=f_r(\alpha)$ for every $\alpha\in\dom(f_r)$ and
\item  $f_{r'}(\alpha)=\tau$, where $\tau\in (T_{r'})_{\eta_{r}}$ with $\tau\supseteq f_q(\alpha)$, for every $\alpha\in \dom(f_q)\setminus\dom(f_r)$. 
\end{itemize}
It is easy to check that $\pi(r')=r$. Therefore $\pi$ is a projection.
\end{proof}

Assume that $\lambda<\kappa$ are uncountable ordinals. Let $H$ be a $\K_\lambda$-generic filter, $T_H=\bigcup\set{T_q}{q\in H}$, and let $\K_{\kappa}/H=\set{r\in\K_\kappa}{\pi(r)\in H}$ be the quotient given by $H$ and the projection $\pi$. Then $\K_\kappa$ is forcing equivalent to a two step iteration $\K_\lambda* \K_\kappa/H$. It is easy to see that in $V[H]$, $\K_\kappa/H$ is forcing equivalent to the forcing notion $\K_{\lambda,\kappa}$, where conditions in $\K_{\lambda,\kappa}$ are pairs $r=(\eta_r,f_r)$ such that $f_r$ is a countable partial function from $\kappa\setminus\lambda$ to $(T_H)_{\eta_r}$ and $r\le q$ if and only if $\eta_r\ge\eta_q$, $\dom(f_q)\sub\dom(f_r)$, and $f_q(\alpha)\sub f_r(\alpha)$ for all $\alpha\in\dom(f_q)$.

\begin{lemma}\label{L:quotient-knaster}
Let $\lambda<\kappa$ be uncountable ordinals and $\mu$ be a regular cardinal such that $\mu>{2^{\omega_1}}$ and $\mu>\gamma^\omega$ for all $\gamma<\mu$. Let $H$ be a $\K_\lambda$-generic filter. Then $\K_{\kappa}/H$ is $\mu$-Knaster in $V[H]$.
\end{lemma}

\begin{proof}

Since $\mu>{2^{\omega_1}}$ is regular and $\mu>\gamma^\omega$ for all $\gamma<\mu$, then $\K_\lambda$ is $\mu$-Knaster and hence $\mu$ is still a regular cardinal in $V[H]$. Moreover since $\K_\lambda$ is $\omega_1$-closed, $\mu>\gamma^\omega$ still 
holds for all $\gamma<\mu$ in $V[H]$.

Work in $V[H]$, where $\K_\kappa/H$ is forcing equivalent to the forcing notion $\K_{\lambda,\kappa}$ isolated above. Let a set of conditions $\set{q_\alpha=(\eta_{\alpha},f_\alpha)\in\K_{\lambda,\kappa}}{\alpha<\mu}$ be given. Since $\mu>2^{\omega_1}$ in $V$, $\mu>\omega_1$ in $V[H]$ and there are only $\omega_1$-many possibilities for $\eta_\alpha$; therefore there is an $\eta$ and an $I\sub\mu$ of size $\mu$ such that $\eta_\alpha=\eta$ for all $\alpha\in I$.

Since $\gamma^\omega<\mu$ for all $\gamma<\mu$, there is $I'\sub I$ of size $\mu$ such that the set $\set{\dom(f_\alpha)}{\alpha\in I'}$ forms a $\Delta$-system with root $a\sub\lambda$. Since $a$ is at most countable, there are at most $\omega_1^\omega<\mu$ many functions from $a$ to $T_\eta$ and therefore there is a countable $f$ from $a$ to $T_\eta$ and $J\sub I'$ of size $\mu$ such $f=f_\alpha\cap f_\beta$ for all $\alpha\neq\beta\in J$. Then all conditions in $\set{q_\alpha}{\alpha\in J}$ are compatible.
\end{proof}

\begin{lemma}
Let $\lambda<\kappa$ be uncountable ordinals and $H$ be a $\K_\lambda$-generic filter. Then $\K_{\kappa}/H$ is $\omega_1$-distributive in $V[H]$.
\end{lemma}

\begin{proof}
The forcing $\K_\kappa$ is forcing equivalent to a two step iteration $\K_\lambda*\K_{\kappa}/\dot{H}$. Since $\K_\kappa$ is $\omega_1$-closed by Lemma \ref{L:K-closed}, $\K_{\kappa}/H$ cannot add new countable sequences of ordinals over $V[H]$, hence it is $\omega_1$-distributive in $V[H]$.
\end{proof}

We fix the following notation: for $r\in \K_\kappa$ and $\lambda<\kappa$, let $r\rest\lambda$ denote $(T_r,f_r\rest\lambda)$, i.e.\  $r\rest\lambda=\pi(r)$ for the projection $\pi$ defined in Lemma \ref{L:wKT-projection}.

\begin{lemma}\label{L:into}
Let $\lambda<\kappa$ be uncountable ordinals and $\dot{H}$ be a canonical $\K_\lambda$-name for the $\K_\lambda$-generic filter. Let $q\in \K_\lambda$ and $r\in\K_\kappa$. Then the following are equivalent.
\begin{enumerate}[(i)]
\item  $q\Vdash r\in \K_\kappa/\dot{H}$;
\item $q\le r\rest\lambda$.
\end{enumerate}
\end{lemma}

\begin{proof}
To see that (i) implies (ii), note that if $q\not\le r\rest\lambda$, then by the separativity of $\K_\lambda$ there is $q'\le q$ which is incompatible with $r\rest\lambda$, and hence $q$ does not force $r$ into $\K_\kappa/\dot{H}$. The other direction is clear, since $q\le r\rest\lambda$ means that $q\le\pi(r)$.
\end{proof}

\begin{lemma}\label{C:wKurepaTree}
Assume $\GCH$ and let $\lambda > \omega_2$ be a cardinal. Let $H$ be a $\K_\lambda$-generic filter over $V$. Then the generic tree $T_H=\bigcup\set{T_q}{q\in H}$ is a weak Kurepa tree with $\lambda$-many cofinal branches.
\end{lemma}

\begin{proof}
Since  $\K_\lambda$ is $\omega_1$-closed, $\omega_1$ is preserved by $\K_\lambda$ and the generic tree $T_H=\bigcup\set{T_q}{q\in H}$ is thus a tree with height and size $\omega_1$. By a standard density argument, $T_H$ has $\lambda$-many cofinal branches in $V[H]$. Since $\GCH$ holds in the ground model, $\K_\lambda$ is $\omega_3$-Knaster by Lemma \ref{L:K-Knaster}, hence all cardinals greater than $\omega_2$ are preserved (recall that $2^{\omega_1}$ is always collapsed); in particular $\lambda$ is preserved. Since $\lambda>\omega_2$ in the ground model, $\lambda>\omega_1$ in $V[H]$; therefore the generic tree $T_H$ is a weak Kurepa tree in $V[H]$.
\end{proof}

\subsection{The consistency of $\wKH$ with $\neg\KH$ and $\GMP^{\omega_2}(\omega_2)$}
Before proving the main theorem of this section, we show that no Kurepa trees exist in the generic extension by the product of the Mitchell forcing (as defined in Subsection \ref{mitchell_section}) and the forcing for adding a weak Kurepa tree.
\begin{theorem}\label{Th:NoKurepa}
Suppose that $\kappa$ is an inaccessible cardinal. Then, in the generic extension by $\M(\omega,\kappa)\times \K_\kappa$, there are no Kurepa trees.
\end{theorem}
\begin{proof}

Let $\kappa$ be an inaccessible cardinal, and assume that $\GCH$ holds.

First note that it is enough to show that there are no Kurepa trees in the generic extension by $\Add(\omega,\kappa)\times \Q\times \K_\kappa$, where $\Q$ is the term forcing of the Mitchell forcing $\M(\omega,\kappa)$ (see Remark \ref{mitchell_remark} for more details; formally, $\Q$ is the set of all conditions $(\emptyset, q) \in 
\M(\omega,\kappa)$, with the order inherited from $\M(\omega,\kappa)$).  To see this, first observe that $\Add(\omega,\kappa)\times \Q\times \K_\kappa$ is forcing equivalent to $(\M(\omega,\kappa)*\S)\times \K_\kappa$ for some quotient forcing $\S$ by Remark \ref{mitchell_remark}(\ref{mitchell_product}). This in turn is forcing equivalent to $(\M(\omega,\kappa)\times \K_\kappa)*\S$. Moreover, $\S$ does not collapse any cardinals over $V[\M(\omega,\kappa)\times \K_\kappa]$, since it is easily checked that the cardinals are the same in both extensions; i.e. $\omega_1$ and all cardinals greater or equal to $\kappa$ are preserved in both extensions and cardinals between $\omega_1$ and $\kappa$ are collapsed in both extensions. Hence, any Kurepa tree in $V[\M(\omega,\kappa)\times \K_\kappa]$ remains a Kurepa tree in $V[\Add(\omega,\kappa)\times \Q\times \K_\kappa]$.

Let $ F\times G\times H$ be $\P\times \Q\times \K_\kappa$-generic over $V$, where $\P$ denotes $\Add(\omega,\kappa)$.  Assume that $S$ is an $\omega_1$-tree in $V[F][G][H]$. Since $S$ has size $\omega_1$ and $\P\times \Q\times\K_\kappa$ is $\kappa$-Knaster, there is a nice $\P\times \Q\times\K_\kappa$-name $\dot{S}$ of size less than $\kappa$ for $S$. Since $\dot{S}$ has size less than $\kappa$,  $\dot{S}$ is a $\P_\theta\times \Q_\theta\times \K_\theta$-name for some regular cardinal $\theta < \kappa$, where $\P_\theta=\Add(\omega,\theta)$ and $\Q_\theta=\set{(\emptyset, q\rest\theta)}{(\emptyset, q)\in \Q}$. Let $F_\theta$ denote the $\P_\theta$-generic over $V$ determined by $F$; i.e.\ $F_\theta=\set{p\rest\theta}{p\in F}$, $G_\theta$ denote the $\Q_\theta$-generic over $V[F_\theta]$ determined by $G$; i.e.\ $G_\theta=\set{(\emptyset, q\rest\theta)}{(\emptyset, q)\in G}$, and let $H_\theta$ denote the $\K_\theta$-generic filter over $V[F_\theta][G_\theta]$ determined by $H$ and $\pi$, where $\pi$ is the projection from $\K_\kappa$ to $\K_\theta$ from Lemma \ref{L:wKT-projection}. The $\omega_1$-tree $S$ is an element of $V[F_\theta][G_\theta][H_\theta]$ and has at most $(2^{\omega_1})^{V[F_\theta][G_\theta][H_\theta]}$-many cofinal branches here, which is less than $\kappa$, since $\kappa$ is still inaccessible in $V[F_\theta][G_\theta][H_\theta]$. We show that the quotient forcing $\P^\theta\times\Q^\theta\times \K_\kappa/H_\theta$ cannot add a cofinal branch to $S$ over $V[F_\theta][G_\theta][H_\theta]$. Here $\P^\theta$ denotes $\Add(\omega,[\theta,\kappa))$ and $\Q^\theta$ denotes the quotient forcing $\Q/{G_\theta}$. It will follow that $S$ has ${<}\kappa=\omega_2$-many cofinal branches in $V[F][G][H]$, i.e.\ it is not a Kurepa tree. Since $S$ was an arbitrary $\omega_1$-tree, this will conclude the proof of Theorem \ref{Th:NoKurepa}. The proof that there are no Kurepa trees in $V[F][G][H]$ is relatively long; for easier reading, it is divided into several claims (Claims \ref{C:diff} to \ref{cl:last}).

First we observe that we can express the generic extension $V[F][G][H]$ as a forcing extension by $(\P_\theta\times\Q_\theta\times\K_\theta)*(\K_{\kappa}/\dot{H}_\theta\times\Q^\theta\times \P^\theta)$: Since $\P$ is forcing equivalent to $\P_\theta\times \P^\theta$, we can view $F$ as a filter $F_\theta\times f$  which is $\P_\theta\times \P^\theta$-generic over $V$. Similarly, since $\Q$ is forcing equivalent to $\Q_\theta*\dot{\Q}^\theta$, we can view $G$ as a filter $G_\theta*g$ which is generic over $V[F_\theta][f]$ and lastly since $\K_\kappa$ is forcing equivalent to $\K_\theta*\K_{\kappa}/\dot{H}_\theta$, we can view $H$ as a filter $H_\theta*h$ which is generic over $V[F_\theta][f][G_\theta][g]$. Therefore $V[F][G][H]$ is equal to the generic extension $V[F_\theta][f][G_\theta][g][H_\theta][h].$

Moreover, since $\P^\theta$ and $\Q_\theta*\Q^\theta$ are both in $V[F_\theta]$ and $G_\theta*g$ is generic over $V[F_\theta][f]$, we have $V[F_\theta][f][G_\theta][g][H_\theta][h]=V[F_\theta][G_\theta][g][f][H_\theta][h]$. Similarly, we can swap $g\times f$ and $H_\theta*h$ since both $\Q^\theta\times\P^\theta$ and $\K_\theta*\K_{\kappa}/H_\theta$ are in $V[F_\theta][G_\theta]$ and $H_\theta*h$ is generic over $V[F_\theta][H_\theta][g][f]$. It follows that $$V[F_\theta][G_\theta][g][f][H_\theta][h]=V[F_\theta][G_\theta][H_\theta][h][g][f].$$

Since $\P^\theta=\Add(\omega,[\theta,\kappa))$, it is $\omega_1$-Knaster in $V[F_\theta][G_\theta][H_\theta][h][g]$ and hence it cannot add new cofinal branches to an $\omega_1$-tree by Fact \ref{F:Knaster}. Therefore, any cofinal branch through $S$ is already in $V[F_\theta][G_\theta][H_\theta][h][g]$. 

Now, note that $\Q^\theta$ is $\omega_1$-closed in $V[G_\theta]$ and also $\K_\kappa$ is $\omega_1$-closed in $V[G_\theta]$ since $\Q_\theta$ and $\K_\kappa$ are both $\omega_1$-closed in $V$; therefore $\Q^\theta$ is $\omega_1$-closed in $V[G_\theta][H_\theta][h]$. The forcing $\P_\theta$ is ccc in $V[G_\theta][H_\theta][h]$ since it is just Cohen forcing for adding subsets of $\omega$. Therefore we can apply Fact \ref{F:ccc_Closed} to $\P_\theta$ and $\Q^\theta$ over $V[G_\theta][H_\theta][h]$ to show that $\Q^\theta$ cannot add new cofinal branches to an $\omega_1$-tree over $V[F_\theta][G_\theta][H_\theta][h]$ and hence any cofinal branch through $S$ is already in $V[F_\theta][G_\theta][H_\theta][h]$

To show that $\K_{\kappa}/H_\theta$ cannot add a cofinal branch to $S$ over $V[F_\theta][G_\theta][H_\theta]$, we will work in $V[G_\theta]$. Note that $\K_{\kappa}/H_\theta$ is only $\omega_1$-distributive in $V[G_\theta][H_\theta]$ and therefore we cannot use Fact \ref{F:ccc_Closed}. Instead, we will show that if there is a $\P_\theta\times\K_\theta*(\K_{\kappa}/\dot{H}_\theta)$-name $\dot{b}$ and a condition $(p^*,q^*,r^*)\in \P_\theta\times\K_\theta*(\K_\kappa/\dot{H}_\theta)$ which forces that $\dot{b}$  is a cofinal branch through $\dot{S}$ that is not in $V[G_\theta][\dot{F}_\theta][\dot{H}_\theta]$, then there is a condition $q\le q^*$ such that $(p^*,q)$ forces that $\dot{S}$ has an uncountable level. This is a contradiction, since we can assume that $(p^*,q^*,r^*)$ is in $F_\theta\times H_\theta*h$ and that $(p^*,q^*)$ forces that $\dot{S}$ is an $\omega_1$-tree.

The proof is similar to the argument that an $\omega_1$-closed forcing over a ccc forcing does not add cofinal branches to $\omega_1$-trees: we will build a tree $\mathcal{T}$ of conditions in $\K_\kappa$ labeled by ${^{<\omega}2}$ and we will diagonalize over antichains in $\P_\theta$, but since we are working with the quotient $\K_{\kappa}/\dot{H}_\theta$ and we work in $V[G_\theta]$ and not in $V[G_\theta][H_\theta]$, we will guide this construction by a decreasing sequence of length $\omega$ of conditions in $\K_\theta$, which will force conditions from $\mathcal{T}$ into the quotient $\K_{\kappa}/\dot{H}_\theta$.

We make a natural identification and view $(\dot{S}, <_{\dot{S}})$ as a name for a tree with underlying set $\omega_1\times\omega$, assuming that, for each $\delta<\omega_1$, the $\delta^{\mathrm{th}}$ level of $\dot{S}$ is forced to  be  equal to $\{\delta\}\times\omega$. Hence the domain of the tree exists in $V[G_\theta]$, with the forcing deciding the ordering $<_{\dot{S}}$ on the tree. Going forward, 
for $\delta < \omega_1$, we let $\dot{S}_\delta$ denote $\{\delta\} \times \omega$.

Work in $V[G_\theta]$. Assume that $\dot{b}$ is a $\P_\theta\times \K_\theta*(\K_{\kappa}/\dot{H}_\theta)$-name and $(p^*,q^*,\dot{r}^*)\in \P_\theta\times\K_\theta*(\K_{\kappa}/\dot{H}_\theta)$ forces that $\dot{b}$ is a cofinal branch through $\dot{S}$ that is not in $V[G_\theta][\dot{F}_\theta][\dot{H}_\theta]$. By our assumption about $\dot{S}$, we can think of $\dot{b}$ as a name for an element of ${^{\omega_1}}\omega$. We can also assume that there is $r^* \in \K_\kappa$ such that $q^* \Vdash \dot{r}^* = r^*$. We will build by induction on $\omega$ the following objects:
\begin{itemize}
\item a decreasing sequence $\seq{q_n}{ n<\omega }$  of conditions in $\K_\theta$ with $q_0 \le q^*$;
\item a labeled tree $\mathcal{T}=\set{r_s}{s\in {^{<\omega}}2 }$ of conditions in $\K_{\kappa}$, with $r_s \le r^*$ for all $s\in {^{<\omega}}2 $;
\item a maximal antichain $A_s$ in $\P_\theta$ below $p^*$ for all $s\in {^{<\omega}}2 $;
\item a strictly increasing sequence $\seq{\gamma_n}{ n<\omega }$ of ordinals below $\omega_1$;
\end{itemize}

such that the following hold for all $n$ and all $s\in {^{<\omega}}2$ of length $n$:

\begin{enumerate}[(a)]
\item $q_n=r_s\rest\theta$; in particular $q_n\Vdash r_s\in \K_\kappa/\dot{H}_\theta$;
\item for all $p\in A_{s}$ the conditions $(p,q_{n+1},r_{s^{\smallfrown}0})$  and  $(p,q_{n+1},r_{s^{\smallfrown}1})$ decide $\dot{b}$ up to $\gamma_{n+1}$ differently; i.e.\ there are $\delta
\le\gamma_{n+1} $ and $\tau_{s^{\smallfrown}0}\neq \tau_{s^{\smallfrown}1}$ in $\dot{S}_\delta$ such that $(p,q_n,r_{s^{\smallfrown}0})\Vdash \dot{b}(\delta)= \tau_{s^{\smallfrown}0}$ and $(p, q_n,r_{s^{\smallfrown}1})\Vdash \dot{b}(\delta)= \tau_{s^{\smallfrown}1}$;
\item $(q_{|t|},r_t)\le (q_{|s|},r_s)$ for all $s\subseteq t$ in  ${^{<\omega}}2$.
\end{enumerate}

\begin{definition}\label{def:system}
Let us call a system as above, satisfying conditions (a)--(c), a \emph{labeled system} for $\dot{b}$. We say just a labeled system if $\dot{b}$ is clear from the context.
\end{definition}

A labeled system will be constructed below, using Claims \ref{C:diff}, \ref{C:antichains}, and \ref{cl:conditions}.

\begin{claim}\label{C:diff}
For every $(p,q,r^0) , (p,q,r^1) \le (p^*,q^*,r^*)$ in $\P_\theta\times \K_\theta*(\K_{\kappa}/\dot{H}_\theta)$ and $\gamma'<\omega_1$ there are $\gamma'<\gamma<\omega_1$, $(p',q',u^0)\le (p,q,r^0)$  and $(p',q',u^1)\le (p,q,r^1)$ such that $(p',q',u^0)$ and $(p',q',u^1)$ decide $\dot{b}(\gamma)$ differently and the condition $q'$ extends both $u^0\rest\theta$ and $u^1\rest\theta$. Moreover we can ensure that $T_{q'}=T_{u^0}=T_{u^1}$.
\end{claim}

\begin{proof}
Let $(p,q,r^0) , (p,q,r^1) \le (p^*,q^*,r^*)$ and $\gamma'<\omega_1$ be given. Fix for the moment a $\P_\theta\times\K_\theta$-generic $G_P\times G_K$ over $V[G_\theta]$ such that $(p,q)\in G_P\times G_K$, and work in $V[G_\theta][G_P\times G_K]$. Since $\dot{b}$ is forced by $(p^*,q^*,r^*)$ to be a new cofinal branch through $\dot{S}$ and $(p^*,q^*)\in G_P\times G_K$, there are $\tilde{r}^0\le r^0$, $\tilde{r}^1\le r^1$ and $\gamma>\gamma'$ such that $\tilde{r}^0$ and $\tilde{r}^1$ decide $\dot{b}(\gamma)$ differently; i.e.\ there are $\tau^0\neq\tau^1$ in $\dot{S}_\gamma$ such that $\tilde{r}^0\Vdash \dot{b}(\gamma)=\tau^0$ and $\tilde{r}^1\Vdash \dot{b}(\gamma)=\tau^1$. 

Since $\tilde{r}^0$ and $\tilde{r}^1$ are in $\K_{\kappa}/G_K$, $\tilde{r}^0\rest\theta$ and $\tilde{r}^1\rest\theta$ are in $G_K$ and hence they are compatible. Let $q'$ be a common extension of $\tilde{r}^0\rest\theta$, $\tilde{r}^1\rest\theta$ and $q$ such that, for each $i < 2$, 
we have $(q', \tilde{r}^i) \Vdash \dot{b}(\gamma) = \tau^i$ in 
$V[G_\theta][G_P]$. Since $G_P$ is $\P_\theta$-generic over $V[G_\theta]$, there is a condition $p'\in G_P$ which forces this, and therefore $(p',q',\tilde{r}^0)$ and $(p',q',\tilde{r}^1)$ decide $\dot{b}(\gamma)$ differently over $V[G_\theta]$. Since we assume that $p\in G_P$ we can take $p'$ such that it extends $p$.

Since $q'$ extends both $\tilde{r}^0 \rest \theta$ and 
$\tilde{r}^1\rest\theta$, $T_{q'}$ is an end-extension of both 
$T_{\tilde{r}^0}$ and $T_{\tilde{r}^1}$. Let us define 
$u^i=(T_{u^i}, f_{u^i})\le (T_{\tilde{r}^i}, f_{\tilde{r}^i})= \tilde{r}^i$ for $i<2$ as follows: set $T_{u^i}=T_{q'}$ and $\dom (f_{u^i})=\dom(f_{q'}) \cup \dom(f_{\tilde{r}^i})$, and define

\begin{itemize}[--]
\item $f_{u^i}(\delta)=f_{q'}(\delta)$ for every $\delta\in\dom(f_{q'})$ and
\item  $f_{u^i}(\delta)=\tau$, where $\tau\in (T_{q'})_{\eta_{q'}}$ with $\tau\supseteq f_{\bar{r}^i}(\delta)$, for every $\delta\in \dom(f_{\bar{r}^i})\setminus\dom(f_{q'})$. 
\end{itemize}

It is easy to see that  $q'$ still extends $u^i\rest\theta$ for $i<2$ and therefore it forces $u^i$ into $\K_{\kappa}/\dot{H}_\theta$. Moreover, by the definition of $u^i$ for $i<2$, $(q', u^i)$ extends $(q',\tilde{r}^i)$ and therefore  $(p',q',u^0)$ and $(p',q',u^1)$ decide $\dot{b}(\gamma)$ differently. 
\end{proof}

\begin{claim}\label{C:antichains}
For every $(q,r^0) , (q,r^1) \le (q^*,r^*)\in \K_\theta*(\K_{\kappa}/\dot{H}_\theta)$ and $\gamma'<\omega_1$ there are $\gamma'<\gamma<\omega_1$, a maximal antichain $A$ in $\P_\theta$ below $p^*$, $(q',\bar{r}^0)\le (q,r^0)$, and $(q',\bar{r}^1)\le (q,r^1)$ such that for every $p\in A$, $(p, q',\bar{r}^0)$ and $(p,q',\bar{r}^1)$ decide $\dot{b}$ up to $\gamma$ differently and the condition $q'$ extends both $\bar{r}^0\rest\theta$ and $\bar{r}^1\rest\theta$.

\end{claim}

\begin{proof}
Let $(q,r^0) , (q,r^1) \le (q^*,r^*)\in \K_\theta*(\K_{\kappa}/\dot{H}_\theta)$ and $\gamma'<\omega_1$ be given. We proceed by induction of length $\alpha$ for some $\alpha$ less than $\omega_1$ that will be determined during the construction, and we construct $q'$, $\bar{r}^0$ and $\bar{r}^1$ such that $(q',\bar{r}^0)$ is a lower bound of a decreasing sequence $\seq{(q'_\beta,r^0_{\beta})\le (q,r^0) }{\beta<\alpha}$ and $(q',\bar{r}^{1})$ is a lower bound of a decreasing sequence $\seq{(q'_\beta,r^1_{\beta})\le (q,r^1) }{\beta<\alpha}$. In the process, we also construct a maximal antichain $A=\set{p_\beta}{\beta < \alpha}$ in $\P_\theta$, and $\gamma$ will be a supremum of increasing sequence of ordinals $\seq{\gamma_\beta<\omega_1}{\beta<\alpha}$.

To begin, let $p$ be an arbitrary condition below $p^*$. By Claim \ref{C:diff}, there are $p_0\le p$, $q'_0\le q$, $r^0_0\le r^0$,$r^1_{0}\le r^1$ and $\gamma_0>\gamma'$ such that $(p_0,q'_0, r^0_{0})$, $(p_0,q'_0,r^1_{0})$ decide $\dot{b}(\gamma_0)$ differently, $q'_0$ extends both conditions $r^0_{0}\rest\theta$ and $r^1_{0}\rest\theta$, and moreover $T_{q'_0}=T_{r^0_{0}}=T_{r^1_{0}}$.

Let $\alpha<\omega_1$ and assume that for all $\beta<\alpha$, we have constructed $p_\beta$, $\gamma_{\beta}$, $q'_\beta$, $r^0_\beta$, and $r^1_\beta$. If there is a condition $p$ below $p^*$ which is incompatible with all conditions in $\set{p_\beta}{\beta<\alpha}$, we fix such a $p$ and proceed as follows. 

If $\alpha=\beta+1$, by Claim \ref{C:diff}, there are $p_\alpha\le p$, $q'_\alpha\le  q'_\beta$, $r^0_\alpha\le r^0_\beta $, $r^1_{\alpha}\le r^1_\beta$ and $\gamma_\alpha>\gamma_\beta$ such that $(p_\alpha,q'_\alpha, r^0_{\alpha})$, $(p_\alpha,q'_\alpha,r^1_{\alpha})$ decide $\dot{b}(\gamma_\alpha)$ differently, $q'_\alpha$ extends both conditions $r^0_{\alpha}\rest\theta$ and $r^1_{\alpha}\rest\theta$, and moreover $T_{q'_\alpha}=T_{r^0_{\alpha}}=T_{r^1_{\alpha}}$. 

If $\alpha$ is limit, we begin by constructing appropriate lower bounds of $\seq{(q'_\beta,r^0_{\beta})\le (q,r^0) }{\beta<\alpha}$ and $\seq{(q'_\beta,r^1_{\beta})\le (q,r^1) }{\beta<\alpha}$ such that the lower bounds will have the same first coordinate.\footnote{Note that $\K_\theta*(\K_{\kappa}/\dot{H}_\theta)$ is $\omega_1$-closed, therefore the lower bounds of $\seq{(q'_\beta,r^0_{\beta})\le (q,r^0) }{\beta<\alpha}$ and $\seq{(q'_\beta,r^1_{\beta})\le (q,r^1) }{\beta<\alpha}$ always exist; however, for our construction we need to ensure that the lower bounds have the same first coordinate.} Let $q'_\beta=(T_\beta,f_\beta)$, $r^0_\beta=(S^0_\beta,g^0_\beta)$, and $r^1_\beta  =(S^1_\beta,g^1_\beta)$ for all $\beta<\alpha$. Note that at each step $\beta<\alpha$ of the construction we ensured that  $T_\beta=S^0_\beta=S^1_\beta$, hence $T=\bigcup_{\beta<\alpha}T_\beta=\bigcup_{\beta<\alpha} S^0_{\beta}=\bigcup_{\beta<\alpha} S^1_{\beta}$

If $\height(T)$ is a limit ordinal $\eta<\omega_1$ we define an end extension $T'$ of $T$ by one level. To define $T'$, first note that there are countably many branches through $T$ that we are obliged to extend because of the functions $\set{f_{\beta}}{ \beta < \alpha}$, $\set{g^0_{\beta}}{ \beta < \alpha}$, and $\set{g^1_{\beta}}{ \beta < \alpha}$. Namely, let $a = \bigcup_{\beta< \alpha} 
    \dom(f_{\beta})$ and, for each $\delta \in a$, let 
    \[
      b_\delta := \bigcup \set{f_{\beta}(\delta)}{\beta < \alpha \wedge \delta\in \dom(f_{\beta})}.
    \]
    Then each $b_\delta$ is a cofinal branch through $T$, and we are obliged to put $b_\delta$ 
    in $T'$.

Similarly for functions $\set{g^i_{\beta}}{ \beta < \alpha}$ for $i<2$, let $c_i = \bigcup_{\beta< \alpha} 
    \dom(g^i_{\beta})$ and, for each $\delta \in c_i$, let 
    \[
      d^i_\delta := \bigcup \set{g^i_{\beta}(\delta)}{\beta < \alpha \wedge \delta\in \dom(g^i_{\beta})}.
    \]
    Then each $d^i_\delta$ is a cofinal branch through $T$, and we are obliged to put $d^i_\delta$ 
    in $T'$.

As in the proof of Lemma \ref{L:K-closed}, for each $t \in T$, let $E_t$ be a cofinal branch in $T$ 
containing $t$, and let $e_t = \bigcup E_t$. Then set 
\[
  T' = T\cup\set{b_\delta}{\delta\in a}\cup\set{d^0_\delta}{\delta \in c_0}\cup \set{d^1_\delta}{\delta   \in c_1} \cup \set{e_t}{t \in T}.
\]
Now let us define $\bar{q}=(T_{\bar{q}}, f_{\bar{q}}) $ as follows: $T_{\bar{q}}=T'$, $\dom (f_{\bar{q}})=a$ and, for all $\delta \in a$, $f_{\bar{q}}(\delta) = b_\delta$. Similarly define $u^i=(T_{u^i}, f_{u^i}) $ for $i<2$ as follows:  $T_{u^i}=T'$, $\dom (f_{u^i})=c^i$ and, for all $\delta \in c^i$, $f_{u^i}(\delta) = d^i_\delta$.

If $\height(T)$ is a successor ordinal $\eta+1<\omega_1$, then there is $\beta^*<\alpha$ such thet for all $\beta$ between $\beta^*$ and $\alpha$, $T_\beta=T_{\beta^*}=T$ and hence $S^0_\beta=S^1_\beta=T_{\beta^*}=T$ for all $\beta$ between $\beta^*$ and $\alpha$. Now, $T$ is a normal tree and we can easily define lower bounds using $T$.  Let us define $\bar{q}=(T_{\bar{q}}, f_{\bar{q}}) $ as follows: $T_{\bar{q}}=T$, $\dom (f_{\bar{q}})=\bigcup_{\beta< \alpha} \dom(f_{\beta})$ and, for all $\delta \in \dom (f_{\bar{q}})$, $f_{\bar{q}}(\delta) = f_\beta(\delta)$, where $\beta\ge\beta^*$ is such that $\delta\in\dom(f_\beta)$. Analogously, we define $u^i=(T_{u^i}, f_{u^i}) $ for $i<2$ as follows: $T_{u^i}=T$, $\dom (f_{u^i})=\bigcup_{\beta< \alpha} 
    \dom(g^i_{\beta})$ and, for all $\delta \in \dom (f_{u^i})$, $f_{u^i}(\delta) = g^i_\beta(\delta)$,  where $\beta\ge\beta^*$ is such that $\delta\in\dom(g^i_\beta)$. 

This finishes the construction of appropriate lower bounds $(\bar{q},u^i)$ of $\seq{(q'_\beta,r^i_{\beta})\le (q,r^i) }{\beta<\alpha}$ for $i<2$. Now let $\gamma$ be the supremum of $\set{\gamma_\beta}{\beta<\alpha}$. By Claim \ref{C:diff}, there are $p_\alpha\le p$, $q'_\alpha\le \bar{q}$, $r^0_\alpha\le u^0$, $r^1_\alpha\le u^1$ and $\gamma_\alpha>\gamma$ such that $(p_0,q'_\alpha, r^0_\alpha)$, $(p_0,q'_\alpha,r^1_\alpha)$ decide $\dot{b}(\gamma_\alpha)$ differently and $q'_\alpha$ extends both $r^0_\alpha\rest\theta$ and $r^1_\alpha\rest\theta$ and moreover $T_{q'_\alpha}=T_{r^0_\alpha}=T_{r^1_\alpha}$.

If there is no $p$ below $p^*$ which is incompatible with all conditions in $\set{p_\beta}{\beta<\alpha}$ we stop the construction, set $A=\set{p_\beta}{\beta<\alpha}$, and  let $\gamma$ be the supremum of $\seq{\gamma_\beta}{\beta<\alpha}$. If $\alpha=\beta+1$ for some $\beta<\omega$, we set $q'=q'_\beta$
 and $\bar{r}^i=r^i_\beta$ for $i<2$. If $\alpha$ is limit, we construct $q'$ and $\bar{r}^i$, for $i<2$, as $\bar{q}$ and $u^i$, for $i<2$, in the limit case of the induction. It is readiliy verified that the objects thus constructed satisfy the conclusion of the claim.
\end{proof}

W are now ready to construct our labeled system. The construction is by induction on $\omega$.

Let $\gamma_0$ be an arbitrary ordinal below $\omega_1$ and let $(q_0,r_\emptyset)$ be $(q^*,r^*)$. By Lemma \ref{L:into}, $q^*\le r^*\rest\theta$ since $q^*$ forces $r^*$ into $\K_{\kappa}/\dot{H}_\theta$ and $\K_\theta$ is separative. If $q^*\neq r^*\rest\theta$, we can extend $r^*$ appropriately to ensure condition (a); for more details, see Claim \ref{cl:conditions} below in the successor step of the construction.

Now fix $n < \omega$ and assume that we have constructed $\gamma_n$, $q_n$, and $r_s$ for all $s \in {^n2}$. Let $\seq{s_i}{i<2^n}$ enumerate ${^n2}$. We describe how to construct 
$\gamma_{n+1}$, $q_{n+1}$, $A_{s\rest n}$, and $r_s$ for $s \in {^{n+1}}2$.

We proceed by induction on $2^n=m$. Let us start with $s_0$. By Claim \ref{C:antichains}, there are $q^0\le q_n$, $r'_{s_0{}^\smallfrown 0}$,$r'_{s_0{}^\smallfrown 1}\le r_{s_0}$,  a maximal antichain $A_{s_0}$ in $\P_\theta$, and $\gamma^0\ge\gamma_n$ such that for all $p\in A_{s_0}$, $(p,q^0,r'_{s_0{}^\smallfrown 0})$ and $(p,q^0,r'_{s_0{}^\smallfrown 1})$ decide $\dot{b}$ up to $\gamma^0$ differently and $q^0$ extends both $r'_{s_0{}^\smallfrown 0}\rest\theta$ and $r'_{s_0{}^\smallfrown 1}\rest\theta$. 

Now fix $1 \leq i < m$, and suppose that $\gamma^{i-1}$ and $q^{i-1}$ have been constructed. By Claim \ref{C:antichains}, there are $q^{i}\le q^{i-1}$, $r'_{s_{i}{}^\smallfrown 0}$, $r'_{s_{i}{}^\smallfrown 1}\le r'_{s_{i}}$, a maximal antichain $A_{s_i}$  in $\P_\theta$, and $\gamma^{i}\ge\gamma^{i-1}$ such that for all $p\in A_{s_i}$, $(p, q^{i},r'_{s_{i}{}^\smallfrown 0})$, $(p,q^{i},r'_{s_{i}{}^\smallfrown 1})$ decide $\dot{b}$ up to $\gamma^{i}$ differently and $q^{i}$ extends both $r'_{s_i{}^\smallfrown 0}\rest\theta$ and $r'_{s_{i}{}^\smallfrown 1}\rest\theta$. 

Let $q_{n+1}$ be $q^{m-1}$ and $\gamma_{n+1}$ be $\gamma^{m-1}$. It follows by the construction that the objects  $q_{n+1}$, $\gamma_{n+1}$, $r'_{s^\smallfrown j}$, for $j<2$ and all $s\in {^n2}$ satisfy the desired conditions (b) and (c).

However, we have only ensured that $q_{n+1}$ extends $r'_{s}\rest\theta$ for $s\in {^{n+1}2}$, but not that they are equal as is required in condition (a). To ensure condition (a), we define appropriate extensions of $r'_{s}$ for all $s\in {^{n+1}2}$.  

\begin{claim}\label{cl:conditions}
The objects constructed above can be extended to satisfy conditions (a)--(c) of a labeled system in Definition \ref{def:system}.
\end{claim}

\begin{proof}
Since $q_{n+1}$ extends $r'_{s}\rest\theta$ for all $s\in {^{n+1}2}$, $T_{q_{n+1}}$ is an end-extension of $T_{r'_{s}}$ for all $s\in {^{n+1}2}$. Let $\eta+1$ be the height of $T_{q_{n+1}}$.

Now, we define extensions of $r'_{s}$ for $s\in {^{n+1}2}$. For $s\in {^{n+1}2}$, let $f_s$ be a function such that $\dom(f_s)=\dom(f_{r'_{s}})\cap [\theta,\kappa)$ and for each $\alpha\in\dom(f_s)$, let $f_s(\alpha)\supseteq f_{r'_{s}}(\alpha)$ be some node of $T_{q_{n+1}}$ on level $\eta$. Set $r_s=(T_{q_{n+1}},f_{q_{n+1}}\cup f_s)$ for $s\in {^{n+1}2}$. Clearly, $r_s$ are conditions in $\K_\kappa$ for all $s\in {^{n+1}2}$ such that $q_{n+1}$ is equal to $r_s\rest\theta$, hence condition (a) holds. Since $r_{s}$ extends $r'_{s}$ for all 
$s\in {^{n+1}2}$, conditions (b) and (c) are still satisfied for $q_{n+1}$, $\gamma_{n+1}$, $A_s$, and $r_{s^\smallfrown j}$, for $j<2$ and all $s\in {^n2}$.
\end{proof}

This completes our construction of a labeled system. Let $\gamma$ be the supremum of $\seq{\gamma_n}{ n<\omega }$. To finish the proof of Theorem \ref{Th:NoKurepa}, we would like to find a lower bound $q$ of the sequence $\seq{q_n}{ n<\omega }$ and a lower bound $r_x$ of sequences $\seq{r_{x\rest n}}{n<\omega}$ for all $x\in {^{\omega}2}$ such that $q$ forces $r_x$ into the quotient $\K_{\kappa}/\dot{H}_\theta$ for every  $x\in {^{\omega}2}$, thus ensuring that every $(q,r_x)$ is a condition in $\K_\theta*(\K_{\kappa}/\dot{H}_\theta)$. If this is the case, then $(p^*,q)$ will force over $V[G_\theta]$ that level $\gamma$ of $\dot{S}$ has size $2^\omega$, as argued below. 

We now construct the required conditions. Begin by letting $T^*=\bigcup_{n<\omega} T_{q_n}$ and $a=\bigcup_{n<\omega}\dom(f_{q_n})$. Let us assume that  $T^*$ is a normal tree with a limit height. If the height of $T^*$ is a successor ordinal, then the construction of the appropriate lower bounds is analogous but simpler and we leave it as an exercise for the reader. Let $\eta<\omega_1$ denote the height of $T^*$. By condition (a), $T^*=\bigcup_{n<\omega} T_{r_{x\rest n}}$ and $a=\bigcup_{n<\omega}\dom(f_{r_{x\rest n}})\cap\theta$ for all $x\in {^{\omega}2}$.

Note that if we want to ensure that a lower bound of $\seq{q_n}{ n<\omega }$ forces a lower bound of $\seq{r_{x\rest n}}{n<\omega}$, for some $x\in{^\omega 2}$, into the quotient, we are obliged not only to extend all cofinal branches through $T^*$ which are given by functions $\set{f_{q_n}}{n < \omega}$ (recall the proof that $\K_\kappa$ is $\omega_1$-closed), but also to extend all cofinal branches given by $\set{f_{r_{x\rest n} }}{n < \omega}$; otherwise it can happen that the lower bounds will be incompatible. Therefore if we want to ensure that a lower bound of $\seq{q_n}{ n<\omega }$ forces lower bounds of $\seq{r_{x\rest n}}{n<\omega}$ for $x\in {^\omega 2}$ into the quotient, we are obliged to extend uncountably many cofinal branches of $T^*$. This can be done because the trees in the conditions can be wide.

To build suitable lower bounds, we will proceed similarly as in the proof of Lemma \ref{L:K-closed}.
First apply the construction in Lemma \ref{L:K-closed} to find $q'$ which is a lower bound of $\seq{{q_n}}{n < \omega}$ such that $T_{q'}$ has height $\eta+1$ and $\dom(f_{q'})=a$.

For each $x\in {^\omega2}$, let $a_x=\bigcup_{n<\omega} \dom(f_{r_x\rest n})\setminus\theta$. In order to ensure that $q$ is compatible with a lower bound of  $\seq{r_{x\rest n} }{n < \omega}$, we are obliged to extend cofinal branches given by $\set{r_{x\rest n} }{n < \omega}$. For $\alpha \in a_x$, let 

\begin{equation}
      d^x_\alpha := \bigcup \set{f_{r_{x\rest n} }(\alpha)}{ n < \omega \wedge \alpha \in \dom(f_{r_{x\rest n} })}
\end{equation}
and let $D_x=\set{d^x_\alpha}{\alpha\in a_x}$. Moreover, let us define $f_x$ to be a function with domain $a_x$ such that for each $\alpha\in a_x$, $f_x(\alpha)=d^x_\alpha$.

At the end of the construction, set $T_q = T_{q'} \cup \bigcup_{x\in {}^\omega 2} D_{x}$ and $f_q=f_{q'}$. Set $T_{r_{x}}=T_q$ and $f_{r_{x}}=f_{q'}\cup f_{x}$ for every $x\in {}^\omega 2$.
   
\begin{claim} \label{cl:last} The following hold:
\begin{enumerate}[(i)]
\item $q$ is a condition in $\K_\theta$;
\item $r_{x}$ is a condition in $\K_\kappa$ for all $x\in {}^\omega 2$;
\item $q$ forces $r_{x}$ into the quotient $\K_\kappa/\dot{H}_\theta$, for all $x\in {}^\omega2$.
\end{enumerate}
\end{claim}

\begin{proof}
Note that $T_q$ is a normal tree since $T_{q'}$ is normal and $T_q = T_{q'} \cup \bigcup_{x \in {^{\omega}}2} D_x$, hence $q$ is a condition in $\K_\theta$ and also $r_x$ are conditions in $\K_\kappa$ for all $x\in{}^\omega 2$.

By the definition of $r_{x}$, $r_{x}\rest\theta= q$ and hence $q$ forces $r_{x}$ into the quotient $\K_\kappa/\dot{H}_\theta$ for every $x\in{}^\omega 2$.
\end{proof}

It is now straightforward to check that $(p^*,q)$ forces over $V[G_\theta]$ that level $\gamma$ of $\dot{S}$ has size $2^\omega$:  Let $G_P\times G_K$ be a $\P_\theta\times\K_\theta$-generic over $V[G_\theta]$ which contains $(p^*,q)$. Since $q$ forces $r_x$ into the quotient $\K/\dot{H}_\theta$ for all $x\in{}^{\omega}2$, $r_x$ is in $\K/G_K$. For all $x\in{}^{\omega}2$, we take $r'_x\le r_x$ which decides $\dot{b}(\gamma)$, i.e., $r'_x\Vdash \dot{b}(\gamma)=\tau_x$ for some $\tau_x\in \dot{S}_\gamma$. Now it is enough to show that for all $x\neq y\in{}^{\omega}2$, $\tau_x\neq \tau_y$. However, this follows from the construction of our labeled system. Let $x\neq y\in{}^{\omega}2$ be given, let $s=x\cap y$, and without loss of generality assume that $x$ extends $s^\smallfrown0$ and $y$ extends $s^ \smallfrown 1$. Now, since $p^*$ is in $G_P$ and $A_s$ is a maximal antichain below $p^*$ there is $p\in A_s$ which is in $G_P$. Since $(p,q_{n+1}, r_{s^\smallfrown0})$ and $(p,q_{n+1}, r_{s^\smallfrown 1})$ decides $\dot{b}$ up to $\gamma_{n+1}<\gamma$ differently, there is $\delta<\gamma$ and $\tau_0\neq \tau_1$ in $\dot{S}_\delta$ such that $(p,q_{n+1}, r_{s^\smallfrown 0})\Vdash \dot{b}(\delta)=\tau_0$ and $(p,q_{n+1}, r_{s^\smallfrown 1})\Vdash\dot{b}(\delta)=\tau_1$. Since $q\le q_{n+1}$ and $q$ is in $G_K$, $(p,q_{n+1})$ is in $G_P\times G_K$ and therefore  $ r_{s^\smallfrown 0}\Vdash \dot{b}(\delta)=\tau_0$ and $r_{s^\smallfrown 1}\Vdash\dot{b}(\delta)=\tau_1$ over $V[G_\theta][G_P][G_K]$. Since $r'_x\le r_x\le r_{s^\smallfrown 0}$ and $r'_x\Vdash \dot{b}(\gamma)=\tau_x$, it holds $\tau_0<_S \tau_x$. Similarly we can show that $\tau_1<_S\tau_x$. Since $\tau_0\neq\tau_1$ it follows that $\tau_x\neq \tau_y$. Therefore $(p^*,q)$ forces over $V[G_\theta]$ that level $\gamma$ of $\dot{S}$ has size $2^\omega$, hence $(p^*,q)$ forces that $\dot{S}$ is not an $\omega_1$-tree, which is a contradiction. This concludes the proof of Theorem \ref{Th:NoKurepa}.
\end{proof}

The following theorem will now establish Theorem B.

\begin{theorem}\label{Th:negKH+wKH+A-subtree}
  Suppose that there is a supercompact cardinal $\kappa$. Then there is a forcing extension 
  in which 
  \begin{enumerate}
    \item $2^{\omega}=\omega_2=\kappa$;
    \item $\GMP^{\omega_2}(\omega_2)$ holds;
    \item there are no Kurepa trees, but there is a weak Kurepa tree.
  \end{enumerate}
\end{theorem}

\begin{proof}
For simplicity, assume that $\GCH$ holds in $V$.
The generic extension is obtained by forcing with a product of the Mitchell forcing (as defined in Subsection \ref{mitchell_section}) and the forcing for adding a weak Kurepa tree with $\kappa$-many branches: $\M(\omega,\kappa)\times\K_\kappa$, which we denote by $\M\times\K$ for simplicity. Recall from Subsection \ref{mitchell_section} that the set $A$ we are using to define $\M(\omega,\kappa)$ is the set of inaccessible cardinals below $\kappa$. Let $G\times H$ be $\M\times \K$-generic over $V$. The Mitchell forcing preserves $\omega_1$ and all cardinal greater than or equal to $\kappa$ and forces $2^\omega=\kappa=\omega_2$. Since $\K$ is $\kappa$-Knaster in $V$, $\K$ is still $\kappa$-Knaster in $V[G]$ by Fact \ref{F:ccc-Knaster}, and therefore it preserves all cardinals greater or equal to $\kappa$. It also preserves $\omega_1$ over $V[G]$, since it is $\omega_1$-closed in $V$ and therefore it is $\omega_1$-distributive in $V[G]$ by the standard product analysis of the Mitchell forcing $\M$. This finishes the proof of (1).

There is a weak Kurepa tree $T$ in $V[H]$ with $\kappa$-many cofinal branches by the definition of $\K$. Since $\omega_1$ is preserved by $\K\times\M$, $T$ is still an $\omega_1$-tree with $\kappa=\omega_2$-many cofinal branches in $V[H][G]=V[G][H]$.

We now establish $\GMP^{\omega_2}(\omega_2)$. Note that by Fact \ref{isp_guessing_fact}, $\GMP^{\omega_2}(\omega_2)$ is equivalent to $\ISP^{\omega_2}(\omega_2)$. To show that $\ISP^{\omega_2}(\omega_2)$ holds in $V[G][H]$ we use the lifting argument and analysis of the quotient as in \cite{C:trees}. Fix a 
cardinal $\lambda > \kappa$; by increasing it if necessary, assume that $\lambda^{<\lambda} = \lambda$. We will establish $\ISP^{\omega_2}(\omega_2,\lambda)$ 
in $V[G][H]$. To begin, let
\begin{equation}\label{eq:j} j: V\to M 
\end{equation} be a supercompact elementary embedding with critical point $\kappa$ given by a normal ultrafilter $U$ on $\power_\kappa(\lambda)$; i.e. $M\cong \Ult(V,U)$. 
We lift the elementary embedding $j$ in two steps. Note that $j(\M(\omega,\kappa))=\M(\omega,j(\kappa))$ and recall that there is a projection from $\M(\omega,j(\kappa))$ to $\M(\omega,\kappa)$. Let us denote the quotient $\M(\omega,j(\kappa))/G$ by $Q_\M$ and let $g$ be $Q_\M$-generic over $V[G][H]$. Then we can lift in $V[G][g]$ the embedding to $j:V[G]\to M[G][g]$. Now, $j(\K_\kappa)=\K_{j(\kappa)}$ and there is a projection from $\K_{j(\kappa)}$ to $\K_\kappa$ by Lemma \ref{L:wKT-projection}.  Let us denote the quotient $\K _{j(\kappa)}/H$ by $Q_\K$ and let $h$ be $Q_\K$-generic over $V[G][H][g]$. We can lift the embedding in $V[G][H][g][h]$ further to 
\begin{equation}\label{eq:lift} j:V[G][H]\to
 M[G][H][g][h].\end{equation} 

Let $D=\seq{d_x}{x\in\power_\kappa(\lambda)^{V[G][H]}}$ be a $\kappa$-slender list in $V[G][H]$. We want to show that there is an ineffable branch $b$ through $D$ in $V[G][H]$. Let us consider the image of $D$ under $j$: $$j(\seq{d_x}{x\in\power_\kappa(\lambda)^{V[G][H]}})=\seq{d'_y}{y\in\power_{j(\kappa)}(j(\lambda))^{M[G][H][g][h]}}.$$ The set $j``\lambda$ is a subset of $j(\lambda)$ of size $<\! j(\kappa)$. We define our ineffable branch  $b:\lambda\to 2$ in $V[G][H][g][h]$  as follows:\footnote{We identify subsets of $\lambda$ and $j''\lambda$ with their characteristic functions to make this formally correct.} \begin{equation} \mbox{for } \alpha<\lambda,~ b(\alpha)=d'_{j''\lambda}(j(\alpha)).\end{equation}

Before we prove the following claim, let us state some simple properties of the lifted embedding $j$. Since we assume $\lambda^{<\lambda}=\lambda$, $\vert H(\lambda)\vert =\lambda$, and this still holds in $V[G][H]$, i.e.\ $\vert H(\lambda)^{V[G][H]}\vert =\lambda$. In (\ref{eq:lift}) we lifted an embedding generated by a normal ultrafilter $U$ on $\power_\kappa(\lambda)$ and therefore $j``\lambda$ is an element of $j(C)$ for every club $C$ in $\power_\kappa(\lambda)^{V[G][H]}$; since there is a bijection between $\lambda$ and $H(\lambda)^{V[G][H]}$, there is a correspondence between $\power_\kappa\lambda^{V[G][H]}$ and $\power_\kappa H(\lambda)^{V[G][H]}$, and in particular it holds for every club $C \in V[G][H]$ in $\power_\kappa H(\lambda)^{V[G][H]}$ that \begin{equation}\label{eq:H} j``H(\lambda)^{V[G][H]}\in j(C).\end{equation}

\begin{claim}\label{claim:rest}
For each $x\in\power_\kappa\lambda^{V[G][H]}$, $b\rest x\in M[G][H]$.
\end{claim}

\begin{proof}
By slenderness of $D$, we can fix a club $C$ in $\power_\kappa H(\lambda)^{V[G][H]}$ such that for every $N \in C$, $N \prec H(\lambda)^{V[G][H]}$, and for every $x \in N$ of size $\omega_1$, we have $d_{N \cap \lambda} \cap x \in N$.

Let $x\in\power_\kappa\lambda^{V[G][H]}$ be arbitrary. By the $\kappa$-cc of the whole forcing, it holds that $$x\in H(\lambda)^{V[G][H]}=H(\lambda)^{M[G][H]}$$ and therefore $j``x=j(x)\in j``H(\lambda)^{M[G][H]}$. Let us denote $H(\lambda)^{M[G][H]}$ by $N$. Notice that $j``N$ is an elementary submodel of $j(H(\lambda)^{V[G][H]})$ by a general model-theoretic argument, or by invoking (\ref{eq:H}). Since $j`` H(\lambda)^{M[G][H]}\in j(C)$, we have $d'_{(j``N)\cap j(\lambda)}\rest j``x=d'_{j``\lambda}\rest j``x\in j``N$. It follows that there is a function $f$ in $N \sub M[G][H]$ with domain $x$ such that for every $\alpha \in x$, $$f(\alpha) = d'_{j``\lambda}(j(\alpha)).$$ By the definition of $b$, $f = b \rest x$, and the proof is finished.
\end{proof}

\begin{claim}
$b$ is in $M[G][H]$.
\end{claim}

\begin{proof}
We show that $b$ -- which is approximated in $M[G][H]$ on sets in $\power_\kappa\lambda^{V[G][H]} = \power_\kappa\lambda^{M[G][H]}$ by Claim \ref{claim:rest} -- cannot be added by $Q_\K\times Q_\M$ over $M[G][H]$ and therefore it is already in $M[G][H]$. 

We first argue that $Q_\K$ is $\omega_1$-distributive and $\omega_2$-Knaster over $M[G][H]$. In $M$, $j(\K)\simeq \K*Q_\K$ is $\omega_1$-closed and therefore by the product analysis of $\M$, it is $\omega_1$-distributive in $M[G]$ and therefore $Q_\K$ is $\omega_1$-distributive in $M[G][H]$. Regarding the $\omega_2$-Knaster property, first note that $Q_\K$ is $\kappa$-Knaster in $M[H]$ by Lemma \ref{L:quotient-knaster}. Since $\M$ is $\kappa$-Knaster in $M$ it is still $\kappa$-Knaster in $M[H]$ by Fact \ref{F:ccc-Knaster} and therefore $Q_\K$ is $\kappa$-Knaster in $M[H][G]=M[G][H]$ by Fact \ref{F:ccc-Knaster}. Since $Q_\K$ is $\kappa$-Knaster in $M[G][H]$ it has the $\kappa=\omega_2$-approximation property and cannot add $b$ over $M[G][H]$.

Now, we will show that over $M[G][H][h]$, $Q_\M$ also cannot add $b$. In $M[G][H][h]$, $\kappa=\omega_2=2^\omega$ and $Q_\M$ is forcing equivalent to 
$\Add(\omega,\kappa^\dagger)*\M^\kappa$, where $\kappa^\dagger$ is the first inaccessible above $\kappa$. The branch $b$ cannot be added by $\Add(\omega,\kappa^\dagger)$ since this forcing is $\omega_1$-Knaster. 

In $M[G][H][h][\Add(\omega,\kappa^\dagger)]$, $2^\omega = \kappa^\dagger > (\omega_2)^{V[G][H]}$ and $D$ is a list with width 
at most $((2^{\omega_1})^+)^{V[G][H]}) = (\omega_3)^{V[G][H]}$. By a standard product analysis (see Section \ref{mitchell_section}), $\M^\kappa$ is a projection of a product of the form $\Add(\omega,j(\kappa))\times \Q^*_\kappa$, where $\Q^*_\kappa$ is $\omega_1$-closed. Again the branch $b$ cannot be added by $\Add(\omega,j(\kappa))$ since this forcing is $\omega_1$-Knaster. Since $D$ has width $(\omega_3)^{V[G][H]}<2^\omega$ in $M[G][H][h][\Add(\omega,\kappa^\dagger)]$, we can apply Fact \ref{closed_preservation_fact} to $\Add(\omega,j(\kappa))$ as $P$ and $\Q^*_\kappa$ as $Q$ over the model $M[G][H][h][\Add(\omega,\kappa^\dagger)]$. Therefore, the forcing $\Q^*_\kappa$ cannot add $b$ over the model 
\[
  M[G][H][h][\Add(\omega,\kappa^\dagger)][\Add(\omega,j(\kappa))],
\] 
hence the cofinal branch $b$ cannot be added by $\M^\kappa$ over $M[G][H][h][\Add(\omega,\kappa^\dagger)]$.
\end{proof}

\begin{claim}
The function $b$ is an ineffable branch through $D$.
\end{claim}

\begin{proof}
We need to show that the set $S=\set{x\in\power_\kappa(\lambda)}{b\rest x= d_x}$ is stationary, hence it is enough to show that $j``\lambda$ is in $j(S)=\set{y\in\power_{j(\kappa)}(j(\lambda))}{j(b)\rest y= d'_y}$. However this follows from the definition of $b$: $j(b)\rest j``\lambda=j``b=d'_{j``\lambda}$.
\end{proof}

This finishes the proof of (2). By Theorem \ref{Th:NoKurepa}, there are no Kurepa trees in the generic extension $V[G][H]$ and hence we have (3), thus finishing the proof of Theorem \ref{Th:negKH+wKH+A-subtree}.
\end{proof}

\begin{remark}
  If one starts with large cardinals weaker than supercompact, one can obtain 
  variations on Theorem \ref{Th:negKH+wKH+A-subtree} with item (2) weakened. 
  For example, a straightforward adaptation of the proof of Theorem 
  \ref{Th:negKH+wKH+A-subtree} shows that, if $\kappa$ is a weakly compact cardinal 
  and $G \times H$ is $\M(\omega,\kappa) \times \K_\kappa$-generic over $V$, 
  then $V[G][H]$ satisfies $\TP(\omega_2) + \wKH + \neg \KH$.
\end{remark}

\section{Failure of Kurepa hypothesis and weak Kurepa hypothesis}
\label{sect: kurepa_sect}

In this section, we survey what is known about the (in)destructibility of $\neg \KH$ and $\neg \wKH$ and provide a direct proof that $\neg \wKH$ is always preserved by $\sigma$-centered forcing. To begin, observe that, since every Kurepa tree is a weak Kurepa tree, $\neg\wKH$ implies $\neg\KH$. In contrast to $\neg\wKH$, which implies the failure of $\CH$, $\neg\KH$ is consistent with $\CH$. The consistency of both principles follows from the consistency of an inaccessible cardinal. The classical model of $\neg\KH$ is the extension by the Levy collapse $\Coll(\omega_1, <\kappa)$, where $\kappa$ is inaccessible in the ground model; the classical model of $\neg\wKH$ is the extension by $\M(\omega, \kappa)$, where $\kappa$ is inaccessible in the ground model.

Assume now that $\kappa$ is an inaccessible cardinal. An argument by Todorcevic \cite{T:wKH} shows that in the extension by $\M(\omega, \kappa)$, $\neg\wKH$ is indestructible under all ccc forcings of cardinality $\omega_1$. An analogous result holds for $\neg\KH$ in the extension by the Levy collapse $\Coll(\omega_1, <\kappa)$.

However, by a result of Jensen, $\square_{\omega_1}$ implies that there is a ccc forcing which adds a Kurepa tree (cf.\ \cite{JS:KH}). Therefore, if $\kappa$ is an inaccessible cardinal which is not Mahlo, then there is a ccc forcing which adds a Kurepa tree in the generic extensions by both $\Coll(\omega_1, <\kappa)$ and $\M(\omega, \kappa)$. By the previous paragraph, this ccc forcing must have size greater than $\omega_1$.

With respect to $\neg\KH$, the assumption of $\square_{\omega_1}$ is necessary by a result of Jensen and Schlechta \cite{JS:KH}. They show that if $\kappa$ is Mahlo, then $\neg\KH$ is indestructible under all ccc forcings in the extension by $\Coll(\omega_1, <\kappa)$. The analogous result remains open for Mitchell forcing $\M(\omega, \kappa)$ and $\neg\wKH$, where $\kappa$ is a Mahlo cardinal in the ground model.

Returning to ccc forcing of size $\omega_1$: in \cite{ks}, Krueger and the second author show that, assuming the consistency of only an inaccessible cardinal, it is consistent with $\CH$ that the Kurepa Hypothesis fails and yet there exists an almost Kurepa Suslin tree. In particular, there is a ccc forcing of size $\omega_1$ which adds a Kurepa tree. As Corollary \ref{cor: cor_11} shows, assuming the consistency of an inaccessible cardinal, it is consistent that $\neg\wKH$ holds while an almost Kurepa Suslin tree exists.

In particular, both $\neg\KH$ and $\neg\wKH$ are indestructible under all ccc forcings of size $\omega_1$ in their classical models, but for both, there are models in which they are destructible by ccc forcing of size $\omega_1$.

So far, we have considered only ccc forcings. What do we know about forcings satisfying a stronger property? In \cite{HLHS}, Honzik and authors showed that $\GMP$ implies that $\neg \wKH$ is preserved by any $\sigma$-centered forcing. That is, if $V$ is a transitive model satisfying $\GMP$, $\mathbb{P} \in V$ is $\sigma$-centered, and $G$ is $\mathbb{P}$-generic over $V$, then $V[G]$ satisfies $\neg\wKH$. In particular, $\neg \wKH$ is preserved over models of $\GMP$ by adding any number of Cohen reals. Since the argument requires only a guessing model property for small sets, which is equivalent to $\neg \wKH$, it follows that $\neg \wKH$ is always preserved by $\sigma$-centered forcings. 

Below, we provide a direct proof that $\neg \wKH$ is preserved by any $\sigma$-centered forcing without referencing the guessing model property. First, let us recall the definition of $\sigma$-centered forcing.

\begin{definition}\label{def:centered} 
Let $\P$ be a forcing. We say that $\P$ is \emph{$\sigma$-centered} if $\P$ can be written as the union of a family $\set{\P_n\sub \P}{n<\omega}$ such that for every $n<\omega$: \begin{equation}\label{eq:c1} \mbox{for every $p,q \in \P_n$ there exists $r \in \P_n$ with $r \le p,q$}.\end{equation}
\end{definition}

It follows that $\P$ can be written as a union of $\omega$-many filters if we close each $\P_n$ upwards.

Some definitions of $\sigma$-centeredness require just the compatibility of the conditions, with a witness not necessarily in $\P_n$. The condition (\ref{eq:c1}) in this case reads:
\begin{multline}
\label{eq:c2} \mbox{for every $k<\omega$ and every sequence $p_0, p_1, \ldots, p_{k-1}$}\\\mbox{of conditions in $\P_n$ there exists $r \in \P$ with $r \le p_i$ for every $0 \le i < k$}.
\end{multline}

The conditions (\ref{eq:c1}) and (\ref{eq:c2}) are not in general equivalent (see Kunen \cite{Kunen:new}, before Exercise III.3.27), but the distinction is not so important for us because the common forcings such as the Cohen forcing and the Prikry forcings are all centered in the stronger sense of (\ref{eq:c1}). Also note that the conditions are equivalent for Boolean algebras: the definition (\ref{eq:c2}) means that each $\P_n$ is a system with FIP (finite intersection property), and as such can be extended into a filter.

\begin{theorem} \label{thm: sigma_centered}
The failure of the weak Kurepa Hypothesis is preserved by any $\sigma$-centered forcing. That is, if $V$ is a transitive model satisfying $\neg\wKH$, $\P \in V$ is $\sigma$-centered, and $G$ is $\P$-generic over $V$, then $V[G]$ satisfies $\neg\wKH$.
\end{theorem}

\begin{proof}
Fix a forcing notion $\P = \bigcup_{n<\omega} \P_n$, where each $\P_n$ is as in the definition of a $\sigma$-centered forcing; i.e., for every $n<\omega$ and every $p,q \in \P_n$, there exists $r \in \P_n$ such that $r \le p,q$. Let $G$ be $\P$-generic over $V$. Suppose that $\dot{T}$ be a $\P$-name for a weak Kurepa tree $T$, and assume for simplicity that $1_{\P}$ forces this. We also assume that the underlying set of $T$ is $\omega_1$.

We will define a weak Kurepa tree $S$ in the ground model, which will complete the proof of the theorem. The underlying set of $S$ will be $\omega \times \omega_1$; we define the ordering $<_S$ as follows: for $(n, \alpha), (m, \beta) \in S$, let $(n, \alpha) <_S (m, \beta)$ if and only if $n = m$ and there is some $p \in \P_n$ which forces $\alpha <_{\dot{T}} \beta$. Note that $<_S$ is a strict partial order: antisymmetry follows because any two conditions in $\P_n$ are compatible, and transitivity uses the fact that the compatibility of conditions in $\P_n$ is witnessed by a condition also in $\P_n$.

Now, we show that $S$ is indeed a weak Kurepa tree. Clearly, $S$ has size $\omega_1$. To see that $S$ has height $\omega_1$, first we prove that it cannot have height greater than $\omega_1$. If so, then there is $(n,\gamma)$ in $S_{\omega_1}$ for some $(n,\gamma)\in S$; this means that $\pred_S((n,\gamma))=\set{(n,\alpha)<_{S}(n,\gamma)}{(n,\alpha)\in S}$ has order type $\omega_1$. Let $I=\set{\alpha<\omega_1}{(n,\alpha)\in \pred_S((n,\gamma))}$. Then from the definition of $<_S$, it follows that for each $\alpha\in I$, there is $p_\alpha\in \P_n$ which forces that $\alpha<_{\dot{T}}\gamma$. Since $\P$ is ccc, there is $\alpha\in I$ such that $p_\alpha\Vdash\set{\beta\in I}{p_\beta\in \dot{G}}$ has size $\omega_1$, where $\dot{G}$ is a name for the generic filter over $\P$. Then we can see that $p_\alpha$ forces that $\gamma$ has $\omega_1$-many predecessors and hence has height at least $\omega_1$ in $T$, which is a contradiction since $\dot{T}$ is forced by $1_\P$ to be a tree of height $\omega_1$.

To see that $S$ does not have countable height, let $\dot{b}$ be a name for a cofinal branch in $T$. Then, since $\dot{b}$ is forced to be a cofinal branch, in particular uncountable, there is $n$ such that the set $I=\set{\alpha<\omega_1}{\exists p\in\P_n(p\Vdash \alpha\in\dot{b})}$ is uncountable. It is easy to see that for any $\alpha, \beta \in I$, $(n, \alpha)$ and $(n, \beta)$ are comparable in $S$; hence, $d=\set{(n, \alpha)}{\alpha \in I}$ is an uncountable chain through $S$. Therefore, $S$ cannot have countable height, which, combined with the previous paragraph, gives that $S$ has height exactly $\omega_1$. 

Moreover, we can show that $d$ as in the previous paragraph is a cofinal branch; i.e., we prove that $d$ is a maximal chain. Let $(n,\gamma)$ be a node of $S$ such that it is comparable with all $(n,\alpha)$ in $d$. Then since $S$ has height $\omega_1$, $(n,\gamma)$ cannot be above all nodes in $d$; therefore there is $(n,\alpha)\in d$ such that $(n,\gamma)<_S(n,\alpha)$. By the definition of $<_S$, there is $p\in \P_n$ which forces $\gamma<_{\dot{T}}\alpha$ and by the definition of $d$, there is $q\in \P_n$ which forces that $\alpha\in \dot{b}$. Since both $p$ and $q$ are in $\P_n$, there is $r\in \P_n$ below both $p$ and $q$, hence $r$ forces that $\gamma\in\dot{b}$. Since $r\in \P_n$, $(n,\gamma)\in d$. 

A similar argument shows that $S$ has at least $\omega_2$ many cofinal branches. Let $\langle \dot{b}_\alpha \mid \alpha < \omega_2 \rangle$ be a sequence of $\P$-names for cofinal branches through $\dot{T}$ such that, for all $\alpha < \beta < \omega_2$, 
$1_{\P} \Vdash \dot{b}_\alpha \neq \dot{b}_\beta$. For each $\alpha < \omega_2$, pick an uncountable subset $I_\alpha \subseteq \omega_1$ and $n_\alpha < \omega$ as in the previous paragraphs. Then $d_\alpha = \set{(n_\alpha, \gamma)}{\gamma \in I_\alpha}$ is a cofinal branch through $S$. Since there are $\omega_2$-many such branches, there exists $J \subseteq \omega_2$ of size $\omega_2$ and some $n < \omega$ such that $n_\alpha = n$ for all $\alpha \in J$.

To finish the proof, we show that the cofinal branches $\set{d_\alpha}{\alpha \in J}$ are pairwise distinct. Let $\alpha \neq \beta \in J$ and assume for a contradiction that $d_\alpha=d_\beta$. This in particular means that each $\gamma\in I_\alpha$ is in $I_\beta$; therefore, for each $\gamma\in I_\alpha$, there is $p_\gamma\in \P_n$ which forces that $\gamma\in \dot{b}_\beta$. Again, since $\P$ is ccc there is $\gamma\in I_\alpha$ such that $p_\gamma\Vdash\set{\delta\in I_\alpha}{p_\delta\in \dot{G}}$ has size $\omega_1$, where $\dot{G}$ is a name for a generic filter over $\P$. If $p_\gamma\in G$, then in $V[G]$, it holds that $b_\alpha=b_\beta$. This is a contradiction with the assumption that $1_\P\Vdash \dot{b}_\alpha\neq\dot{b}_\beta$.
\end{proof}

Let us note that the status of the corresponding question regarding $\neg \KH$ remains open, even for the forcing adding a single Cohen real. See the open questions in the next section.

\section{Open questions} \label{sect: questions}

We show in Theorem A that $\GMP$ is consistently destructible by a ccc forcing, and even by one of size 
$\omega_1$. The corresponding question for $\TP(\omega_2)$, asked already in 
\cite{UNGER:1} and \cite{hs_indestructibility}, remains open. We restate it here.

\begin{question} \label{question: tp}
  Is it consistent that $\TP(\omega_2)$ holds and is indestructible under ccc forcing? 
  In the opposite direction, is it consistent that $\TP(\omega_2)$ holds and there is a ccc forcing 
  that adds an $\omega_2$-Aronszajn tree?
\end{question}

A natural first place to look in an attempt to answer Question \ref{question: tp} is our 
model $V[\M]$ for Theorem A. Note that, in that model $\GMP$ fails in the extension after forcing with an almost Kurepa Suslin tree $T$, simply because $T$ becomes a Kurepa tree and $\GMP$ implies there are no Kurepa trees. This is not the case for $\TP(\omega_2)$ or $\GMP^{\omega_2}(\omega_2)$. By a result of Cummings \cite{C:trees}, it is consistent from a weakly compact cardinal that $\TP$ and $\KH$ hold together, and by a result of the authors \cite{kurepa_paper}, it is consistent from a supercompact cardinal that $\GMP^{\omega_2}(\omega_2)$ and $\KH$ hold together.

\begin{question}
Suppose that $T$ is a free Suslin tree and $\kappa$ is a supercompact cardinal. Let $V[\M]$ be our model for Theorem A from Section \ref{sect: thma_sect}. Does $\TP(\omega_2)$, or even $\GMP^{\omega_2}(\omega_2)$, hold in $V[\bb{M}][T]$?
\end{question}

As mentioned above, it was proven in \cite{HLHS} that $\GMP$ is always preserved by adding any 
number of Cohen reals, providing one direction in which our Theorem A is sharp. By Theorem \ref{thm: sigma_centered}, the same is true for 
$\neg \wKH$, a consequence of $\GMP$. However, the status of the corresponding question regarding 
$\neg \KH$ and $\TP(\omega_2)$, two other consequences of $\GMP$, remains open, even for the forcing 
to add a single Cohen real.

\begin{question}
  Is it consistent that $\neg \KH$ holds but is destructible under adding a single Cohen real? 
  Is it consistent that $\TP(\omega_2)$ holds but is destructible under adding a single 
  Cohen real?
\end{question}

\appendix
\section{} \label{appendix_a}

In this appendix, we fulfill a promise made in Remark \ref{remark: term_forcing} 
by providing an example of a two-step iteration of the form $\Add(\omega,1) \ast 
\dot{\Q}$ such that $\dot{\Q}$ is forced to be totally proper, and hence 
$\omega_1$-distributive, and yet forcing with the associated term forcing 
$\Q$ over $V$ collapses $\omega_1$. 

Let $T$ be a normal coherent Aronszajn tree consisting of finite-to-one functions into 
$\omega$. More precisely:
\begin{itemize}
  \item $T \subseteq {^{<\omega_1}}\omega$ is $\subseteq$-downward closed, and 
  $(T, \subseteq)$ is a normal Aronszajn tree;
  \item every element of $T$ is a finite-to-one function;
  \item for all $\alpha < \omega_1$ and all $s,t \in T_\alpha$, we have 
  $s =^* t$, i.e., the set $\{\eta < \alpha \mid s(\alpha) \neq t(\alpha)\}$ 
  is finite.
\end{itemize}
It is straightforward to construct such Aronszajn trees in $\ZFC$.
Given a real $r \in {^{\omega}}\omega$, let $T^r$ be the subtree of 
${^{<\omega_1}}\omega$ defined by setting $T^r = \{r \circ t \mid t \in T\}$; 
note that this definition continues to make sense for reals $r$ existing only 
in outer models of $V$. The following is well-known (see, e.g., the end of \cite{todorcevic_ctble_ordinals}).

\begin{theorem} \label{thm: cohen_suslin}
  Suppose that $r \in {^{\omega}}\omega$ is Cohen-generic over $V$. Then, in 
  $V[r]$, $T^r$ is a Suslin tree.
\end{theorem}

Let $\P = \Add(\omega,1)$ be the forcing to add a single Cohen real. In this section, we think of 
the underlying set of $\P$ as ${^{<\omega}}\omega$, ordered by reverse inclusion. 
Let $\dot{G}$ be the canonical $\P$-name for the generic filter, and let 
$\dot{r}$ be a $\P$-name for $\bigcup \dot{G}$. By Theorem 
\ref{thm: cohen_suslin}, we have $\Vdash_{\P} ``T^{\dot{r}} \text{ is Suslin}"$.
In particular, considering $T^{\dot{r}}$ as a $\P$-name for a forcing notion 
(recall that the forcing order for a tree is the opposite of the tree order), 
we have $\Vdash_{\P} ``T^{\dot{r}} \text{ is ccc and totally proper}"$. In 
particular, the two-step iteration $\P \ast T^{\dot{r}}$ is ccc. 
Let $\Q$ denote the term forcing associated with this two-step iteration. 
In particular, conditions in $\Q$ are all $\P$-names forced by $1_{\P}$
to be elements of $T^{\dot{r}}$, and, given $\dot{s},\dot{t} \in \Q$, we set 
$\dot{t} \leq_{\Q} \dot{s}$ iff $1_{\P}$ forces that $\dot{t}$ extends 
$\dot{s}$ as a forcing condition, i.e., iff $1_{\P} \Vdash_{\P} `` 
\dot{s} \leq_{T^{\dot{r}}} \dot{t}"$.

\begin{proposition}
  Forcing with $\Q$ over $V$ collapses $\omega_1$.
\end{proposition}

\begin{proof}
  Let $H$ be $\Q$-generic over $V$, and suppose for the sake of contradiction 
  that $(\omega_1)^V$ remains uncountable in $V[H]$. In particular, since 
  $T$ is a coherent Aronszajn tree consisting of finite-to-one functions, 
  $T$ remains Aronszajn in $V[H]$. In $V$,
  for each $\alpha < \omega_1$, let $D_\alpha$ be the set of $\dot{t} \in \Q$ 
  such that $\Vdash_{\P}``\height_{T^{\dot{r}}}(\dot{t}) \geq \alpha"$. By a 
  straightforward argument using the normality of $T$ and the maximum principle, 
  each $D_\alpha$ is a dense, open subset of $\Q$; therefore, in $V[H]$, for 
  each $\alpha < \omega_1$, we can fix a $\dot{t}_\alpha \in H \cap D_\alpha$. 
  For each $\alpha < \omega_1$, fix a $p_\alpha \in \P$, a $\beta_\alpha \in 
  [\alpha,\omega_1)$, and an $s_\alpha \in T_{\beta_\alpha}$ such that, in $V$, 
  we have $p_\alpha \Vdash_{\P} ``\dot{t}_\alpha = \dot{r} \circ s_\alpha"$. 
  Since $\omega_1$ is preserved in $V[H]$, we can fix in $V[H]$ a single 
  $p \in \P$ and a stationary $A \subseteq \omega_1$ such that 
  \begin{itemize}
    \item for all $\alpha \in A$, we have $p_\alpha = p$;
    \item for all $\alpha < \alpha'$, both in $A$, we have $\beta_\alpha < 
    \alpha'$.
  \end{itemize}
  In particular, it follows that, for all $\alpha < \alpha'$, both in $A$, 
  we have, in $V$, $p \Vdash_{\P} ``\dot{r} \circ s_\alpha \subseteq 
  \dot{r} \circ s_{\alpha'}"$. Let $k = \dom(p)$, and, for each $\alpha \in A$, 
  let $B_\alpha = \{\eta < \beta_\alpha \mid s_\alpha(\eta) < k\}$; note that 
  $B_\alpha \in [\beta_\alpha]^{<\omega}$.
  
  \begin{claim}
    For all $\alpha < \alpha'$, both in $A$, we have $B_\alpha \subseteq 
    B_{\alpha'}$.
  \end{claim}
  
  \begin{proof}
    Suppose for the sake of contradiction that $\alpha < \alpha'$ are both in 
    $A$ and $\eta \in B_\alpha \setminus B_{\alpha'}$. Let 
    $j' = s_{\alpha'}(\eta)$ and $j = s_{\alpha}(\eta)$, noting that 
    $j < k \leq j'$. Find $p' \leq_{\P} p$ such that $p'(j') \neq p(j)$. 
    Then $p' \Vdash_{\P} ``(\dot{r} \circ s_\alpha)(\eta) \neq 
    (\dot{r} \circ s_{\alpha'})(\eta)"$, contradicting the fact that 
    $p \Vdash_{\P} ``\dot{r} \circ s_\alpha \subseteq 
    \dot{r} \circ s_{\alpha'}"$.
  \end{proof}
  Since $B_\alpha$ is finite for each $\alpha \in A$ and $\langle B_\alpha \mid 
  \alpha \in A \rangle$ is $\subseteq$-increasing, by removing an 
  initial segment of $A$ if necessary, we can assume that there is 
  $B \in [\omega_1]^{<\omega}$ such that $B_\alpha = B$ for all 
  $\alpha \in A$. Now find an uncountable $A^* \subseteq A$ and a fixed 
  function $f:B \rightarrow \omega$ such that $s_\alpha \restriction B = 
  f$ for all $\alpha \in A^*$.
  
  \begin{claim}
    For all $\alpha < \alpha'$, both in $A^*$, we have $s_\alpha = s_{\alpha'} 
    \restriction \beta_\alpha$.
  \end{claim}
  
  \begin{proof}
    Fix $\alpha < \alpha'$ in $A^*$ and $\eta < \beta_\alpha$. If 
    $\eta \in B$, then $s_\alpha(\eta) = f(\eta) = s_{\alpha'}(\eta)$. 
    If $\eta \notin B$, then we have $s_\alpha(\eta), s_{\alpha'}(\eta) 
    \geq k$. Suppose for the sake of contradiction that 
    $s_\alpha(\eta) = j \neq j' = s_{\alpha'}(\eta)$. Find $p' \leq_{\P} 
    p$ such that $p'(j) \neq p'(j')$. Then $p' \Vdash_{\P} ``
    (\dot{r} \circ s_\alpha)(\eta) \neq (\dot{r} \circ s_{\alpha'})(\eta)"$, 
    contradicting the fact that 
    $p \Vdash_{\P} ``\dot{r} \circ s_\alpha \subseteq 
    \dot{r} \circ s_{\alpha'}"$.
  \end{proof}
  
  Now the claim implies that the downward closure of $\{s_\alpha \mid \alpha \in 
  A^*\}$ is a cofinal branch through $T$ in $V[H]$, contradicting the fact that 
  $T$ remains an Aronszajn tree in $V[H]$.
\end{proof}
\bibliographystyle{plain}
\bibliography{bib}

\end{document}